\documentclass[11pt]{article}

\usepackage[margin=1in]{geometry}
\usepackage[utf8]{inputenc}
\usepackage[T1]{fontenc}
\usepackage[english]{babel}
\usepackage{indentfirst}

\usepackage{amsmath} 
\usepackage{amssymb}
\usepackage{amsthm}
\usepackage{xcolor}
\usepackage{subcaption}
\usepackage{enumitem}
\usepackage{multirow}
\usepackage{menukeys}

\usepackage{tikz}
\graphicspath{{figures/}}

\usepackage{todonotes}

\usepackage{algorithm}
\usepackage{algorithmic}

\usepackage[normalem]{ulem}
\usepackage{cancel} 

\usepackage{natbib}
\usepackage{cleveref}

\theoremstyle{plain}
\newtheorem{theorem}{Theorem}
\newtheorem{lemma}[theorem]{Lemma}
\newtheorem{proposition}[theorem]{Proposition}
\newtheorem{corollary}[theorem]{Corollary}

\theoremstyle{definition}
\newtheorem{definition}[theorem]{Definition}
\newtheorem{example}[theorem]{Example}
\newtheorem{assumption}{Assumption}

\theoremstyle{remark}
\newtheorem{remark}[theorem]{Remark}

\renewcommand{\epsilon}{\varepsilon}

\DeclareMathOperator{\diam}{diam}

\DeclareMathOperator{\vol}{vol}

\DeclareMathOperator*{\TOL}{TOL}

\newcommand{\M}{\mathcal{M}}
\newcommand{\R}{\mathbb{R}}
\newcommand{\E}{\mathbb{E}}

\newcommand{\GME}{C_{\operatorname{GME}}}
\newcommand{\widehatGME}{\widehat{C}_{\operatorname{GME}}}

\renewcommand{\cite}{\citep}

\title{Geometry-Preserving Encoder/Decoder\\ In Latent Generative Models}

\author{Wonjun Lee\thanks{Department of Mathematics, The Ohio State University. \texttt{lee.8222@osu.edu}}\and
Riley C. W. O'Neill\thanks{Department of Mathematics, University of Minnesota. \texttt{oneil571@umn.edu}}\and
Dongmian Zou\thanks{Zu Chongzhi Center for Mathematics and Computational Sciences, Duke Kunshan University. \texttt{dongmian.zou@duke.edu}}\and
Jeff Calder\thanks{Department of Mathematics, University of Minnesota. \texttt{jwcalder@umn.edu}}\and
Gilad Lerman\thanks{Department of Mathematics, University of Minnesota. \texttt{lerman@umn.edu}}}

\begin{document}

\maketitle

\begin{abstract}

    Generative modeling aims to generate new data samples that resemble a given dataset. When using diffusion models for this task, one of the main challenges is solving the problem in the input space, which tends to be very high-dimensional. To address this, recent approaches solve diffusion models in the latent space through an encoder that maps from the data space to a lower-dimensional latent space, improving training efficiency and achieving state-of-the-art results. The variational autoencoder (VAE) is the most commonly used encoder/decoder framework in this domain, known for its ability to learn latent representations and generate data samples. In this paper, we introduce a novel encoder/decoder framework with theoretical properties distinct from those of the VAE, specifically designed to preserve the geometric structure of the data distribution. We demonstrate the significant advantages of this geometry-preserving encoder in the training process of both the encoder and decoder. Additionally, we provide theoretical results proving convergence of the training process, including convergence guarantees for encoder training, and results showing faster convergence of decoder training when using the geometry-preserving encoder.
\end{abstract}

\section{Introduction}\label{sec:intro}
Generating complex, high-dimensional continuous data, such as high-fidelity images or physical measurements, is a central challenge in machine learning. Because operating directly in the raw data space (e.g., across millions of pixels) is computationally overwhelming, state-of-the-art frameworks like latent diffusion models \cite{rombach2022high} rely on a two-stage approach. First, an encoder maps the complex data into a lower-dimensional, more manageable representation called a "latent space." Then, a decoder translates points from this compressed space back into the original high-dimensional format. Once this encoder-decoder bridge is built, advanced generative systems can efficiently synthesize new data simply by exploring the small latent space and using the decoder to bring those new points to the data space.

However, the efficiency and quality of this entire process depend entirely on how well the encoder organizes the latent space. Traditionally, the encoder and decoder are trained simultaneously, most commonly through the Variational Autoencoder (VAE) framework \cite{kingma2013auto}. VAEs attempt to balance two competing goals: minimizing the reconstruction error of the decoder, while using a statistical penalty (the Kullback-Leibler divergence) to force the latent data into a pre-defined shape, such as a standard Gaussian distribution.

This joint training together with the statistical penalty can introduce instability and often alters the intrinsic geometry of the data. Forcing the latent space to conform to an arbitrary statistical prior warps the global structural relationships between data points. When the latent space becomes distorted, the decoder requires significantly more optimization capacity to invert the representation, which can lead to training inefficiencies, potential information loss (such as posterior collapse), and degraded sample quality.

To address these limitations, we rethink the role of the encoder. We propose that a well-designed encoder must satisfy two primary criteria: first, it must remain true to the underlying structure of the data and maintain it accurately within the latent space; second, it must provide a representation that enables favorable optimization conditions for the decoder. Rather than forcing a compromise between reconstruction and an arbitrary statistical shape, we explicitly design the encoder to preserve data geometry. By establishing these geometric properties independently of the decoder, we can decouple the training process. The encoder is trained in isolation, and its precomputed representations are subsequently used to train the decoder. This decoupled approach offers explicit structural guarantees for the latent space, which can be more challenging to achieve through standard joint training where both components are optimized simultaneously.

To this end, we introduce the Geometry-Preserving Encoder/Decoder (GPE) framework. Instead of standard joint optimization, our approach trains the encoder using a Gromov-Monge embedding loss~\cite{lee2023monotone} to preserve the global pairwise geometric distances of the original data. In this paper, we demonstrate both theoretically and empirically that preserving this geometric structure directly yields better optimization performance for the decoder. By ensuring that macroscopic data relationships are proportionally mirrored in the latent space, we provide a stable, geometrically aligned foundation. We then train the decoder sequentially based purely on reconstruction. Because the latent space already reflects the true structure of the data, training the decoder and any subsequent generative models operating in that latent space becomes highly efficient. Unlike classical manifold learning techniques (e.g., Isomap~\cite{tenenbaum1997mapping} or MDS~\cite{kruskal1964multidimensional,carroll1998multidimensional,borg2005modern}) that focus purely on representation without consideration for a decoder, our framework is specifically designed to construct a globally coherent latent space that guarantees the optimization efficiency of downstream generative decoding.

Our specific contributions are as follows:

\begin{itemize}
    \item We propose an encoder-decoder architecture where the primary objective of the encoder is to maintain the true data geometry within the latent space. By decoupling the training into sequential steps (encoder first, then decoder), we create an unwarped latent space that simplifies the optimization and training of the decoder (Section~\ref{sec:gpe}).
    
    \item For the encoder, we prove that optimization stability inherently improves as the Gromov-Monge embedding loss decreases, establishing the improved conditioning of the loss landscape. Crucially, we demonstrate that as the encoder adheres more closely to the geometry-preserving condition, effectively functioning as a \emph{local isometry}, the condition number of the decoder's optimization landscape is bounded. This theoretically guarantees both higher stability and an accelerated convergence rate for training the decoder, showing that geometric preservation yields better optimization (Section~\ref{sec:theory-gpe}).
    
    \item We establish explicit upper bounds for the Wasserstein distance between the true data distribution and the final generated distribution. These bounds analytically connect the end-to-end performance of the generative system to the GPE encoder's geometry preservation, the decoder's reconstruction accuracy, and the alignment of the latent spaces (Section~\ref{sec:gpe-lgm}).
    
    \item We demonstrate through comprehensive experiments that the GPE framework converges faster and achieves lower reconstruction errors than standard VAEs. Furthermore, we show that our global geometry-preserving approach provides a superior foundation for generation tasks, yielding high-quality results on complex continuous datasets (Section~\ref{sec:numeric}). The complete implementation is publicly available on GitHub\footnote{https://github.com/wonjunee/GPE\_codes}.
\end{itemize}

\subsection{Related works}

\subsubsection{Encoder-decoder frameworks for generative tasks}
Encoder-decoder architectures are widely used in generative modeling. By projecting high-dimensional data into a lower-dimensional latent space, frameworks such as Variational Autoencoders (VAEs)~\cite{kingma2013auto, van2017neural} and latent-augmented GANs~\cite{larsen2016vaegan, srivastava2017veegan, gao2020zero, makkuva2020optimal} reduce data dimensionality. This approach is frequently applied in Latent Diffusion Models (LDMs)~\cite{rombach2022high, pandey2021vaes, wang2021dydiff, vahdat2021score, gu2022vector, dao2023flow, kim2024pagoda}, where executing the generative process directly in pixel space is computationally expensive. The optimization efficiency and sample quality of the generative decoder depend on the structure of the latent space produced by the encoder. If the latent space configuration is irregular, the decoder requires more optimization to reconstruct or map the data effectively.

\subsubsection{Geometry preservation for generative modeling} 
To structure this latent space, recent works have explored geometry-preserving encoders. In a generative context, this differs from classical manifold learning tasks aimed purely at representation or visualization. Classical manifold unrolling methods often rely on computing intrinsic geodesic distances to capture underlying geometry. However, calculating exact intrinsic geodesics from discrete data in high dimensions is computationally expensive and difficult to scale. Consequently, alternative frameworks focus on recovering low-dimensional geometry or enforcing strictly \emph{local} isometries between the input and latent spaces~\cite{falorsi2019reparameterizing, connor2021variational, chen2021neighborhood, huh2024isometric, kim2024gamma, kato2020rate, gropp2020isometric, lee2022regularized, refId0}. Yet, local methods do not guarantee the global structural cohesion required when sampling from a standard generative prior. Instead, our goal is to construct an encoder that provides an approximate \emph{global} isometry using ambient metrics. This ensures that macroscopic data relationships are proportionally mirrored in the latent space, providing a stable foundation for the subsequent generative decoder. Theoretical frameworks for the existence of such global geometry-preserving maps under various cost functions have been formalized by~\cite{murray2024probabilistic}.

\subsubsection{Multi-Dimensional Scaling and Gromov-Monge alignments}
The objective of preserving global pairwise distances naturally connects to classical Multi-Dimensional Scaling (MDS)~\cite{kruskal1964multidimensional,carroll1998multidimensional, borg2005modern}, which directly optimizes discrete embedding points to match given pairwise distances. In the continuous setting of probability measures, this structural alignment is traditionally formulated via the Gromov-Wasserstein (GW) distance~\cite{memoli2011gromov, peyre2016gromov}, which has seen broad application in aligning metric spaces for generative modeling~\cite{bunne2019learning, titouan2019sliced, nakagawa2023gromovwasserstein}. While conceptually similar, MDS and GW are fundamentally distinct: MDS searches for optimal target points, whereas GW searches for an optimal coupling or map that best preserves pairwise distances between two spaces. However, one can easily connect the two concepts. By considering the distance $d_{GW}(\mu, T_\#\mu)$ and optimizing over the pushforward map $T$, the GW formulation becomes identical to a continuous MDS problem.

The Gromov-Monge Embedding (GME) loss functions as a deterministic, map-based counterpart to these concepts, combining the pairwise distance preservation objective of MDS with the Gromov-Monge mapping formulation~\cite{lee2023monotone, yang2024improving}. While this objective shares a conceptual connection with the Gromov-Wasserstein (GW) distance, our framework does not rely on optimal transport theory. Instead of computing complex probabilistic couplings between spaces, the encoder is trained entirely via the deterministic GME loss, a choice that sidesteps the computational overhead associated with optimal transport. 

Consequently, the cost function within the GME loss is designed to maximize optimization efficiency, which we validate through both theoretical analysis and numerical experiments. By avoiding the complexities of GW theory, we are able to establish standalone bounds on training stability and decoder efficiency, as well as an explicit upper bound on the error of the final generated distribution using direct probability bounds.

\subsection{Structure of the paper}

This paper is organized as follows. We begin in Section~\ref{sec:asmpt}, where we present the notations that will be used throughout the paper. In Section~\ref{sec:gpe}, we introduce the geometry-preserving encoder/decoder (GPE) framework and contrast it with the Variational Autoencoder (VAE) framework. This section highlights the key differences in how each framework handles the embedding of data distributions into latent space, with a focus on the preservation of geometric structure.

Section~\ref{sec:theory-gpe} presents our theoretical analysis of the GPE framework, including the local isometric property of the GPE encoder and the efficiency of computing the GPE encoder and decoder from the proposed formulations. We provide detailed proofs and discuss the implications of these results. Following this, in Section~\ref{sec:gpe-lgm}, we briefly review the latent diffusion model and extend our theoretical analysis to the GPE-based latent diffusion model. Here, we demonstrate how the framework's performance, as measured by the Wasserstein distance between the underlying and generated distributions, improves with enhanced geometry-preserving properties of the encoder and decoder.

In Section~\ref{sec:alg}, we outline the algorithm for computing the GPE encoder and decoder. Section~\ref{sec:numeric} follows with numerical results that validate the theoretical findings from the preceding sections. We present experiments on both artificial datasets and real-world datasets, such as MNIST \cite{deng2012mnist}  and CelebA-HQ ($256 \times 256$) \cite{celebahq}, and compare the performance of the GPE framework with the VAE framework in both reconstruction and generation tasks.

Finally, the paper concludes in Section~\ref{sec:conclusion}, where we summarize the key findings, discuss the broader implications of our results, and suggest potential directions for future research.

\section{Notation}\label{sec:asmpt}
Let $\mathcal{M} \subset \mathbb{R}^D$ be a compact, $m$-dimensional smooth Riemannian submanifold ($m \ll D$). We equip $\mathcal{M}$ with a Riemannian metric $g$, an intrinsic inner product on the tangent space $\mathcal{T}_x\mathcal{M}$ that allows us to measure distances and volumes directly on the curved manifold \cite{john2012introduction}. Associated with this geometry is the exponential map $\exp_x$, which locally maps tangent vectors from $\mathcal{T}_x\mathcal{M}$ back onto the manifold. We assume the data distribution $\mu \in \mathbb{P}(\mathcal{M})$ admits a density $\rho$ with respect to the Riemannian volume measure $\mathrm{vol}_{\mathcal{M}}$ (which is naturally induced by $g$), bounded by positive constants
\begin{equation} \label{eq:asmpt-mu}
    0 < \rho_{\min} \leq \rho(x) \leq \rho_{\max}, \quad \forall x \in \mathcal{M}.
\end{equation}
Let the latent distribution $\nu$ be the $d$-dimensional standard normal distribution ($d \geq m$). We define the encoder $T: \mathcal{M} \to \mathbb{R}^d$ as a smooth embedding (a diffeomorphism onto its image $T(\mathcal{M})$). The pushforward measure $T_\#\mu$ is defined by the identity
\begin{equation*}
    \int_{\mathcal{M}} f(T(x)) \, d\mu(x) = \int_{\mathbb{R}^d} f(y) \, dT_\#\mu(y)
\end{equation*}
for all continuous functions $f: \mathbb{R}^d \to \mathbb{R}$. Because $T$ is a diffeomorphism onto its image, $T_\#\mu$ is naturally absolutely continuous with respect to the volume measure $\mathrm{vol}_{T(\mathcal{M})}$ on the image submanifold $T(\mathcal{M}) \subset \mathbb{R}^d$.

\section{Local and Global Isometry}
The objective of our encoder design is not classical representation learning or dimensionality reduction, but the construction of a latent space tailored for \emph{latent generative models}. Traditional embedding methods focus on the low-dimensional representation as the final product without accounting for how easily a decoder can map those latent points back to the original data manifold. In contrast, our framework aims to find an encoder configuration that improves the training efficiency and optimization of the \emph{decoder}.

Classical manifold learning techniques, such as Isometric Feature Mapping (Isomap) \cite{tenenbaum1997mapping}, Locally Linear Embedding (LLE)~\cite{roweis2000nonlinear}, and Laplacian Eigenmaps~\cite{belkin2003laplacian}, focus on local geometry or rely on calculating global intrinsic geodesic distances. While geodesic-based approaches can capture manifold structures, computing intrinsic distances on high-volume, high-dimensional discrete data is computationally expensive, difficult to scale, and does not account for downstream decoder optimization.  When global alignment must emerge by chaining local neighborhoods, it can lead to topological drift and warped global layouts. This forces the decoder to learn convoluted, non-smooth mappings, which can reduce training efficiency.

To address these limitations, we introduce a scalable approach that targets global geometry preservation using ambient metrics, which consequently enforces local isometry. Because exact global isometry is topologically restricted for complex manifolds, we adopt a relaxed global notion. This formulation provides a mathematical advantage analogous to selecting a viscosity solution for the Hamilton-Jacobi equation. The condition for local isometry is under-determined and admits infinitely many weak solutions, many of which exhibit global distortions. Enforcing a relaxed version of global isometry on top of local properties acts as a selection mechanism, filtering out poorly behaved configurations. While the solution space remains infinite due to standard translation and rotation invariances, this global constraint restricts the encoder to a structured class of geometrically consistent mappings. By ensuring macroscopic distances are proportionally mirrored in the latent space, we preserve local isometry in a computationally scalable manner, providing a smoother target space that improves decoder training efficiency and overall numerical stability.

\paragraph{Geometry-preserving map}

We call an encoder \emph{$\alpha$–geometry–preserving} if it is an $\alpha$–bi-Lipschitz map $T:\mathcal M\to\mathbb R^d$ with $\alpha\ge1$, i.e.,
\begin{equation}\label{eq:theta-bi-lip}
\frac{1}{\alpha}\,\|x-x'\|\ \le\ \|T(x)-T(x')\|\ \le\ \alpha\,\|x-x'\|\qquad \forall\,x,x'\in\mathcal M.
\end{equation}
In practice, guaranteeing \eqref{eq:theta-bi-lip} uniformly for \emph{all} pairs for $\alpha\approx 1$ is challenging and may not even be achievable. We therefore adopt a high-probability relaxation that still provides the necessary global control for generative modeling.

\begin{definition}[Weak $\alpha$–bi-Lipschitz]\label{def:weak-alpha-bi}
Let $T:\mathcal M\to\mathbb R^d$ and $A\subset\mathcal M^2$. We say that $T$ is \emph{weakly $\alpha$–bi-Lipschitz on $A$} if, for all $(x,x')\in A$,
\begin{equation*}
\frac{1}{\alpha^2}\,\|x-x'\|^2-\Bigl(1-\frac{1}{\alpha^2}\Bigr)
\ \le\ \|T(x)-T(x')\|^2
\ \le\ \alpha^2\,\|x-x'\|^2+\bigl(\alpha^2-1\bigr).
\end{equation*}
\end{definition}

The additive slack terms make the constraint meaningful at both small and large separations. The next proposition shows that, on sufficiently separated pairs, the weak condition recovers a standard bi-Lipschitz bound.

\begin{proposition}[Bi-Lipschitz on well-separated pairs \cite{lee2023monotone}]\label{prop:weak-to-strong}
Let $T:\mathcal M\to\mathbb R^d$ be weakly $\alpha$–bi-Lipschitz on $A\subset\mathcal M^2$. Fix $0<\gamma<1$ and define
\begin{equation*}
\tilde A\ :=\ \Bigl\{(x,x')\in A:\ \|x-x'\|^2\ \ge\ \frac{\alpha^2-1}{\gamma}\Bigr\}.
\end{equation*}
Then, for all $(x,x')\in\tilde A$,
\begin{equation*}
\frac{1-\gamma}{\alpha^2}\,\|x-x'\|^2\ \le\ \|T(x)-T(x')\|^2\ \le\ \bigl(\alpha^2+\gamma\bigr)\,\|x-x'\|^2.
\end{equation*}
\end{proposition}

\noindent
In words, the weak bi-Lipschitz condition yields \emph{global} (scale–aware) control: it implies standard bi-Lipschitz bounds on all pairs beyond a data–dependent separation threshold, while permitting mild slack at very small scales.
This is the regime targeted by our GPE training objective, which operates on pairwise terms to enforce high–probability global geometry preservation.

Next, we establish the existence of a weak $\alpha$-bi-Lipschitz map for an $m$-dimensional smooth compact manifold embedded in $\R^D$. The construction uses a finite $\epsilon$-net and a global Johnson-Lindenstrauss (JL) projection. By mapping points to their nearest net anchors, the JL step tightly preserves discrete anchor distances in $\R^d$ (for sufficiently large $d$). The continuous map then inherits these bounds globally, accumulating only a small, controlled additive error.

For clarity, our theorem and proof use linear distances rather than the standard squared distances (which bound the ratio $\frac{1+\|T(x)-T(y)\|^2}{1+\|x-y\|^2}$). Deriving the squared bounds from the linear ones is a straightforward algebraic step with a slight adjustment on the quantity $\alpha$. This deliberate simplification avoids cumbersome cross-terms and directly highlights the core geometric intuition.

\begin{theorem}[Existence of a weak $\alpha$-bi-Lipschitz map]\label{thm:weak-bilip}
Let $\M\subset\R^D$ be a compact $m$ dimensional submanifold. For any target distortion $\alpha > 1$, there exists a target dimension $d = O\big(\alpha^2 m \log(\frac{\alpha+\sqrt{D}}{\alpha-1})\big)$ and a map $T:\M\to\R^d$ such that for all $x,y\in\M$, the following weak $\alpha$-bi-Lipschitz condition holds globally
\begin{equation}\label{eq:weak-alpha}
\alpha^{-1}\|x-y\| - (1 - \alpha^{-1}) \ \leq\ \|T(x) - T(y)\| \ \leq\ \alpha \|x-y\| + \alpha - 1.
\end{equation}
\end{theorem}

From an optimization perspective, finding a weak $\alpha$-bi-Lipschitz map is much easier than enforcing a strict $\alpha$-bi-Lipschitz constraint given the same value $\alpha$. Strict conditions require multiplicative distance preservation at all scales, which imposes rigid structural constraints and severely restricts the solution space. In contrast, the weak formulation uses additive error terms to tolerate microscopic distortions, expanding the space of admissible functions. Because this solution space is so much larger, the optimization landscape becomes smoother. This makes finding a valid map both mathematically and practically advantageous, allowing the algorithm to easily capture the global geometry of $\M$.

\section{Geometry–Preserving Encoder/Decoder (GPE)}\label{sec:gpe}

Next, we discuss how to compute a weak $\alpha$-bi-Lipschitz encoder. Before we introduce our framework, we first review the Gromov-Monge embedding (GME) cost function introduced in \cite{lee2023monotone,yang2024improving}
\[
    \GME(T)
    := \mathbb{E}_{(x,x')\sim \mu^2} \Big(c_X(x,x') - c_Y(T(x),T(x'))\Big)^2 ,
\]
where $c_X:\mathcal{M}\times\mathcal{M}\rightarrow \mathbb{R}$ and $c_Y:\mathbb{R}^d \times \mathbb{R}^d \rightarrow \mathbb{R}$ are cost functions defined in the data space and the latent space, respectively. 
Throughout this paper, we fix 
\[
    c_X(x,x') = \log(1+\|x-x'\|^2) \quad \text{and} \quad c_Y(y,y') = \log(1+\|y-y'\|^2).
\]
Thus, the GME cost function we consider takes the form
\begin{equation}\label{eq:gme-cost}
    \GME(T)
    := \mathbb{E}_{(x,x')\sim \mu^2} \left[ \log\left(\frac{1+\|T(x)-T(x')\|^2}{1+\|x-x'\|^2}\right)^2\right].
\end{equation}
These cost functions are motivated by the slow growth of the logarithm, which enables tight bounds on the Hessian of the cost function and thereby accelerates optimization; we analyze this in detail in \Cref{sec:stability-encoder}. Furthermore, in \cite{lee2023monotone}, it is shown that with this particular choice of cost functions, if $T$ satisfies a small value of the GME cost, then it follows that $T$ satisfies the weak $\alpha$-bi-Lipschitz condition for the majority of pairs of points in $\mathcal{M}^2$. We review this result in the following theorem.

\begin{theorem}[\cite{lee2023monotone}]\label{thm:T-bi-lip-mu}
    Fix $\alpha\geq1$ and $\epsilon_{\rm{GME}} \geq 0$. Given a distribution $\mu \in \mathbb{P}(\mathcal{M})$, suppose $T:\mathcal{M} \rightarrow \mathbb{R}^d$ satisfies $\GME(T) \leq \epsilon_{\rm{GME}}$. Then, $T$ satisfies the weak $\alpha$-bi-Lipschitz condition on a subset $A \subset \mathcal{M}^2$ where $A$ satisfies
    \begin{align*}
        \mu^2\left(A\right) > 1 - \frac{\epsilon_{\rm{GME}}}{4(\log\alpha)^2}.
    \end{align*}
\end{theorem}

\noindent We are now ready to present the geometry-preserving encoder/decoder (GPE) framework:

\paragraph{GPE Framework}
Given an empirical distribution $\hat\mu_n = \frac{1}{n}\sum^n_{i=1} \delta_{x_i}$, where $\{x_i\}$ are sampled from the data distribution $\mu \in \mathbb{P}(\mathcal{M})$:

\begin{enumerate}
    \item Compute the GPE encoder by solving the following minimization problem:
    \begin{equation}\label{eq:gme-min}
        \inf_{T:\mathcal{M}\rightarrow \mathbb{R}^d} \GME(T, \hat\mu_n).
    \end{equation}
    \item Once the GPE encoder is computed, calculate the GPE decoder by solving:
    \begin{equation*} 
        \min_{S:{T(\mathcal{M})}\rightarrow {\mathbb{R}^D}}  L_{\rm{rec}}(S,T)
    \end{equation*}
    where $L_{\rm{rec}}$ is a quadratic cost function defined as
    \[
        L_{\rm{rec}}(S,T)
        =\mathbb{E}_{x\sim\hat\mu_n} \|S\circ T(x) - x\|^2.
    \]
\end{enumerate}

Note that there are several differences in the usage of the GME cost between this paper and the works by \cite{lee2023monotone} and \cite{yang2024improving}. The key distinction is that the cost function in \eqref{eq:gme-cost} was used as a regularization term in a minimax problem within the GAN framework and hyperbolic neural networks to weakly enforce the geometry-preserving property on the encoder. However, in this paper, we use the GME cost as the sole minimization objective, making it crucial to achieve the smallest possible value, which implies the weak bi-Lipschitz condition on a large subset of the domain $\mathcal{M}^2$. This approach necessitates the development of a more refined theory for this variational formulation to establish stability and convergence results for the GME cost. We also discuss Hessian analysis to explore the smoothness properties and convergence results of the gradient descent algorithm applied to the GME cost. Additionally, given an encoder that minimizes the GME cost, we present new theorems for error bounds, demonstrating that this theory can be both useful and practical in the latent diffusion model when using the GPE framework.

\begin{remark}
  The encoder loss in \cref{eq:gme-cost} shares the foundational goal of Multidimensional Scaling (MDS): finding an embedding that preserves pairwise distances to approximate a $1$-bi-Lipschitz map. However, our GME formulation offers two critical advantages. First, classical MDS requires the entire dataset to compute the embedding, rendering it intractable for large-scale datasets like CIFAR-10 and CelebA (\Cref{sec:numeric}). In contrast, our neural network-based GME supports minibatch training, ensuring practical scalability.

Second, and most importantly, GME introduces a \emph{logarithmic} cost function that provides a profound computational and optimization advantage over the squared cost of standard MDS. For an $L$-Lipschitz embedding map $T$, the Hessian of the MDS loss grows polynomially as $\mathcal{O}(L^2)$, which rapidly degrades the condition number of the optimization landscape and forces computationally inefficient, vanishing step sizes. By contrast, the logarithmic formulation of GME exponentially dampens this curvature, strictly bounding the Hessian by $\mathcal{O}(\log L)$. This ensures a highly stable, well-conditioned optimization trajectory that remains computationally efficient even when $L$ becomes large, a theoretical property we directly validate in our empirical results (\Cref{subsec:isometry_experiments}).
\end{remark}

\subsection{Comparing GPE with VAE}\label{sec:comparison-gpe-vae}

The GPE framework fundamentally differs from the Variational Autoencoder (VAE) \cite{kingma2013auto} in both its objective and its treatment of the latent space. VAEs model the data distribution $\mu$ on $\mathcal{M} \subset \mathbb{R}^D$ using a probabilistic encoder $q(z \mid x)$ and decoder $p(x \mid z)$, trained by maximizing the Evidence Lower Bound (ELBO)
\[
\mathcal{L}_{\mathrm{ELBO}} := \mathbb{E}_{x \sim \mu} \Big[ \mathbb{E}_{z \sim q(z \mid x)} [\log p(x \mid z)] - \beta \, \mathrm{KL}(q(z \mid x) \, \| \, p(z)) \Big],
\]
where $p(z) = \mathcal{N}(0, I)$ is typically a fixed standard Gaussian prior. 

To connect this probabilistic formulation to our geometric analysis, it is highly instructive to consider the deterministic limit of the VAE (i.e., as the variance in both the approximate posterior and conditional likelihood approaches zero). In this regime, the stochastic mappings reduce to deterministic functions $T$ and $S$, and the VAE training implicitly simplifies to balancing a deterministic reconstruction error against a prior-matching penalty
\[
\min_{T, S} \Big\{ \mathbb{E}_{x \sim \mu}\|S(T(x)) - x\|^2 + \beta\, \mathrm{KL}(T_\#\mu \,\|\, \mathcal{N}) \Big\}.
\]
We note that this deterministic limit has been rigorously investigated in recent literature \cite{reizinger2022embrace}, which reveals that VAEs can exhibit an inductive bias toward orthogonal, geometry-preserving coordinate transforms under certain conditions. However, standard VAEs still force the macroscopic encoded distribution to fit a fixed, rigid prior ($T_\#\mu \approx \mathcal{N}$). 

Enforcing this strict prior matching often severely distorts the intrinsic data geometry. If the data manifold is topologically complex or clustered (e.g., a mixture of Gaussians), forcing it into a single unimodal Gaussian inevitably leads to significant geometric expansion or contraction, causing the bi-Lipschitz constant to explode. The following 1D example explicitly illustrates this phenomenon.

\begin{example}
Consider the case where $\mu$ is a well-separated two-component Gaussian mixture, $f_X(x)=\tfrac12\,\phi_\sigma(x-m)+\tfrac12\,\phi_\sigma(x+m)$ with $m\gg\sigma$, and the target $\nu$ is the standard normal distribution $f_Z(z)=\phi_1(z)$. The optimal monotone transport map $T$ pushing $\mu$ forward to $\nu$ satisfies the differential equation \cite{villani2009optimal, santambrogio2015optimal}
\begin{equation}\label{eq:quantile-derivative}
T'(x)\;=\;\frac{f_X(x)}{f_Z\bigl(T(x)\bigr)}.
\end{equation}
Near the barycentric region $x\approx 0$, the mixture density is exponentially small: $f_X(0)\approx \frac{1}{\sigma\sqrt{2\pi}} e^{-m^2/(2\sigma^2)}$. Meanwhile, symmetry implies $T(0) \approx 0$, so $f_Z(T(0)) \approx f_Z(0) = (2\pi)^{-1/2}$. Substituting these into \eqref{eq:quantile-derivative} yields
\[
T'(0)\ \approx\ \frac{1}{\sigma}\,e^{-m^2/(2\sigma^2)}\ \ll\ 1.
\]
This demonstrates that points mapped into the central peak of the Gaussian prior are violently contracted, destroying the natural distance structure between the separated clusters. By contrast, GPE drops the restrictive $T_\#\mu = \mathcal{N}$ constraint entirely, allowing the latent distribution to organically adapt and strictly preserve pairwise distances.
\end{example}

In contrast, the GPE framework allows for independent training of the encoder and decoder through
\begin{align*}
    &\inf_{T:\mathcal{M}\to\mathbb{R}^d} \GME(T),\quad
    \min_{S:{T(\mathcal{M})\to \mathbb{R}^D}} L_{\rm{rec}}(S,T).
\end{align*}
This independence simplifies the training process, eliminating the need to balance or trade off parameters. Each minimization problem can be analyzed separately, with both demonstrating stable and efficient training properties. Additionally, by \Cref{thm:T-bi-lip-mu}, minimizing the GME cost ensures that $T$ satisfies the weak $\alpha$-bi-Lipschitz condition with $\alpha$ being closer to 1, which also implies the regular bi-Lipschitz condition for well-separated pairs of points. The difference between the embedded distributions from the VAE and GPE encoders is also discussed in an experiment presented by \cite{lee2023monotone}.

\begin{figure}[!ht]
    \centering

\tikzset{every picture/.style={line width=0.75pt}} 

\begin{tikzpicture}[x=0.57pt,y=0.57pt,yscale=-1,xscale=1]

\draw    (269,292.5) -- (448.5,292.5) ;
\draw    (271,285.25) .. controls (331,285.25) and (338.91,223.93) .. (358.5,223.75) .. controls (378.09,223.57) and (391.5,285.25) .. (447.5,285.25) ;
\draw    (270.5,122.5) -- (450,122.5) ;
\draw    (272.5,115.25) .. controls (300,114.25) and (292.91,51.93) .. (312.5,51.75) .. controls (332.09,51.57) and (331.5,113.25) .. (358.5,113.75) ;
\draw    (358.5,113.75) .. controls (386,112.75) and (381.41,51.43) .. (401,51.25) .. controls (420.59,51.07) and (421,115.75) .. (448,116.25) ;
\draw    (360,130) -- (360.97,215.25) ;
\draw [shift={(361,218.25)}, rotate = 269.35] [fill={rgb, 255:red, 0; green, 0; blue, 0 }  ][line width=0.08]  [draw opacity=0] (8.93,-4.29) -- (0,0) -- (8.93,4.29) -- cycle    ;
\draw    (400,130) -- (380.59,227.06) ;
\draw [shift={(380,230)}, rotate = 281.31] [fill={rgb, 255:red, 0; green, 0; blue, 0 }  ][line width=0.08]  [draw opacity=0] (8.93,-4.29) -- (0,0) -- (8.93,4.29) -- cycle    ;
\draw    (320,130) -- (339.41,227.06) ;
\draw [shift={(340,230)}, rotate = 258.69] [fill={rgb, 255:red, 0; green, 0; blue, 0 }  ][line width=0.08]  [draw opacity=0] (8.93,-4.29) -- (0,0) -- (8.93,4.29) -- cycle    ;
\draw    (340,130) -- (351.12,216.27) ;
\draw [shift={(351.5,219.25)}, rotate = 262.66] [fill={rgb, 255:red, 0; green, 0; blue, 0 }  ][line width=0.08]  [draw opacity=0] (8.93,-4.29) -- (0,0) -- (8.93,4.29) -- cycle    ;
\draw    (380,130) -- (370.33,217.02) ;
\draw [shift={(370,220)}, rotate = 276.34] [fill={rgb, 255:red, 0; green, 0; blue, 0 }  ][line width=0.08]  [draw opacity=0] (8.93,-4.29) -- (0,0) -- (8.93,4.29) -- cycle    ;
\draw    (420.5,129.63) -- (394.65,246.32) ;
\draw [shift={(394,249.25)}, rotate = 282.49] [fill={rgb, 255:red, 0; green, 0; blue, 0 }  ][line width=0.08]  [draw opacity=0] (8.93,-4.29) -- (0,0) -- (8.93,4.29) -- cycle    ;
\draw    (438.5,130.63) -- (434.58,274.75) ;
\draw [shift={(434.5,277.75)}, rotate = 271.56] [fill={rgb, 255:red, 0; green, 0; blue, 0 }  ][line width=0.08]  [draw opacity=0] (8.93,-4.29) -- (0,0) -- (8.93,4.29) -- cycle    ;
\draw    (280,130) -- (284.41,273.75) ;
\draw [shift={(284.5,276.75)}, rotate = 268.24] [fill={rgb, 255:red, 0; green, 0; blue, 0 }  ][line width=0.08]  [draw opacity=0] (8.93,-4.29) -- (0,0) -- (8.93,4.29) -- cycle    ;
\draw    (300,130) -- (315.62,251.77) ;
\draw [shift={(316,254.75)}, rotate = 262.69] [fill={rgb, 255:red, 0; green, 0; blue, 0 }  ][line width=0.08]  [draw opacity=0] (8.93,-4.29) -- (0,0) -- (8.93,4.29) -- cycle    ;
\draw    (30,292.5) -- (209.5,292.5) ;
\draw    (31.5,122.5) -- (211,122.5) ;
\draw    (31,115.25) .. controls (58.5,114.25) and (51.41,51.93) .. (71,51.75) .. controls (90.59,51.57) and (90,113.25) .. (117,113.75) ;
\draw    (117,113.75) .. controls (144.5,112.75) and (139.91,51.43) .. (159.5,51.25) .. controls (179.09,51.07) and (179.5,115.75) .. (206.5,116.25) ;
\draw    (31,287) .. controls (58.5,286) and (51.41,223.68) .. (71,223.5) .. controls (90.59,223.32) and (90,285) .. (117,285.5) ;
\draw    (117,285.5) .. controls (144.5,284.5) and (139.91,223.18) .. (159.5,223) .. controls (179.09,222.82) and (179.5,287.5) .. (206.5,288) ;
\draw    (40,129.75) -- (40.98,265) ;
\draw [shift={(41,268)}, rotate = 269.59] [fill={rgb, 255:red, 0; green, 0; blue, 0 }  ][line width=0.08]  [draw opacity=0] (8.93,-4.29) -- (0,0) -- (8.93,4.29) -- cycle    ;
\draw    (60,129.75) -- (60.97,215) ;
\draw [shift={(61,218)}, rotate = 269.35] [fill={rgb, 255:red, 0; green, 0; blue, 0 }  ][line width=0.08]  [draw opacity=0] (8.93,-4.29) -- (0,0) -- (8.93,4.29) -- cycle    ;
\draw    (80,129.75) -- (80.97,215) ;
\draw [shift={(81,218)}, rotate = 269.35] [fill={rgb, 255:red, 0; green, 0; blue, 0 }  ][line width=0.08]  [draw opacity=0] (8.93,-4.29) -- (0,0) -- (8.93,4.29) -- cycle    ;
\draw    (100,129.75) -- (100.98,265) ;
\draw [shift={(101,268)}, rotate = 269.59] [fill={rgb, 255:red, 0; green, 0; blue, 0 }  ][line width=0.08]  [draw opacity=0] (8.93,-4.29) -- (0,0) -- (8.93,4.29) -- cycle    ;
\draw    (120,129.75) -- (120.98,275) ;
\draw [shift={(121,278)}, rotate = 269.61] [fill={rgb, 255:red, 0; green, 0; blue, 0 }  ][line width=0.08]  [draw opacity=0] (8.93,-4.29) -- (0,0) -- (8.93,4.29) -- cycle    ;
\draw    (140,129.75) -- (140.97,225) ;
\draw [shift={(141,228)}, rotate = 269.42] [fill={rgb, 255:red, 0; green, 0; blue, 0 }  ][line width=0.08]  [draw opacity=0] (8.93,-4.29) -- (0,0) -- (8.93,4.29) -- cycle    ;
\draw    (161,128) -- (161,215) ;
\draw [shift={(161,218)}, rotate = 270] [fill={rgb, 255:red, 0; green, 0; blue, 0 }  ][line width=0.08]  [draw opacity=0] (8.93,-4.29) -- (0,0) -- (8.93,4.29) -- cycle    ;
\draw    (180,128) -- (180.97,235) ;
\draw [shift={(181,238)}, rotate = 269.48] [fill={rgb, 255:red, 0; green, 0; blue, 0 }  ][line width=0.08]  [draw opacity=0] (8.93,-4.29) -- (0,0) -- (8.93,4.29) -- cycle    ;
\draw    (201,128) -- (201,275) ;
\draw [shift={(201,278)}, rotate = 270] [fill={rgb, 255:red, 0; green, 0; blue, 0 }  ][line width=0.08]  [draw opacity=0] (8.93,-4.29) -- (0,0) -- (8.93,4.29) -- cycle    ;

\draw (330,299) node [anchor=north west][inner sep=0.75pt]   [align=left] {$\displaystyle \nu \in \mathbb{P}(\mathbb{R}^d)$};
\draw (330.5,17) node [anchor=north west][inner sep=0.75pt]   [align=left] {$\displaystyle \mu\in \mathbb{P}(\mathcal{M}) $};
\draw (80,299) node [anchor=north west][inner sep=0.75pt]   [align=left] {$\displaystyle T_\#\mu \in \mathbb{P}(\mathbb{R}^d)$};
\draw (85.5,17) node [anchor=north west][inner sep=0.75pt]   [align=left] {$\displaystyle \mu \in \mathbb{P}(\mathcal{M})$};
\draw (99,323) node [anchor=north west][inner sep=0.75pt]   [align=left] {GPE};
\draw (340,323) node [anchor=north west][inner sep=0.75pt]   [align=left] {VAE};

\end{tikzpicture}
\caption{
Illustration showing how GPE (left) and VAE (right) encoders work. The GPE encoder embeds the data distribution into the latent space while maintaining the geometric structure of the data distribution. In contrast, the VAE encoder maps the data distribution onto the latent distribution (or a prior distribution) $\nu$ such that $T_\#\mu \approx \nu$. Thus, there exist mappings that send distinct inputs to nearby points, resulting in $\alpha$ close to $0$.}
\label{fig:illus-vae-gpe}
\end{figure}

In what follows, we provide theoretical properties of the GPE framework for the reconstruction task from an optimization perspective.

\section{Theoretical properties of GPE}\label{sec:theory-gpe}
In this section, we provide theoretical properties of the GME costs and focus on the efficiency of training the GPE encoder and decoder, as well as the convergence analysis of using gradient descent algorithms for this purpose. Throughout our analysis, we study the optimization landscape directly in the function space (evaluating the continuous maps $T$ and $S$) rather than the parameter space (e.g., the neural network weights $\theta$). We take this approach because it isolates the intrinsic geometric properties of the GME objective from the arbitrary architectural choices of a specific neural network. Because parameter-space dynamics are fundamentally governed by their function-space counterparts, establishing stability and convergence in the function space provides the foundation needed to justify the robust optimization observed in practice.

To make this connection explicit for an arbitrary cost function $F$, consider a map parameterized by a neural network, $T_\theta$. In practice, gradient descent updates the parameters via $\theta \leftarrow \theta - \eta \nabla_\theta F(T_\theta)$. By the chain rule, this parameter-space gradient is fundamentally driven by the function-space gradient
\[
    \nabla_\theta F(T_\theta) = J_\theta^* \nabla_{L^2(\mu)} F(T_\theta),
\]
where $J_\theta^*$ is the adjoint of the Jacobian of the network with respect to $\theta$ and $\nabla_{L^2(\mu)}$ is a graident in $L^2(\mu)$ norm discussed in \Cref{def:gradient}. Because the network's learning dynamics are directly proportional to $\nabla_{L^2(\mu)} F(T_\theta)$, analyzing the smoothness and stability of the cost in the $L^2(\mu)$ space naturally dictates the stability of the empirical training process, provided the chosen neural network architecture has well-behaved gradients. 

To achieve this, we provide theoretical results that include the local isometry property of the GPE encoder (\Cref{ssec:local_isometry}), alongside a Hessian analysis of the minimization problems for both the GPE encoder in \eqref{eq:gme-cost} (\Cref{sec:stability-encoder}) and the decoder given the GPE encoder (\Cref{sec:effi-decoder}).

The first theoretical property of the GME cost is that it does not admit a collapsed map as a local minimizer. A collapsed map is one that sends every point of $\mathcal{M}$ to the same vector $c \in \mathbb{R}^d$. Although such a map trivially has zero gradient (since constant shifts do not affect pairwise distances), one can construct a non-constant perturbation along which the second variation of the GME cost is strictly negative. Hence, no collapsed map can be a local minimizer.

\begin{proposition}
  The constant map $T(x)=c$ is not a local minimizer of the GME cost \cref{eq:gme-cost}
  for any non‐degenerate data distribution $\mu$ on~$\M$ such that $\mu^2\bigl[(x,x'): \ x\neq x'\bigr]>0$.
\end{proposition}

\begin{proof}
  Let $T_0(x)\equiv c$.  For any perturbation ${h}:\M\to\R^d$, define
    $T_\varepsilon(x)
    := c + \varepsilon\,{h}(x).$
  The GME cost becomes
  \begin{align*}
    \GME(T_\varepsilon)
    \;=\;\E_{(x,x')\sim\mu^2}\Bigl[\log\!\bigl(\tfrac{1+ \varepsilon^2 \|{h}(x)-{h}(x')\|^2}{1+\|x-x'\|^2}\bigr)^2\Bigr]
  \end{align*}
  For small $\varepsilon$, 
  \[
    \log\bigl(1+\varepsilon^2\|{h}(x)-{h}(x')\|^2\bigr)
    = \varepsilon^2\,\|{h}(x)-{h}(x')\|^2 + O(\varepsilon^4).
  \]
  Hence
  \begin{align*}
    \GME(T_\varepsilon) 
    &= 
     \E_{(x,x')\sim\mu^2}\Big(-\log (1+\|x-x'\|^2) \;+\;\varepsilon^2\,\|{h}(x)-{h}(x')\|^2 + O(\varepsilon^4)\Big)^2\\
    &= \E_{(x,x')\sim\mu^2}(\log (1+\|x-x'\|^2))^2\\
      &\hspace{2cm} - 2\,\varepsilon^2\,\E_{(x,x')\sim\mu^2}\,\Big[\log (1+\|x-x'\|^2)\;\|{h}(x)-{h}(x')\|^2\Big]
      + O(\varepsilon^4).
  \end{align*}
  Since $\log(1+\|x-x'\|^2)>0$ on a set of positive $\mu^2$-measure, any non-trivial perturbation ${h}$ (i.e.\ ${h}(x)\neq{h}(x')$ on a set of positive $\mu^2$-measure) satisfies
\[
  \E\bigl[\log(1+\|x-x'\|^2)\,\|{h}(x)-{h}(x')\|^2\bigr] > 0.
\]
Therefore, the coefficient of $\varepsilon^2$ is strictly negative, which implies that the second variation in the direction of ${h}$ is strictly negative. Hence $T_0$ cannot be a local minimizer.
\end{proof}
Throughout the theoretical results presented in \Cref{sec:theory-gpe} and \Cref{sec:gpe-lgm}, we assume a minimal regularity condition on the GPE encoders $T$.

\medskip
\paragraph*{Minimal Regularity Assumption on GPE Encoders:}
\begin{equation}\label{eq:reg-a-t}
    \mathcal{T}_\beta^k := 
    \left\{
        T \in C^k(\mathcal{M}, \mathbb{R}^d),\; (k\ge 3) \;\middle|\; 
        \|T\|_{C^2} \leq \beta,\quad T(0)= 0
    \right\}.
\end{equation}
The assumption that $T \in \mathcal{T}_\beta^k$ is a mild regularity requirement. Neural networks with smooth activations (e.g., SiLU, GELU) naturally possess bounded first and second derivatives on a compact manifold $\mathcal{M}$, provided their weights remain bounded. 

To guarantee this boundedness throughout the training loop (e.g., during the parameter updates in Algorithm 1), one can explicitly restrict the network parameters. For instance, enforcing a uniform bound on the weights such that each parameter $\theta \in [-a, a]$ for some constant $a > 0$ restricts the parameter space to a compact set. Because a neural network is fundamentally composed of sequential linear matrix multiplications and smooth activations, strictly bounding the individual weights naturally imposes a strict upper bound on the network's Lipschitz constant and its higher-order derivatives. Along with other standard deep learning techniques such as weight decay ($L_2$ regularization) or spectral normalization, this artificially bounds the parameter space \cite{Miyato2018Spectral, Gouk2021Regularisation, Fazlyab2019Efficient}. 

However, in practice, we observe that the optimization of the GME objective is stable without requiring these restrictive interventions. Because the GME loss controls the geometric distortion and enjoys the tight logarithmic Hessian bounds discussed previously, the network parameters do not explode. Consequently, training proceeds in a very simple, unconstrained manner while implicitly remaining within a bounded regime. Because the theoretical bound $\beta$ can be arbitrarily large, this provides the minimal regularity necessary to ensure stability while maintaining the practical simplicity of the model.

\bigskip

Throughout the paper, we assume that $\beta$ is sufficiently large so that
\begin{equation}\label{eq:assmpt-beta-diam}
    1+\diam(\mathcal{M})^2 < \beta^2.
\end{equation}
This condition is easily satisfied under the neural network parameterization, as the bound on $\beta$ simply accommodates the naturally occurring bounds of the network weights during stable training. This simplification allows us to streamline the notation in the results that follow.

An immediate consequence of our regularity assumptions is that the constraint set $\mathcal{T}_\beta^k$ forms a compact space. This geometric property, coupled with the continuity of the GME objective, formally guarantees that the minimization problem in \eqref{eq:gme-min} is well-posed, bounded, and admits a global minimizer within the set. The proof is provided in the appendix.

\begin{proposition}[Compactness and Existence of a Minimizer]\label{prop:compact-minimizer}
Let $\mu \in \mathbb{P}(\mathcal{M})$ and fix $\beta > 0$. The set ${\mathcal{T}}_\beta^k$ is a compact subset of $C^1(\mathcal{M}, \mathbb{R}^d)$. Furthermore, the GME cost defined in \eqref{eq:gme-cost} satisfies the uniform bound $0 \leq \GME(T) \leq 4 (\log(\beta))^2$ for all $T \in \mathcal{T}_\beta^k$,
and attains its minimum on this set; that is, there exists a minimizer $ T^* \in \mathcal{T}_\beta^k $ such that $\GME(T^*) = \min_{T \in \mathcal{T}_\beta^k} \GME(T)$.
\end{proposition}


\subsection{Local isometry of a GPE encoder}\label{ssec:local_isometry}

A fundamental objective in manifold representation learning is to construct an encoder that preserves the intrinsic geometric structure of the data. While minimizing the Global Mapping Error (GME) cost explicitly penalizes the distortion of pairwise distances across the manifold, it is not immediately obvious that a small global integral error guarantees metric preservation at the infinitesimal level. Ideally, an optimal encoder should act as a local isometry, meaning its differential mapping exactly preserves the lengths of tangent vectors, neither stretching nor compressing the manifold locally.

In this subsection, we establish a theoretical link showing that minimizing the global GME cost implicitly forces the encoder to become a local isometry. Theorem \ref{thm:local-distortion-bound} quantifies this relationship by bounding the average local directional distortion, a measure of how much the lengths of unit tangent vectors deviate from 1 under the pushforward of the encoder. 

The significance of this theorem is twofold. First, it provides a strong theoretical justification for using GME as a training objective: it guarantees that as the global training loss approaches zero, the learned mapping inherently converges to a strict local isometry. Second, by explicitly incorporating the manifold's dimension $m$ and the encoder's smoothness order $k$, the theorem reveals how higher-order geometric regularity allows us to translate coarse, global pairwise bounds into tight, local metric bounds.

\begin{theorem}[Local isometry of GPE encoder]\label{thm:local-distortion-bound}
Let $ \mathcal{M} \subset \mathbb{R}^D $ be a compact, smooth $ m $-dimensional manifold and $ \mu  \in \mathbb{P}(\mathcal{M})$. Suppose $T \in \mathcal{T}_{\beta}^k$. Fix any integer $k \ge 3$. Let $A(x) = \int_{U_x\M} (\|dT_x(v)\|^2 - 1)^2 d\sigma(v)$ be the local directional distortion at $x$, where $dT_x: \mathcal{T}_x\mathcal{M} \to \mathbb{R}^d$ is the differential of $T$ and $U_x\mathcal{M} = \{v \in \mathcal{T}_x\mathcal{M} : \|v\|_{g_x} = 1\}$ denotes the unit tangent sphere at $x$. Then, there exists a constant $C > 0$, depending only on $m$, $k$, and the local geometry of $\mathcal{M}$, such that
\[
\int_{\mathcal{M}} A(x) d\mu(x) \leq C \big(\GME(T)\big)^{\frac{2k-4}{m+2k}}.
\]
\end{theorem}

\begin{proof}
Define the discrepancy function $f(x,x') = \|T(x)-T(x')\|^2 - \|x-x'\|^2$. Fix $x \in \mathcal{M}$ and a unit tangent vector $v \in U_x\mathcal{M}$. Define the 1D slice along the intrinsic geodesic as $u(t) = f(x, \exp_x(tv))$. Because $\exp_x(tv)$ has initial position $x$ and initial velocity $v$ (with $\|v\|_{g_x}=1$), the Taylor expansions of the squared ambient distances yield $u(0)=0$, $u'(0)=0$, and $u''(0) = 2B_{x,v}$, where $B_{x,v} = \|dT_x(v)\|^2 - 1$. 

Since $T$ is smooth and the exponential map is smooth, $u(t)$ admits a $k$-th order Taylor expansion around $t=0$
\[
u(t) = B_{x,v} t^2 + \sum_{j=3}^{k-1} a_j t^j + R_k(t)
\]
where the intermediate coefficients $a_j$ encode the extrinsic curvature of $\mathcal{M}$, and the remainder satisfies $|R_k(t)| \leq K t^k$ for a constant $K$ depending uniformly on $T$ (via the bounds in $\mathcal{T}_\beta^k$) and $\mathcal{M}$.

Let $\epsilon > 0$ be a globally uniform integration radius, strictly smaller than the injectivity radius of $\mathcal{M}$. We construct a weight function $W_\epsilon(t)$ on $[0,\epsilon]$ to isolate $B_{x,v}$ by annihilating the intermediate terms when integrated against the spherical volume element $t^{m-1}dt$. 

Treating $L^2([0,1])$ as a Hilbert space, consider the finite-dimensional subspace $V = \text{span}\{s^{m+2}, \dots, s^{m+k-2}\}$. Let $q(s) = s^{m+1} \notin V$ be the target function. We define a polynomial $P(s) = Q(s) / \|Q\|^2$, where $Q(s)$ is the orthogonal projection of $q(s)$ onto $V^\perp$. By construction, $\int_0^1 P(s) s^{j+m-1} ds = 0$ for $3 \leq j \leq k-1$, and $\int_0^1 P(s) s^{m+1} ds = 1$.

Scaling to $[0, \epsilon]$, we define $W_\epsilon(t) = \frac{1}{\epsilon^{m+2}} P\left(\frac{t}{\epsilon}\right)$. Multiplying $u(t)$ by $W_\epsilon(t) t^{m-1}$ and integrating from $0$ to $\epsilon$ eliminates the sum over $a_j$
\[
\int_0^\epsilon u(t) W_\epsilon(t) t^{m-1} dt = B_{x,v} + \int_0^\epsilon R_k(t) W_\epsilon(t) t^{m-1} dt
\]

Rearranging for $B_{x,v}$ and applying the Cauchy-Schwarz inequality to the first term yields
\[
|B_{x,v}| \leq \|u\|_{L^2(\epsilon)} \left( \int_0^\epsilon W_\epsilon(t)^2 t^{m-1} dt \right)^{1/2} + \int_0^\epsilon |R_k(t)| |W_\epsilon(t)| t^{m-1} dt
\]
where $\|u\|_{L^2(\epsilon)}^2 = \int_0^\epsilon u(t)^2 t^{m-1} dt$. By substitution, $\int_0^\epsilon W_\epsilon(t)^2 t^{m-1} dt = C_P^2 \epsilon^{-(m+4)}$. Using $|R_k(t)| \leq K t^k$, the remainder term is bounded by $C_K \epsilon^{k-2}$. Squaring the inequality results in
\[
B_{x,v}^2 \leq 2 C_P^2 \epsilon^{-(m+4)} \|u\|_{L^2(\epsilon)}^2 + 2 C_K^2 \epsilon^{2k-4}
\]

\vspace{1em}
Integrating $B_{x,v}^2$ over all unit tangent directions $v \in U_x\mathcal{M}$ yields the local directional distortion $A(x)$
\[
A(x) = \int_{U_x\mathcal{M}} B_{x,v}^2 d\sigma_x(v) \leq 2 C_P^2 \epsilon^{-(m+4)} \left( \int_{U_x\mathcal{M}} \|u\|_{L^2(\epsilon)}^2 d\sigma_x(v) \right) + 2 C_K^2 \epsilon^{2k-4}
\]

We relate the integral of $\|u\|_{L^2(\epsilon)}^2$ to the integral over the local geodesic ball $B_{\mathcal{M}}(x, \epsilon)$. Using normal coordinates $x' = \exp_x(tv)$, the intrinsic Riemannian volume element is $d\mu(x') = \theta(t,v) t^{m-1} dt d\sigma_x(v)$. Because $\mathcal{M}$ is compact and $\epsilon$ is small, the density $\theta(t,v)$ is bounded below by a geometric constant $c > 0$. Therefore
\[
\int_{B_{\mathcal{M}}(x, \epsilon)} f(x,x')^2 d\mu(x') \geq c \int_{U_x\mathcal{M}} \int_0^\epsilon f(x, \exp_x(tv))^2 t^{m-1} dt d\sigma_x(v)
\]
Recognizing the inner integral as $\|u\|_{L^2(\epsilon)}^2$, we obtain
\[
\int_{U_x\mathcal{M}} \|u\|_{L^2(\epsilon)}^2 d\sigma_x(v) \leq \frac{1}{c} \int_{B_{\mathcal{M}}(x, \epsilon)} f(x,x')^2 d\mu(x')
\]
Substituting this back establishes the pointwise bound for $x$
\[
A(x) \leq C_1 \epsilon^{-(m+4)} \int_{B_{\mathcal{M}}(x, \epsilon)} f(x,x')^2 d\mu(x') + C_2 \epsilon^{2k-4}
\]
where $C_1 = \frac{2 C_P^2}{c}$ and $C_2 = 2 C_K^2$.

\vspace{1em}
We integrate this pointwise bound over the entire manifold $\mathcal{M}$ with respect to the probability measure $d\mu(x)$
\[
\int_{\mathcal{M}} A(x) d\mu(x) \leq C_1 \epsilon^{-(m+4)} \int_{\mathcal{M}} \int_{B_{\mathcal{M}}(x, \epsilon)} f(x,x')^2 d\mu(x') d\mu(x) + C_2 \epsilon^{2k-4}
\]
Since $f(x,x')^2$ is non-negative, the integral over the local geodesic balls is trivially bounded by the integral over the entire domain $\mathcal{M} \times \mathcal{M}$, which is exactly the global mapping error
\[
\int_{\mathcal{M}} \int_{B_{\mathcal{M}}(x, \epsilon)} f(x,x')^2 d\mu(x') d\mu(x) \leq \int_{\mathcal{M}} \int_{\mathcal{M}} f(x,x')^2 d\mu(x') d\mu(x) = \GME(T)
\]
This yields a global inequality strictly in terms of the uniform radius $\epsilon$
\[
\int_{\mathcal{M}} A(x) d\mu(x) \leq C_1 \epsilon^{-(m+4)} \GME(T) + C_2 \epsilon^{2k-4}
\]
To yield the tightest bound, we balance the competing terms by setting $C_1 \epsilon^{-(m+4)} \GME(T) = C_2 \epsilon^{2k-4}$. This produces the optimal scaling radius
\[
\epsilon = \left( \frac{C_1}{C_2} \GME(T) \right)^{\frac{1}{m+2k}}
\]
Substituting this optimal $\epsilon$ back into the inequality completes the proof
\[
\int_{\mathcal{M}} A(x) d\mu(x) \leq C \cdot \GME(T)^{\frac{2k-4}{m+2k}},\quad \text{where $C = 2 C_1 \left( \frac{C_1}{C_2} \right)^{-\frac{m+4}{m+2k}}$.}
\]
\end{proof}
While Theorem \ref{thm:local-distortion-bound} establishes that the global mapping error tightly controls the average directional distortion across the manifold, probability density estimation requires a strictly stronger guarantee. If the encoder $T$ were allowed to have isolated, microscopic regions of extreme distortion (spikes), the pulled-back density $\rho_T$ could still exhibit severe local errors even if the integral of the distortion was small. 

To bridge the gap between an average $L^1$ bound and a pointwise $L^\infty$ bound, we leverage the higher-order smoothness of the encoder. Because the first and second derivatives of $T$ are bounded, the local distortion $A(x)$ is Lipschitz continuous. Geometrically, this prevents $A(x)$ from exhibiting arbitrarily narrow spikes; any large local distortion must be accompanied by a wide cone of distortion around it, which would violate the integral bound established in the previous theorem. This topological constraint allows us to convert our global integral bound directly into a uniform, pointwise guarantee on the Jacobian determinant everywhere on the manifold.

\begin{corollary}[Uniform Bounds on Jacobian of a GPE encoder]\label{cor:A-bound}
Let $\mathcal{M} \subset \mathbb{R}^D$ be a compact, smooth $m$-dimensional Riemannian manifold and $\mu \in \mathbb{P}(\mathcal{M})$. Suppose $T \in \mathcal{T}_{\beta}^k$, and let $J_T(x)$ denote the determinant of the Jacobian of $T$ at $x$. There exists a constant $K > 0$ depending on $m$, $\beta$, and the local geometry of $\mathcal{M}$, such that
\[
\exp\left( - K \big(\GME(T)\big)^{c_{m,k}} \right) \leq J_T(x) \leq \exp\left( K \big(\GME(T)\big)^{c_{m,k}} \right),\quad \forall x \in \mathcal{M}
\]
where the exponent is given by $c_{m,k} = \frac{k-2}{(m+2k)(m+1)}$.
\end{corollary}

\begin{proof}
The local directional distortion is defined as $A(x) = \int_{U_x\M} (v^\top \tilde{g}(x) v - 1)^2 d\sigma(v)$, where $\tilde{g}(x) = dT_x^\top dT_x$ is the induced metric. To establish that $A(x)$ is Lipschitz, we bound its gradient $\nabla A(x)$. By the chain rule
\[
\nabla A(x) = \int_{U_x\M} 2(v^\top \tilde{g}(x) v - 1) \cdot (v^\top \nabla \tilde{g}(x) v) d\sigma(v).
\]
Note that $\nabla \tilde{g}(x) = \nabla(dT^\top dT) = (\nabla dT)^\top dT + dT^\top (\nabla dT)$, which involves the Hessian $D^2 T$. Since $T \in \mathcal{T}_{\beta}^k$ is smooth on a compact manifold, its $C^2$ norm $\|T\|_{C^2} = \max(\|dT\|_\infty, \|D^2 T\|_\infty)$ is bounded by $\beta$. Thus, the gradient magnitude is bounded by $\|\nabla A(x)\| \leq 4\beta^2(\beta^2+1) := L_\beta$. This establishes that $A(x)$ is $L_\beta$-Lipschitz.

Let $\eta = \sup_{x \in \mathcal{M}} A(x)$, attained at $x_0$. By Lipschitz continuity, $A(x) \geq \eta - L_\beta \cdot d(x, x_0)$. In a geodesic ball $B_r(x_0)$ with radius $r = \eta / (2L_\beta)$, we have $A(x) \geq \eta/2$. Integrating over $\mathcal{M}$, we have
\[
\int_{\mathcal{M}} A(x) d\mu(x) \geq \int_{B_r(x_0)} \frac{\eta}{2} d\mu(x) \geq \frac{\eta}{2} \cdot c \left( \frac{\eta}{2L_\beta} \right)^m = C_1 \eta^{m+1}.
\]
Using the integral bound $\|A\|_{L^1} \leq C_0 \GME(T)^{\frac{2k-4}{m+2k}}$ established in Theorem \ref{thm:local-distortion-bound}, we solve for $\eta$
\begin{equation}\label{eq:eta-final}
\eta \leq C^* \big(\GME(T)\big)^{\frac{2k-4}{(m+2k)(m+1)}}.
\end{equation}

To evaluate $A(x)$, we consider the unit sphere in the tangent space $U_x\mathcal{M} = \{v \in \mathcal\mathcal{T}_x\mathcal{M} : g_x(v, v) = 1\}$. We use the exponential map $\exp_x: \mathcal\mathcal{T}_x\mathcal{M} \to \mathcal{M}$ to define Riemann normal coordinates centered at $x$. Under this map, the tangent space is isometrically identified with $\mathbb{R}^m$, and the Riemannian metric satisfies $g_{ij}(x) = \delta_{ij}$. Consequently, $U_x\mathcal{M}$ is identified with the Euclidean unit sphere $\mathbb{S}^{m-1} \subset \mathbb{R}^m$. Let $\tilde{g}(x) = dT_x^\top dT_x$ be the induced metric represented in these coordinates, and let $\{\lambda_i\}_{i=1}^m$ be its eigenvalues. For any unit vector $v = \sum_{i=1}^m a_i v_i \in \mathbb{S}^{m-1}$ (where $\{v_i\}$ is an orthonormal eigenbasis), the local directional distortion is
\[
A(x) = \int_{\mathbb{S}^{m-1}} \left( \sum_{i=1}^m \lambda_i a_i^2 - 1 \right)^2 d\sigma(v) = \int_{\mathbb{S}^{m-1}} \left( \sum_{i=1}^m (\lambda_i - 1) a_i^2 \right)^2 d\sigma(v).
\]
Let $x_i = \lambda_i - 1$. Expanding the integrand, we obtain
\[
\left( \sum_{i=1}^m x_i a_i^2 \right)^2 = \sum_{i=1}^m x_i^2 a_i^4 + \sum_{i \neq j} x_i x_j a_i^2 a_j^2.
\]
Applying the spherical moment identities \cite{folland2001integrate}, which state that $\int_{\mathbb{S}^{m-1}} a_i^4 d\sigma = \frac{3}{m(m+2)}$ and $\int_{\mathbb{S}^{m-1}} a_i^2 a_j^2 d\sigma = \frac{1}{m(m+2)}$ for $i \neq j$, the integral becomes
\begin{align*}
    A(x) &= \frac{3}{m(m+2)} \sum_{i=1}^m x_i^2 + \frac{1}{m(m+2)} \sum_{i \neq j} x_i x_j \\
    &= \frac{1}{m(m+2)} \left[ 2 \sum_{i=1}^m x_i^2 + \left( \sum_{i=1}^m x_i^2 + \sum_{i \neq j} x_i x_j \right) \right] \\
    &= \frac{1}{m(m+2)} \left[ 2 \sum_{i=1}^m (\lambda_i - 1)^2 + \left( \sum_{i=1}^m (\lambda_i - 1) \right)^2 \right] \geq\frac{2}{m(m+2)} \sum_{i=1}^m (\lambda_i - 1)^2.
\end{align*}
This implies that for every $i \in \{1, \dots, m\}$, the eigenvalue deviation is controlled by the local distortion: $|\lambda_i - 1| \leq \sqrt{\frac{m(m+2)}{2}} \sqrt{A(x)}$.

The Jacobian determinant is $J_T(x) = (\prod \lambda_i)^{1/2}$. Taking the logarithm yields $|\ln J_T(x)| = \left| \frac{1}{2} \sum_{i=1}^m \ln(1 + (\lambda_i - 1)) \right|$. Assuming the GME is sufficiently small such that the deviations satisfy $|\lambda_i - 1| \leq 0.5$, we apply the Taylor bound $|\ln(1+z)| \leq 2|z|$
\[
|\ln J_T(x)| \leq \frac{1}{2} \sum_{i=1}^m 2|\lambda_i - 1| \leq m \sqrt{\frac{m(m+2)}{2}} \sqrt{\eta}.
\]
Substituting the expression for $\eta$ from \eqref{eq:eta-final} and defining $c_{m,k} = \frac{k-2}{(m+2k)(m+1)}$, we obtain $|\ln J_T(x)| \leq K \GME(T)^{c_{m,k}}$, where $K = m \sqrt{\frac{m(m+2)}{2}} \sqrt{C^*}$. Exponentiating yields the desired uniform bounds
\[
\exp\left( - K \GME(T)^{c_{m,k}} \right) \leq J_T(x) \leq \exp\left( K \GME(T)^{c_{m,k}} \right).
\]
This completes the proof.
\end{proof}

\subsection{Stability of GPE encoder training}\label{sec:stability-encoder}
In this section, we analyze the optimization landscape of the GME cost to establish how minimizing this objective improves the stability of the learned representation. Crucially, throughout this section and Section 5.3, we evaluate the first and second variations of our objectives with respect to the uniform Riemannian volume measures on the manifolds (i.e., the norms $\|\cdot\|_{L^2(\mathcal{M})}$ and $\|\cdot\|_{L^2(T(\mathcal{M}))}$) rather than the data-dependent distribution spaces $L^2(\mu)$ and $L^2(T_\#\mu)$.

This choice allows us to consistently isolate the intrinsic geometry of the mappings from the specific density of the data. In practice, models optimized strictly under data-dependent measures frequently suffer from geometric pathologies when the data distribution exhibits highly irregular densities or sparse regions between disconnected modes. For example, in normalizing flows and invertible neural networks, low-density regions often receive vanishing gradient signals, permitting the network to learn mappings with arbitrarily large local Lipschitz constants that severely destabilize inversion and generation \cite{behrmann2021understanding, cornish2020relaxing}. When evaluating the variations of a general functional $F$ under a data-dependent measure, the local concentration of data heavily obscures these severe geometric distortions. As we will demonstrate in Section 5.3, evaluating the decoder's reconstruction loss under the pushforward measure $L^2(T_\#\mu)$ trivially yields an identity-like Hessian, completely hiding the optimization instabilities caused by latent spatial distortion. 

To formalize our geometric framework, we first review the notions of first and second variations.

\begin{definition}\label{def:gradient}
    Given a Hilbert space $\mathcal{H}$ equipped with an inner product $\langle\cdot,\cdot\rangle_{\mathcal H}$, and a function $f:\mathcal{H}\to\mathbb{R}$, the first variation of $f$ at $u\in\mathcal{H}$ in the direction of $h\in\mathcal{H}$ is
    \begin{equation}\label{eq:def-first}
        \delta f(u)(h) = \lim_{t\rightarrow 0} \frac{f(u+th) - f(u)}{t}.
    \end{equation}
    If there exists an element $ \nabla f(u) \in \mathcal{H} $ such that
    \[
        \delta f(u)(h) = \langle \nabla f(u), h \rangle_{\mathcal{H}} \quad \text{for all } h \in \mathcal{H},
    \]
    then $ \nabla f(u) $ is called the $\mathcal{H}$-gradient of $ f $ at $ u $.

    The second variation of $f$ at $u$ in the direction of $h$ is defined by
    \begin{equation}\label{eq:def-second}
        \delta^2 f(u)(h,h) = \lim_{t\rightarrow 0} \frac{\delta f(u+th)(h) - \delta f(u)(h)}{t}.
    \end{equation}    
\end{definition}

With these definitions established, analyzing the variations under the Riemannian measure provides clear insights into the parameter-space optimization dynamics. By \Cref{def:gradient}, the functional gradients in the two spaces are fundamentally linked by this density: $\nabla_{L^2(\mathcal{M})} F(x) = \rho(x) \nabla_{L^2(\mu)} F(x)$. Applying the chain rule, the expected parameter-space gradient over the data distribution explicitly reveals how the intrinsic geometry drives learning
\begin{align}\label{eq:first-var-chain}
    \nabla_\theta F(T_\theta) &= \mathbb{E}_{x \sim \mu} \left[ J_\theta(x)^* \nabla_{L^2(\mu)} F(T_\theta)(x) \right]
    = \int_{\mathcal{M}} J_\theta(x)^* \nabla_{L^2(\mathcal{M})} F(T_\theta)(x) \, d\vol_{\mathcal{M}}(x),
\end{align}
where $J_\theta(x)^*$ is the adjoint Jacobian of the neural network with respect to its parameters $\theta$. 

While the first variation governs the direction of the gradient update, the ultimate efficiency and stability of optimization are dictated by the second variation (the functional Hessian), which characterizes the local curvature of the loss landscape. In gradient-based learning, ill-conditioned landscapes, where the Hessian exhibits a severely large condition number (the ratio of maximum to minimum eigenvalues), cause optimization to stall. The following proposition formalizes how bounding the functional Hessian strictly dictates the parameter-space optimization efficiency, governing both the maximum allowable learning rate and the convergence speed.
\begin{proposition}[Parameter-Space Curvature and Optimization]\label{prop:hessian-chain}
Let $\mathcal{M}$ be a manifold, $\mu$ be a data distribution supported on $\mathcal{M}$, and $u_\theta: \mathcal{M} \to \mathbb{R}^d$ be a parameterized mapping that is twice-differentiable with respect to both $x \in \mathcal{M}$ and $\theta \in \mathbb{R}^p$. For a functional $F$, let $H_\theta = \nabla_\theta^2 F(u_\theta)$ denote the parameter-space Hessian. Assume that near a local minimum, the functional gradient vanishes ($\nabla_{L^2(\mathcal{M})} F(u_\theta) \approx 0$), and the functional Hessian is bounded such that for all $h \in L^2(\mathcal{M}; \mathbb{R}^d)$
\begin{equation*}
    \lambda_{\min} \|h\|^2_{L^2(\mathcal{M})} \leq \delta^2 F(u_\theta)(h, h) \leq \lambda_{\max} \|h\|^2_{L^2(\mathcal{M})}, \quad \text{with } \lambda_{\max} \ge \lambda_{\min} > 0.
\end{equation*}
Furthermore, assume the mapping's Jacobian operator with respect to the parameters, $J_\theta(x) = \nabla_\theta u_\theta(x)$, maps any parameter variation $v \in \mathbb{R}^p$ to a square-integrable function $J_\theta v \in L^2(\mathcal{M}; \mathbb{R}^d)$, with its spectrum bounded by $\sigma_{\max} \ge \sigma_{\min} > 0$ such that $\sigma_{\min}^2 \|v\|^2 \leq \|J_\theta v\|_{L^2(\mathcal{M})}^2 \leq \sigma_{\max}^2 \|v\|^2$. Then, the parameter-space Hessian satisfies
\begin{equation*}
    \lambda_{\min} \sigma_{\min}^2 I \preceq H_\theta \preceq \lambda_{\max} \sigma_{\max}^2 I.
\end{equation*}
Consequently, for stable gradient descent, the maximum step size is bounded by $\eta < \frac{2}{\lambda_{\max} \sigma_{\max}^2}$, and the local linear convergence rate is governed by the condition number $\kappa \leq \frac{\lambda_{\max} \sigma_{\max}^2}{\lambda_{\min} \sigma_{\min}^2}$.
\end{proposition}

\Cref{prop:hessian-chain} provides a bridge between functional geometry and practical training. When the network has enough capacity ($\sigma_{\min} > 0$), a positive functional Hessian prevents the optimizer from getting stuck in flat regions. This is a formal consequence of the functional Polyak-Łojasiewicz (P-L) condition \cite{karimi2016linear}, because the Hessian is bounded away from zero, the gradient stays large enough to ensure rapid, linear convergence. Even in overparameterized cases where the network becomes redundant ($\sigma_{\min} = 0$), our analysis still offers a practical benefit. The upper bound ($\lambda_{\max}$) ensures the landscape stays smooth and free of sharp, sudden drops. This allows practitioners to use larger step sizes safely, ensuring the training remains stable even in locally flat regions of the parameter space.

Although this proposition analyzes the strictly positive bounds characteristic of locally convex basins, the actual GME cost function is globally non-convex. Building on this foundation, the following theorem characterizes the regularity of this non-convex GME optimization landscape for the encoder. Crucially, our results demonstrate that as the optimization process progresses and the GME cost decreases, the negative eigenvalues of the functional Hessian are strictly bounded and asymptotically vanish. This indicates that as the encoder minimizes the objective and approaches a local isometry, the landscape dynamically suppresses its non-convexity and enters the highly favorable, well-conditioned regime described above, theoretically guaranteeing higher stability and accelerated convergence.
\begin{theorem}\label{thm:log-GME-bound}
    Let $\mathcal M$ be a compact Riemannian manifold. For any $\beta > 0$ satisfying the bound in \eqref{eq:assmpt-beta-diam}, suppose $T \in C^1(\mathcal{M},\mathbb{R}^d)$ is $\beta$-Lipschitz and $\GME(T) = \varepsilon$ for some $\varepsilon \ge 0$. Then, for any perturbation $h \in C^1(\mathcal{M}, \mathbb{R}^d)$, the second variation of the GME cost satisfies
    \[
        -16 \sqrt{\varepsilon \rho_{\max}}\,\|h\|_{L^4(\mathcal M)}^2 
        \le \delta^2 \GME(T)(h,h) 
        \le 8\rho_{\max}(4\log(\beta)+1)\,\|h\|_{L^2(\mathcal M)}^2.
    \]
\end{theorem}

\begin{proof}
    For simplicity throughout the proof, we use the notation $f_{x,x'} := f(x) - f(x')$ for any function $f:\mathcal{M}\rightarrow {\mathbb{R}^d}$.
    From Proposition~\ref{prop:second-var-gme} in the appendix, the second variation of the GME cost for $T, h \in C^1(\mathcal{M}, \mathbb{R}^d)$ implies
    \begin{align}
        \frac{1}{4} \delta^2 \GME(T)(h,h) 
    &= \int_{{\mathcal{M}^2}}  \log\left( \frac{\|T_{x,x'}\|^2 + 1}{\|x-x'\|^2 + 1}\right)\left( \frac{ \| h_{x,x'} \|^2 }{\|T_{x,x'}\|^2 + 1} - 2 H^2 \right)  
    + 2 H^2 d\mu d\mu \label{eq:gm_upperbound3} \\
    &\geq - \left(\GME(T) \right)^{1/2}\left(\int_{\mathcal{M}^2}\left( \frac{ \| h_{x,x'} \|^2 }{\|T_{x,x'}\|^2 + 1} - 2 H^2 \right)^2d\mu d\mu\right)^{1/2} \label{eq:gm_upperbound2}
    \end{align}
    where we used Hölder's inequality in the last line and $H = \frac{\langle T_{x,x'}, h_{x,x'} \rangle}{\|T_{x,x'}\|^2 + 1}$. 

    Using the algebraic bounds $0 \leq \|h_{x,x'}\|^2 / (\|T_{x,x'}\|^2+1) \leq \|h_{x,x'}\|^2$ and the Cauchy-Schwarz inequality, it follows that
    \begin{align}\label{eq:thm-h-t-ineq}
        \left| \frac{ \| h_{x,x'} \|^2 }{\|T_{x,x'}\|^2 + 1} - 2 \left(\frac{\langle T_{x,x'}, h_{x,x'} \rangle}{\|T_{x,x'}\|^2 + 1}\right)^2 \right| \leq \|h_{x,x'}\|^2.
    \end{align}
    
    To establish the lower bound, we use the property that $\mu$ is a probability measure on a compact manifold, meaning the $L^p$ norms of $h$ are bounded by its $C^1$ norm. We can bound $\int_{\mathcal{M}^2}\|h_{x,x'}\|^4 d\mu d\mu$ by
    \[
        \int_{\mathcal{M}^2}\|h_{x,x'}\|^4 d\mu d\mu
        \leq \int_{\mathcal{M}^2} 8(\|h(x)\|^4 + \|h(x')\|^4) d\mu d\mu
        \leq 16 \rho_{\max}\|h\|_{L^4(\mathcal{M})}^4.
    \]
    Substituting this into \eqref{eq:gm_upperbound2} and the premise that $\GME(T) = \varepsilon$, we bound the right-hand side below by $- 4  \sqrt{\varepsilon \rho_{\max}} \|h\|^2_{L^4(\mathcal M)}$, which after multiplying by $4$ yields the lower bound.
    
    Now, we establish the upper bound of the Hessian in the $L^2$ topology. Since $T$ is $\beta$-Lipschitz, the absolute bound for the logarithmic term is
    \begin{align}
        \left|\log\left(\frac{\|T_{x,x'}\|^2+1}{\|x-x'\|^2+1}\right)\right|
        \leq 2 \log(\beta)
        \label{eq:log-up-bound}
    \end{align}
    for all $(x,x')\in \mathcal{M}^2$.
    Returning to \eqref{eq:gm_upperbound3}, applying our bounds and using $H^2\leq \tfrac{\|T_{x,x'}\|^2\|h_{x,x'}\|^2}{(\|T_{x,x'}\|^2+1)^2} \leq \tfrac{1}{2}\|h_{x,x'}\|^2$, we restrict the upper bound purely to the $L^2$ topology
    \begin{align*}
        \frac{1}{4} \delta^2 \GME(T)(h,h) 
        &\leq 2\log(\beta) \int_{\mathcal{M}^2} \|h_{x,x'}\|^2 d\mu d\mu + \frac{1}{2} \int_{\mathcal{M}^2} \|h_{x,x'}\|^2 d\mu d\mu \\
        &= \left(2\log(\beta) + \frac{1}{2}\right) \int_{\mathcal{M}^2} \|h_{x,x'}\|^2 d\mu d\mu.
    \end{align*}
    Since $\int_{\mathcal{M}^2} \|h_{x,x'}\|^2 d\mu d\mu \leq 4 \|h\|^2_{L^2(\mu)}\leq 4 \rho_{\max} \|h\|_{L^2(\mathcal{M})}^2$, this implies
    \[
        \frac{1}{4} \delta^2 \GME(T)(h,h) \leq 2 \rho_{\max} (4\log(\beta) + 1) \|h\|_{L^2(\mathcal{M})}^2.
    \]
    Multiplying by $4$ yields the final $L^2$ upper bound, completely validating the $L^2$ gradient descent convergence step.
\end{proof}
\Cref{thm:log-GME-bound} establishes that the logarithmic formulation of the GME cost tightly controls the Hessian's growth, limiting it to $\mathcal{O}(\log \beta)$, where $\beta$ represents the geometric distortion arising from the neural network's functional structure, which can be large. In contrast, standard distance-preserving objectives (such as the quadratic MDS cost) exhibit polynomial growth, $\mathcal{O}(\beta^2)$. Because the maximum stable step size is inversely proportional to this bound (\Cref{prop:hessian-chain}), standard models with large initial distortions are forced to use vanishingly small step sizes ($\propto 1/\beta^2$) and frequently stall. GME resolves this: its required step size decays extremely slowly ($\propto 1/\log \beta$), permitting aggressive learning rates that reliably navigate highly distorted regions without stalling, an advantage validated empirically in \Cref{subsec:isometry_experiments}.

The following corollary (proof in \Cref{appendix:proof-gme-gd}) translates this geometric bound into a concrete optimization guarantee, demonstrating that the adaptive logarithmic step size ensures strict monotonic descent.

\begin{corollary}[Logarithmic Step-Size Stability]\label{cor:gme-cost-gd-conv}
    Let $T^{(0)} \in C^2(\mathcal{M}, \mathbb{R}^d)$ be an initial mapping, and consider the gradient descent updates in $L^2(\mu)$:
    \[
        T^{(k+1)} = T^{(k)} - \sigma_k \nabla_{L^2(\mu)} \GME(T^{(k)}), \quad k \geq 0.
    \]
    Assume each iterate $T^{(k)}$ maintains finite regularity $\beta_k = \max(1, \|T^{(k)}\|_{C^2}) < \infty$. By setting the adaptive step size $\sigma_k = \frac{1}{8(4\log(\beta_k) + 1)}$, the objective value strictly monotonically decreases. Furthermore, the optimization trajectory satisfies the finite-energy condition:
    \[
        \sum_{k=0}^{\infty} \frac{\|\nabla_{L^2(\mu)} \GME(T^{(k)})\|_{L^2(\mathcal{M})}^2}{8(4\log(\beta_k) + 1)} \leq 2 \left( \GME(T^{(0)}) - \inf_{T} \GME(T) \right).
    \]
\end{corollary}

This finite-energy condition proves that the function-space gradients must asymptotically vanish. Because the logarithmic denominator prevents the learning rate from artificially collapsing, the optimization process guarantees convergence without prematurely stalling in suboptimal states.


\subsection{Efficiency of decoder training given a GPE encoder}\label{sec:effi-decoder}

In this section, we show that the decoder training becomes more efficient and can be accelerated given the encoder with the small GME loss from \cref{eq:gme-cost}. Consider the following reconstruction cost function given an encoder $T$
\begin{equation}\label{eq:L-recons}
    L_{\rm{rec}}(S) := \mathbb{E}_{x \sim \mu}  [\|S \circ T(x)- x\|^2]
\end{equation}
Using the local isometry result from \Cref{ssec:local_isometry}, we can derive lower and upper bounds on the Hessian of the reconstruction loss, given an encoder $ T $ that minimizes the GME cost.

\begin{theorem}\label{thm:S-2}
Let $\mathcal{M} \subset \mathbb{R}^D$ be a compact, smooth $m$-dimensional Riemannian manifold. Let $\mu \in \mathbb{P}(\mathcal{M})$ be a data distribution with density $\rho$ bounded strictly by $0 < \rho_{\min} \leq \rho(x) \leq \rho_{\max}$ relative to the Riemannian volume measure $\mathrm{vol}_{\mathcal{M}}$. Suppose $T \in \mathcal{T}_\beta^k$ is a smooth embedding of $\mathcal{M}$ into $\mathbb{R}^d$, and let $\varepsilon_{\mathrm{GME}} := \GME(T)$. Consider the reconstruction cost $L_{\mathrm{rec}}(S)$ defined in \eqref{eq:L-recons}. Then, the second variation of the cost function satisfies, for any perturbation $h$,
\begin{equation}\label{eq:hessian-bounds}
\begin{aligned}
    2 \rho_{\min} \exp\left(- K \varepsilon_{\mathrm{GME}}^{c_{m,k}} \right) 
    &\leq \frac{\delta^2 L_{\mathrm{rec}}(S)(h, h)}{\|h\|^2_{L^2(T(\M))}} 
    \leq 2 \rho_{\max} \exp\left(K \varepsilon_{\mathrm{GME}}^{c_{m,k}} \right) ,
\end{aligned}
\end{equation}
where $K > 0$ is a constant depending on $m$, the local geometry of $\mathcal{M}$, and the $C^2$ norm of $T$, the exponent is $c_{m,k} = \frac{k-2}{(m+2k)(m+1)}$ for an integer $k \ge 3$, and
\[
\|h\|^2_{L^2(T(\mathcal{M}))} = \int_{T(\mathcal{M})} \|h(y)\|^2 d\mathrm{vol}_{T(\mathcal{M})}(y).
\]
\end{theorem}

\begin{proof}
We begin by computing the second variation of the reconstruction cost $L_{\mathrm{rec}}(S)$ with respect to the decoder $S$. For a given perturbation $h: T(\mathcal{M}) \to \mathcal{M}$ and a scalar $\epsilon$, we evaluate $L_{\mathrm{rec}}(S + \epsilon h)$
\begin{align*}
    L_{\mathrm{rec}}(S + \epsilon h) &= \mathbb{E}_{x \sim \mu} \left[ \|(S + \epsilon h)(T(x)) - x\|^2 \right] \\
    &= \mathbb{E}_{x \sim \mu} \left[ \|S(T(x)) - x\|^2 + 2\epsilon \langle S(T(x)) - x, h(T(x)) \rangle + \epsilon^2 \|h(T(x))\|^2 \right].
\end{align*}
Taking the second derivative with respect to $\epsilon$ and evaluating at $\epsilon = 0$ yields the second variation
\begin{equation}\label{eq:second-variation-def}
    \delta^2 L_{\mathrm{rec}}(S)(h, h) = \left. \frac{d^2}{d\epsilon^2} L_{\mathrm{rec}}(S+\epsilon h) \right|_{\epsilon=0} = 2 \mathbb{E}_{x \sim \mu} \left[ \|h(T(x))\|^2 \right].
\end{equation}
Writing the expectation explicitly as an integral over the manifold $\mathcal{M}$, we have
\begin{equation}\label{eq:integral-M}
    \delta^2 L_{\mathrm{rec}}(S)(h, h) = 2 \int_{\mathcal{M}} \|h(T(x))\|^2 \rho(x) d\mathrm{vol}_{\mathcal{M}}(x).
\end{equation}
By assumption, $T$ is a smooth embedding. We apply the change of variables formula from \Cref{lem:change-of-var}. Letting $y = T(x)$, the relationship between the volume measures is governed by the Jacobian $J_T(x) = \sqrt{\det(dT_x^\top dT_x)}$, such that $d\mathrm{vol}_{T(\mathcal{M})}(y) = J_T(x) d\mathrm{vol}_{\mathcal{M}}(x)$. 

Substituting $d\mathrm{vol}_{\mathcal{M}}(x) = J_T(x)^{-1} d\mathrm{vol}_{T(\mathcal{M})}(y)$ and $x = T^{-1}(y)$ into \eqref{eq:integral-M}, the integral over the pushforward space $T(\mathcal{M})$ becomes
\begin{equation}\label{eq:integral-TM}
    \delta^2 L_{\mathrm{rec}}(S)(h, h) = 2 \int_{T(\mathcal{M})} \|h(y)\|^2 \frac{\rho(T^{-1}(y))}{J_T(T^{-1}(y))} d\mathrm{vol}_{T(\mathcal{M})}(y).
\end{equation}
Next, we bound the weighting term inside the integral. The data density is strictly bounded by $0 < \rho_{\min} \leq \rho(x) \leq \rho_{\max}$. Furthermore, \Cref{cor:A-bound} provides uniform global bounds on the Jacobian of the encoder in terms of the GME cost
\begin{equation*}
    \exp\left( - K \varepsilon_{\mathrm{GME}}^{c_{m,k}} \right) \leq J_T(x) \leq \exp\left( K \varepsilon_{\mathrm{GME}}^{c_{m,k}} \right).
\end{equation*}
Taking the reciprocal of the Jacobian bounds and multiplying by the density bounds, we obtain
\begin{equation}\label{eq:weight-bounds}
    \rho_{\min} \exp\left( - K \varepsilon_{\mathrm{GME}}^{c_{m,k}} \right) \leq \frac{\rho(T^{-1}(y))}{J_T(T^{-1}(y))} \leq \rho_{\max} \exp\left( K \varepsilon_{\mathrm{GME}}^{c_{m,k}} \right).
\end{equation}
Applying the lower and upper bounds from \eqref{eq:weight-bounds} directly to the integral in \eqref{eq:integral-TM} yields
\begin{equation*}
    2 \rho_{\min} \exp\left( - K \varepsilon_{\mathrm{GME}}^{c_{m,k}} \right) \int_{T(\mathcal{M})} \|h(y)\|^2 d\mathrm{vol}_{T(\mathcal{M})}(y) \leq \delta^2 L_{\mathrm{rec}}(S)(h, h),
\end{equation*}
and
\begin{equation*}
    \delta^2 L_{\mathrm{rec}}(S)(h, h) \leq 2 \rho_{\max} \exp\left( K \varepsilon_{\mathrm{GME}}^{c_{m,k}} \right) \int_{T(\mathcal{M})} \|h(y)\|^2 d\mathrm{vol}_{T(\mathcal{M})}(y).
\end{equation*}
Recognizing that $\int_{T(\mathcal{M})} \|h(y)\|^2 d\mathrm{vol}_{T(\mathcal{M})}(y) = \|h\|^2_{L^2(T(\mathcal{M}))}$, we recover the desired inequalities in \eqref{eq:hessian-bounds}. This completes the proof.
\end{proof}

Theorem \ref{thm:S-2} provides a crucial theoretical link between the geometric distortion of the latent space and the optimization dynamics of the decoder. Specifically, the theorem establishes that the eigenvalues of the Hessian of the reconstruction loss are bounded by an exponential function of the GME error, $\varepsilon_{\mathrm{GME}}$. As the encoder successfully minimizes the GME cost ($\varepsilon_{\mathrm{GME}} \to 0$), the exponential terms approach $1$, and the condition number of the Hessian tightly converges to the inherent data density ratio $\rho_{\max} / \rho_{\min}$. 

Because a well-conditioned Hessian directly translates to faster and more stable convergence for gradient-based optimization methods, this result implies that minimizing the GME cost inherently accelerates the training of the decoder. By preserving the intrinsic geometry of the manifold, the encoder sculpts a latent landscape where the decoder's objective function is highly regularized and easy to navigate. 

Furthermore, while this result is framed around our specific GME formulation, the underlying mechanism holds for any general pairwise metric-matching objective with sufficient regularity of the form $\mathbb{E}[(c(x,x') - c(T(x),T(x')))^2]$. As discussed in the previous section, our specific choice to apply a logarithmic transformation to the distances in the GME objective is not strictly necessary for this theoretical guarantee to hold; rather, it was introduced primarily to guarantee numerical stability during empirical training.

\section{Efficiency of GPE in LGM}\label{sec:gpe-lgm}
In this section, we provide theoretical error bounds to demonstrate how the GPE framework guarantees high-quality generation of the input data distribution. A central insight of our analysis is that minimizing the GME cost explicitly regularizes the geometric distortion of the encoder. Specifically, a smaller GME cost implies that the embedding $T$ satisfies a practical, probabilistic relaxation of the bi-Lipschitz property, which we refer to as being \textit{weak $\alpha$-bi-Lipschitz} (with $\alpha \ge 1$ close to $1$). Because the quality of the final generated samples relies heavily on the regularity of the latent space, controlling this geometric distortion $\alpha$ is paramount.

To formalize this, let us briefly review the mechanics of the Latent Generative Model (LGM) framework (see \Cref{fig:desc-t-s-r} for a visual summary). Assume we have an input data distribution $\mu \in \mathbb{P}(\mathcal{M})$ supported on a manifold $\mathcal{M} \subset \mathbb{R}^D$, and a simple prior latent distribution $\nu \in \mathbb{P}(\mathbb{R}^d)$. The generation process consists of three core components
\begin{enumerate}
    \item \textbf{Encoding:} An encoder $T: \mathcal{M} \to \mathbb{R}^d$ embeds the data distribution into the latent space, producing the pushforward distribution $T_\#\mu$.
    \item \textbf{Latent Flow:} A generative flow map $R: \mathbb{R}^d \to \mathbb{R}^d$ is trained to map samples from the prior $\nu$ to the embedded distribution, such that $R_\#\nu \approx T_\#\mu$. In modern architectures, $R$ is typically parameterized via diffusion models, such as score-based generative models \cite{songscore} or conditional flow matching \cite{lipman2022flow}.
    \item \textbf{Decoding:} A decoder $S: T(\mathcal{M}) \to \mathbb{R}^D$, trained to approximate the inverse map $T^{-1}$, projects the generated latent samples back to the original data manifold.
\end{enumerate}
Ultimately, the overall generative pipeline is defined by the composite map $G = S \circ R$. The discrepancy between the true distribution $\mu$ and the generated distribution $G_\#\nu$ is inextricably linked to the geometric properties of $T$ and $S$. 

\begin{figure}[!ht]
    \centering
\tikzset{every picture/.style={line width=0.75pt}} 
\begin{tikzpicture}[x=0.75pt,y=0.75pt,yscale=-0.85,xscale=0.85]
\draw   (78.58,65.8) .. controls (74.93,74.71) and (90.93,81.93) .. (114.31,81.93) .. controls (137.69,81.93) and (159.6,74.71) .. (163.25,65.8) .. controls (166.9,56.89) and (188.82,49.67) .. (212.2,49.67) .. controls (235.58,49.67) and (251.57,56.89) .. (247.92,65.8) -- (195.04,194.87) .. controls (198.69,185.96) and (182.7,178.73) .. (159.31,178.73) .. controls (135.93,178.73) and (114.02,185.96) .. (110.37,194.87) .. controls (106.72,203.78) and (84.81,211) .. (61.43,211) .. controls (38.05,211) and (22.05,203.78) .. (25.7,194.87) -- cycle ;
\draw   (158.97,92.67) -- (160,107.85) -- (186.99,110.29) -- (156.82,122.11) -- (148.69,138.8) -- (129.01,130.92) -- (96.99,138.8) -- (115,122.11) -- (103.35,110.29) -- (134.16,107.85) -- cycle ;
\draw    (248.67,123.67) -- (350,123.67) ;
\draw [shift={(350,123.67)}, rotate = 180] [fill={rgb, 255:red, 0; green, 0; blue, 0 }  ][line width=0.08]  [draw opacity=0] (8.93,-4.29) -- (0,0) -- (8.93,4.29) -- cycle    ;
\draw    (248.33,150.33) -- (350,150.33) ;
\draw [shift={(245.33,150.33)}, rotate = 0] [fill={rgb, 255:red, 0; green, 0; blue, 0 }  ][line width=0.08]  [draw opacity=0] (8.93,-4.29) -- (0,0) -- (8.93,4.29) -- cycle    ;
\draw   (437.17,88.33) -- (449.66,114.53) -- (477.59,118.74) -- (457.38,139.13) -- (462.15,167.93) -- (437.17,154.33) -- (412.19,167.93) -- (416.96,139.13) -- (396.75,118.74) -- (424.68,114.53) -- cycle ;
\draw   (605,133.17) .. controls (605,110.06) and (623.73,91.33) .. (646.83,91.33) .. controls (669.94,91.33) and (688.67,110.06) .. (688.67,133.17) .. controls (688.67,156.27) and (669.94,175) .. (646.83,175) .. controls (623.73,175) and (605,156.27) .. (605,133.17) -- cycle ;
\draw    (483.67,133.67) -- (595.67,133.67) ;
\draw [shift={(480.67,133.67)}, rotate = 0] [fill={rgb, 255:red, 0; green, 0; blue, 0 }  ][line width=0.08]  [draw opacity=0] (8.93,-4.29) -- (0,0) -- (8.93,4.29) -- cycle    ;
\draw   (375,44.67) -- (695.67,44.67) -- (695.67,219.67) -- (375,219.67) -- cycle ;
\draw (260,95) node [anchor=north west][inner sep=0.75pt]   [align=left] {Encoder $T$};
\draw (260,158) node [anchor=north west][inner sep=0.75pt]   [align=left] {Decoder $S$};
\draw (494,106.33) node [anchor=north west][inner sep=0.75pt]   [align=left] {Flow map $R$};
\draw (108.67,223) node [anchor=north west][inner sep=0.75pt]   [align=left] {$\mathcal{M}\subset \mathbb{R}^D$};
\draw (85.33,150) node [anchor=north west][inner sep=0.75pt]   [align=left] {$\mu \in \mathbb{P}(\mathcal{M}) $};
\draw (390,180) node [anchor=north west][inner sep=0.75pt]   [align=left] {$T_\#\mu \in \mathbb{P}(T(\mathcal{M}))$};
\draw (610,180) node [anchor=north west][inner sep=0.75pt]   [align=left] {$\nu \in\mathbb{P}(\mathbb{R}^d) $};
\draw (520,223) node [anchor=north west][inner sep=0.75pt]   [align=left] {$\mathbb{R}^d$};
\end{tikzpicture}
\caption{Visualization of the latent generative model framework.}
\label[figure]{fig:desc-t-s-r}
\end{figure}

To guarantee that expected reconstruction errors over the generated latent samples remain well-behaved, we require a necessary regularity condition on the latent flow map $R$. 

\begin{assumption}[Good Training]\label{assump:good-training}
    We assume the latent generative flow map $R$ is sufficiently well-trained such that the generated density $R_\#\nu$ remains proportionally bounded relative to the maximum density of the embedded data. Specifically, there exists a constant $c \geq 1$ such that $\sup_{y \in T(\mathcal{M})} \rho_R(y) \leq c \rho_{\max} \beta^m$, where $\beta$ bounds the Jacobian of $T$.
\end{assumption}
This assumption acts as a critical safeguard, ensuring the generative model does not pathologically concentrate mass in highly distorted or poorly reconstructed regions of the latent space. With this condition met, the total generation error of the pipeline is governed by two primary sources of inaccuracy: the population reconstruction error $\epsilon_{\mathrm{rec}}$ (how accurately the decoder $S$ reconstructs the original data) and the latent generative error $\epsilon_{\mathrm{dif}}$ (how well the flow map $R$ matches the embedded data).

Conceptually, if the encoder $T$ were \textit{strictly} $\alpha$-bi-Lipschitz everywhere, triangle inequalities would bound the total generation error by $\mathcal{O}(\alpha \epsilon_{\mathrm{dif}} + \epsilon_{\mathrm{rec}})$. In this ideal scenario, the geometric distortion factor $\alpha$ acts as a direct multiplier on the latent error. When $\alpha \to 1$, the encoder perfectly preserves distances, meaning any small mistakes made by the latent flow model $R$ remain small when mapped back to the data space. Conversely, a large $\alpha$ means the decoder violently expands certain regions, drastically amplifying latent inaccuracies.

While this highlights the necessity of learning a nearly isometric embedding, requiring a neural network to be strictly $\alpha$-bi-Lipschitz everywhere is often too rigid in practice. Instead, by minimizing the GME cost, the GPE encoder is forced to be \textit{weakly} $\alpha$-bi-Lipschitz, preserving geometry tightly over a set of high probability rather than strictly everywhere. \Cref{thm:generative-map3} formalizes this intuition, demonstrating that we can bound the true generation error using the empirical GME cost directly. Crucially, to separate the well-known challenges of deep learning generalization from our geometric claims, we state our theorem directly in terms of the population reconstruction loss $\epsilon_{\mathrm{rec}}$ of the trained decoder $S$.
\begin{theorem}[Error Bound of GPE-based LGM]\label{thm:generative-map3}
    Fix $\epsilon > 0$ and $\alpha, \beta \geq 1$ such that $\beta > \alpha$. Let $\mu \in \mathbb{P}(\mathcal{M})$ be the data distribution, and let $\hat{\mu}_n = \frac{1}{n} \sum_{i=1}^n \delta_{x_i}$ be an empirical distribution of $n$ independent samples. Let $T \in \mathcal{T}_\beta^k$ be the encoder with discrete GME cost $\epsilon_{\GME} = \widehat{\GME}(T, \hat{\mu}_n)$ (see \Cref{def:discrete-gme}). 
    Suppose the trained generative flow map $R$ satisfies \Cref{assump:good-training}, yielding the latent generative approximation error $\epsilon_{\mathrm{dif}} = W_1(T_\#\hat\mu_n, R_\#\nu)$. Furthermore, assume the trained decoder $S$ achieves a population reconstruction loss $\epsilon_{\mathrm{rec}} = \mathbb{E}_{x\sim\mu} \|S(T(x)) - x\|^2$. 
    
    Then, there exist positive constants $C_1$ (depending on $\mathrm{diam}(\mathcal{M}), \rho_{\max},$ and $\rho_{\min}$) and $C_2$ (depending on $\mu$ and $m$) such that with probability at least $1 - 2 \exp\Bigl(- \frac{n\epsilon^2}{2(1 + \epsilon/3)}\Bigr)$, the total generation error is bounded by:
    \begin{equation*}
        W_1(\mu, (S\circ R)_\#\nu) \leq C_1\beta^m \left(\frac{\epsilon_{\GME} + (\log\beta)^2\epsilon}{(\log\alpha)^2} + \epsilon_{\mathrm{rec}} \right) + \alpha \Bigl(\epsilon_{\mathrm{dif}} + C_2 \beta n^{-1/m}\Bigr) + 2(\alpha^2-1)^{1/2}.
    \end{equation*}
\end{theorem}

\begin{remark}[Bounding Latent Error and the Curse of Dimensionality]\label{rem:flow-matching-bound}
    The latent error $\epsilon_{\mathrm{dif}} = W_1(T_\#\hat\mu_n, R_\#\nu)$ measures the discrepancy introduced exclusively by the generative flow map $R$. If $R$ is parameterized as a continuous normalizing flow trained via flow matching, this distance is bounded by $\mathcal{O}(e^{\hat{K}} \sqrt{H(v)})$, where $H(v)$ is the velocity estimation error and $\hat{K}$ is the Lipschitz constant \cite[Proposition~3]{albergo2022building}. Crucially, estimating $H(v)$ directly in the ambient space $\mathbb{R}^D$ incurs a severe sample complexity penalty. By restricting the flow model to a compact latent space $\mathbb{R}^d$ ($d \ll D$), our framework avoids this statistical curse of dimensionality, ensuring $\epsilon_{\mathrm{dif}}$ remains small.
\end{remark}

\subsection{Proof of \Cref{thm:generative-map3}}\label{sec:proof-cor}

We first bound the decoder's error when mapping the generated latent distribution back to the data manifold, utilizing the population reconstruction loss.

\begin{lemma}\label{lem:rec-bound}
    Suppose $\mu \in \mathbb{P}(\mathcal{M})$ satisfies \Cref{assump:good-training}, $T \in \mathcal{T}_\beta^k$ is a diffeomorphism onto its image, and $\xi \in \mathbb{P}(T(\mathcal{M}))$ has a density bounded by $\xi_{\max}$. Then the expected reconstruction error over the latent distribution satisfies:
    \begin{equation*}
        \|S - T^{-1}\|_{L^1(\xi)} = \int_{T(\mathcal{M})} \|S(y) - T^{-1}(y)\| \, d\xi(y) \leq \frac{\xi_{\max}\beta^m}{\rho_{\min}} \epsilon_{\mathrm{rec}}.
    \end{equation*}
\end{lemma}
\begin{proof}
    By the change of variables formula, $dT^{-1}_\#\xi(x) = f(T(x)) J_T(x) \, d\mathrm{vol}_\mathcal{M}(x)$. Since $T \in \mathcal{T}_\beta^k$, its Jacobian determinant is bounded by $\beta^m$. Using this and $\rho(x) \geq \rho_{\min}$, we bound the integral:
    \begin{align*}
        \int_{T(\mathcal{M})} \|S(y) - T^{-1}(y)\| \, d\xi(y)
        &= \int_{\mathcal{M}} \|S(T(x)) - x\| f(T(x)) J_T(x) \, d\mathrm{vol}_{\mathcal{M}}(x)\\
        &\leq \xi_{\max}\beta^m \int_{\mathcal{M}} \|S(T(x)) - x\| \frac{\rho(x)}{\rho(x)} \, d\mathrm{vol}_\mathcal{M}(x) \leq \frac{\xi_{\max}\beta^m}{\rho_{\min}} \epsilon_{\mathrm{rec}}.
    \end{align*}
\end{proof}

We now formally define the discrete GME cost and its concentration properties.

\begin{definition}\label{def:discrete-gme}
    Given an empirical measure $\hat{\mu}_n = \frac{1}{n}\sum_{i=1}^n \delta_{x_i}$ sampled independently from $\mu$, the discrete GME cost is:
    \begin{equation*}
        \widehat{\GME}(T,\hat{\mu}_n) := \frac{1}{n(n-1)} \sum_{i=1}^n\sum_{j \neq i} \left(\log\left(\frac{1+\|T(x_i)-T(x_j)\|^2}{1+\|x_i-x_j\|^2}\right)\right)^2.
    \end{equation*}
\end{definition}

\begin{proposition}\label{prop:discrete-cont}
    Under the setup of \Cref{thm:generative-map3}, the event $|\GME(T) - \widehat{\GME}(T,\hat{\mu}_n)| < 4(\log \beta)^2\epsilon$ holds with probability at least $1 - 2 \exp\Bigl(- \frac{n\epsilon^2}{2(1 + \epsilon/3)}\Bigr)$.
\end{proposition}
\begin{proof}
    The discrete cost is a $U$-statistic with symmetric kernel $g(x,x') = \left(\log\frac{1+\|T(x)-T(x')\|^2}{1+\|x-x'\|^2}\right)^2$. Since $T \in \mathcal{T}_\beta^k$, the kernel is uniformly bounded by $b \leq 4(\log\beta)^2$, and its variance by $\sigma^2 \leq 16(\log \beta)^4$. Applying Bernstein's inequality for $U$-statistics \cite[Theorem~5.15]{calder2020calculus} yields the result.
\end{proof}

The following lemma bounds the latent mapping distance by separating the error into a well-behaved weak bi-Lipschitz region and a bounded penalty term for the highly distorted remainder. 

\begin{lemma}\label{lem:D-d-bound}
    Given the setup of \Cref{thm:generative-map3}, let $\eta \in \mathbb{P}(\mathcal{M})$ possess a density bounded by $\eta_{\max}$. Then, there exists a constant $C > 0$ depending only on $\mathrm{diam}({\mathcal{M}})$, $\rho_{\min}$, $\rho_{\max}$, and $\eta_{\max}$ such that with probability at least $1-2\exp\left(-\frac{n\epsilon^2}{2(1+\epsilon/3)}\right)$:
    \begin{equation*}
        W_1(\hat{\mu}_n, \eta) \leq \alpha W_1(T_\#\hat{\mu}_n, T_\#\eta) + (\alpha^2 - 1)^{1/2} + \frac{C(\epsilon_{\GME} + (\log\beta)^2\epsilon)}{(\log\alpha)^2}.
    \end{equation*}
\end{lemma}
\begin{proof}
    Let $g(x,y) = \left(\log\frac{1+\|T(x)-T(y)\|^2}{1+\|x-y\|^2}\right)^2$. By \Cref{prop:discrete-cont}, with high probability $\GME(T) \leq \epsilon_{\GME} + 4(\log\beta)^2\epsilon$. We define the "bad" set $B^c := \{ x \in \mathcal{M} : \int_{\mathcal{M}} g(x,y) d\mu(y) > (\log\alpha)^2 \}$. By Markov's inequality, $\mu(B^c) \leq \GME(T) / (\log\alpha)^2$. Since $\eta$ has a bounded density, $\eta(B^c) \leq \frac{\eta_{\max}}{\rho_{\min}} \mu(B^c)$. For any optimal transport plan $\gamma \in \Pi(\hat{\mu}_n, \eta)$, the mass of the joint bad set $(B \times B)^c$ is bounded by the marginals. 
    
    On the "good" set $A := B \times B$, the weak bi-Lipschitz condition implies $\|x-y\| \leq \alpha\|T(x)-T(y)\| + (\alpha^2 - 1)^{1/2}$. Splitting the transport integral:
    \begin{align*}
        \int_{\mathcal{M}^2} \|x-y\| \, d\gamma(x,y) &= \int_{A} \|x-y\| \, d\gamma(x,y) + \int_{\mathcal{M}^2 \setminus A} \|x-y\| \, d\gamma(x,y) \\
        &\leq \alpha \int_A \|T(x) - T(y)\| d\gamma(x,y) + (\alpha^2 - 1)^{1/2} + \mathrm{diam}(\mathcal{M})\gamma(A^c).
    \end{align*}
    Recognizing that $\gamma_T = (T \times T)_\# \gamma \in \Pi(T_\# \hat{\mu}_n, T_\# \eta)$, taking the infimum over $\gamma$ yields the desired bound, where the penalty term scales directly with $\mu(B^c)$.
\end{proof}

\begin{proof}[Proof of \Cref{thm:generative-map3}]
    By the definition of the Wasserstein metric, the composite mapping allows us to write:
    \begin{equation}\label{eq:thm-proof-00}
        W_1(\hat{\mu}_n, (S \circ R)_\#\nu) \leq W_1(\hat{\mu}_n, (T^{-1}\circ R)_\#\nu) + \int_{T(\mathcal{M})}\|T^{-1}(y) - S(y)\| \, dR_\#\nu(y).
    \end{equation}
    Let $\eta = (T^{-1}\circ R)_\#\nu$. By the change of variables, its density is bounded by $\eta_{\max} \leq \beta^m c \rho_{\max}$. Applying \Cref{lem:D-d-bound} to the first term, we obtain:
    \begin{equation}\label{eq:xy-T-bound-0}
        W_1(\hat{\mu}_n, \eta) \leq \alpha W_1(T_\#\hat{\mu}_n, R_\#\nu) + (\alpha^2-1)^{1/2} + \frac{C' (\epsilon_{\GME} + 4(\log \beta)^2 \epsilon)}{(\log \alpha)^2}.
    \end{equation}
    For the second term, applying \Cref{lem:rec-bound} with $\xi = R_\#\nu$ gives:
    \begin{equation}\label{eq:eps-rec-bound}
        \int_{\mathbb{R}^d} \|T^{-1}(y) - S(y)\| \, dR_\#\nu \leq \frac{c\rho_{\max}\beta^{2m}}{\rho_{\min}}\epsilon_{\mathrm{rec}}.
    \end{equation}
    Combining \eqref{eq:thm-proof-00}, \eqref{eq:xy-T-bound-0}, and \eqref{eq:eps-rec-bound}, and recalling $\epsilon_{\mathrm{dif}} = W_1(T_\#\hat{\mu}_n, R_\#\nu)$, we have:
    \begin{equation}\label{eq:w1-discrete-nu}
        W_1(\hat{\mu}_n, (S \circ R)_\#\nu) \leq \alpha \epsilon_{\mathrm{dif}} + (\alpha^2 - 1)^{1/2} + C_1 \beta^m\left(\frac{\epsilon_{\GME} + (\log\beta)^2\epsilon}{(\log\alpha)^2} + \epsilon_{\mathrm{rec}}\right).
    \end{equation}
    Finally, by the triangle inequality, $W_1(\mu, (S\circ R)_\#\nu) \leq W_1(\hat{\mu}_n, (S\circ R)_\#\nu) + W_1(\mu, \hat{\mu}_n)$. We bound the empirical approximation error $W_1(\mu, \hat{\mu}_n)$ by applying \Cref{lem:D-d-bound} to map the distance into the latent space. Using empirical bounds for distributions on $m$-dimensional manifolds \cite{weed2019sharp}, the latent distance satisfies $W_1(T_\#\mu, T_\#\hat{\mu}_n) \leq C_2 \beta n^{-1/m}$. Summing these terms yields the final bound, completing the proof.
\end{proof}

\section{Algorithm}\label{sec:alg}
%

In this section, we outline the training procedures for the GPE framework. The process consists of three primary components: the geometry-preserving encoder $T$, the decoder $S$, and the latent flow-based model $R$. 

    \textbf{Parallel Training:} Once the encoder $T$ is fixed and the latent codes $z_i = T(x_i)$ are computed, the training of the decoder $S$ and the flow model $R$ are independent. This allows them to be trained in parallel, reducing the overall wall-clock training time.

\subsection{Algorithm for the GPE Encoder}

The first step is to compute the geometry-preserving encoder $T:\mathcal{M}\rightarrow\mathbb{R}^d$ by minimizing the discrete GME cost. We parameterize $T$ as a neural network $T_{w_T}$. This step ensures that the intrinsic geometry of the data manifold $\mathcal{M}$ is preserved in the latent space $\mathbb{R}^d$. The algorithm can be found in \Cref{alg:T}.

\begin{algorithm}[h]
    \caption{Computing the GPE Encoder}
    \label{alg:T}
 \begin{algorithmic}
    \STATE {\bfseries Input:} Dataset $\{x_i\}_{i=1}^n \subset \mathcal{M}$, latent dimension $d$, tolerance $\TOL$, learning rate $\eta_T$.
    \STATE {\bfseries Output:} Trained GPE encoder $T_{w_T}$.
    \WHILE{$\widehatGME(T_{w_T}, \{x_i\}) > \TOL$}
    \STATE ${w_T} \gets {w_T} - \eta_T \nabla_{w_T} \widehatGME(T_{w_T}, \{x_i\})$.
    \ENDWHILE
 \end{algorithmic}
\end{algorithm}

\subsection{Algorithm for the GPE Decoder}

The decoder $S:T(\mathcal{M})\to \R^D$ is trained to act as the inverse of the encoder ($S \approx T^{-1}$). We parameterize the decoder as $S_{w_S}$ and minimize a standard reconstruction loss. Because $T$ is geometry-preserving, the reconstruction task is well-conditioned. The algorithm can be found in \Cref{alg:S}.

\begin{algorithm}[th!]
    \caption{Computing the GPE Decoder}
    \label{alg:S}
 \begin{algorithmic}
    \STATE {\bfseries Input:} Dataset $\{x_i\}_{i=1}^n \subset \mathcal{M}$, trained encoder $T_{w_T}$, tolerance $\TOL$, learning rate $\eta_S$.
    \STATE {\bfseries Output:} GPE decoder $S_{w_S}$.
    \WHILE{$L_{\rm{rec}}(S_{w_S}) > {\TOL}$}
    \STATE $w_S \gets w_S - \eta_S \nabla_{w_S} \frac{1}{n} \sum_{i=1}^n \|x_i - S_{w_S}(T_{w_T}(x_i))\|^2$.
    \ENDWHILE
 \end{algorithmic}
 \end{algorithm}

As per \Cref{thm:S-2}, the optimal step size $\eta_S$ for the decoder depends on the distortion $\alpha$. A small GME cost results in lower distortion, allowing for a larger learning rate and faster convergence during the parallel training of the decoder.

\subsection{Training the Flow-Based Model with Latent Centering and Scaling}
To generate new samples, we train a flow-based model $R$ to map a standard normal prior $\nu = \mathcal{N}(0, I_d)$ to the latent distribution $T_\#\mu$. In practice, the empirical mean and scale of the latent embeddings $z_i = T(x_i) \in \mathbb{R}^d$ often deviate from the zero mean and unit variance of the Gaussian prior. To stabilize the training of the flow model, we first center the latent distribution and then introduce a global scaling constant $C > 0$ to align its empirical variation with that of the standard Gaussian.

We define the empirical mean of the latent codes as $\bar{z} = \frac{1}{n} \sum_{i=1}^n z_i.$
Let $z'_i = z_i - \bar{z}$ denote the centered latent codes. We seek a constant $C$ such that the total variance of the scaled codes, $\tilde{z}_i = z'_i / C$, matches the theoretical total variance of a random variable $\xi \sim \mathcal{N}(0, I_d)$, which is exactly $d$
\begin{equation*}\label{eq:variance-matching}
    \frac{1}{n} \sum_{i=1}^n \|\tilde{z}_i\|^2 = \mathbb{E}_{\xi \sim \mathcal{N}(0, I_d)} \big[\|\xi\|^2\big] = d.
\end{equation*}
Substituting $\tilde{z}_i = (z_i - \bar{z}) / C$ and solving for $C$ yields
\begin{equation}\label{eq:compute-C}
    C = \sqrt{\frac{1}{nd} \sum_{i=1}^n \|z_i - \bar{z}\|^2}.
\end{equation}
The flow model $R$ is trained to map the prior $\nu$ to the distribution of the normalized codes $\{\tilde{z}_i\}$. During inference, the output of the flow model must be re-scaled by $C$ and shifted by $\bar{z}$ to restore the original latent geometry. Thus, for a novel sample $\xi \sim \nu$, the generated data point is $x_{\rm{gen}} = S_{w_S}\big( \bar{z} + C \cdot R(\xi) \big).$

\begin{algorithm}[h]
    \caption{Training the Latent Flow and Generation}
    \label{alg:R}
 \begin{algorithmic}
    \STATE {\bfseries Input:} Latent codes $\{z_i\}_{i=1}^n \subset \mathbb{R}^d$, prior $\nu = \mathcal{N}(0, I)$.
    \STATE {\bfseries 1. Pre-processing (Centering and Scaling):} 
    \STATE Compute empirical mean: $\bar{z} = \frac{1}{n} \sum_{i=1}^n z_i$.
    \STATE Compute scaling factor: $C = \sqrt{\frac{1}{nd} \sum_{i=1}^n \|z_i - \bar{z}\|^2}$.
    \STATE Set $\tilde{z}_i = (z_i - \bar{z}) / C$ for all $i=1, \dots, n$.
    \STATE {\bfseries 2. Training:} Train flow map $R$ such that $R_\# \nu \approx \text{Law}(\tilde{z})$.
    \STATE {\bfseries 3. Generation:} 
    \STATE Sample $\xi \sim \nu$.
    \STATE Compute $x_{\rm{gen}} = S_{w_S}(\bar{z} + C \cdot R(\xi))$.
 \end{algorithmic}
\end{algorithm}

\section{Numerical results}\label{sec:numeric}

In this section, we present numerical experiments demonstrating the effectiveness of our proposed method across various datasets, highlighting its efficiency in both reconstruction and generation tasks compared to other methods.

\subsection{Visualization of encoder differences in GPE and VAE}

In this experiment, our goal is to visualize the differences between the encoders from GPE and VAE. The primary objective of the GPE encoder is to embed the data distribution while preserving its local geometric structure, whereas the VAE encoder aims to map the data distribution to a standard Gaussian prior in the latent space. 

To demonstrate this, we use an artificial dataset in $\mathbb{R}^{500}$ generated from a mixture of $K$ Gaussians (where $K = 4, 16$). We explicitly define our underlying data manifold $\mathcal{M}$ as the two-dimensional subspace spanned by the first two coordinate axes of the ambient space, i.e., $\mathcal{M} = \{(x_1, x_2, 0, \dots, 0) \in \mathbb{R}^{500}\}$. 

The centers of these Gaussians are placed exclusively on $\mathcal{M}$. Specifically, the mean of the $k$-th Gaussian is defined as $m_k = (\bar{x}_{k,1}, \bar{x}_{k,2}, 0, \dots, 0) \in \mathbb{R}^{500}$, where the 2D coordinates $(\bar{x}_{k,1}, \bar{x}_{k,2})$ are arranged in a $2 \times 2$ and $4 \times 4$ grid with coordinates in $\{0, 2\}^2$ and $\{-2, 0, 2, 4\}^2$. The covariance matrix $\Sigma$ for each Gaussian is identical and diagonal, with the first two diagonal entries being $0.15^2$ (representing variance along the manifold $\mathcal{M}$) and the remaining 498 entries set to $10^{-4}$ (representing small ambient noise orthogonal to $\mathcal{M}$). We uniformly sample $N=10,000$ points from this mixture model to train both GPE and VAE encoders using a 4-layer fully-connected neural network.

The results are displayed in \Cref{fig:toy-comparison}. The first column shows the input data distribution $\mu$ in $\mathbb{R}^{500}$ by plotting the first two coordinates (the projection onto $\mathcal{M}$), while the second and third columns show the embedded distributions $T_\#\mu$ in $\mathbb{R}^2$ from the trained GPE and VAE encoders,

\begin{figure}[h!]
    \centering
    
    \begin{subfigure}[b]{0.25\textwidth}
        \centering
        \includegraphics[height=2.9cm]{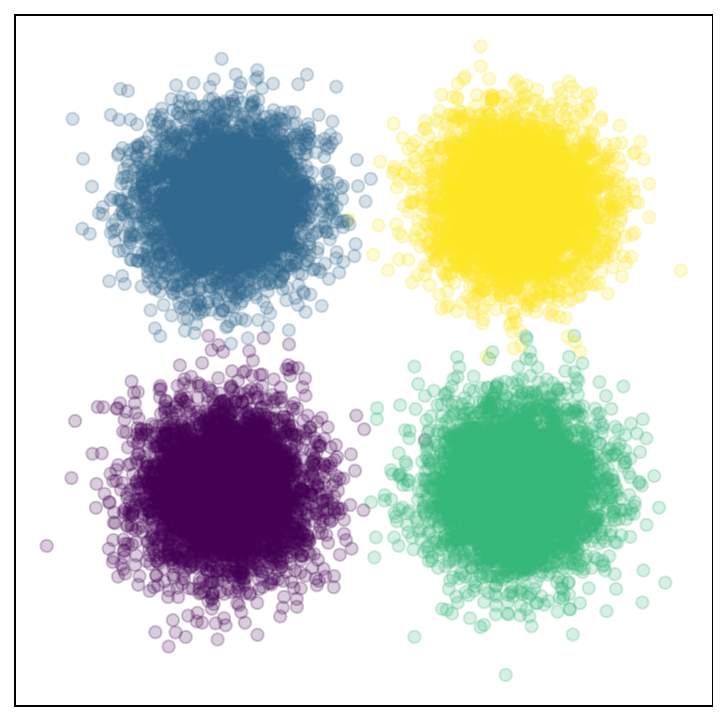}
        \caption{$\mu$ in 500D}
    \end{subfigure}
    \begin{subfigure}[b]{0.25\textwidth}
        \centering
        \includegraphics[height=2.9cm]{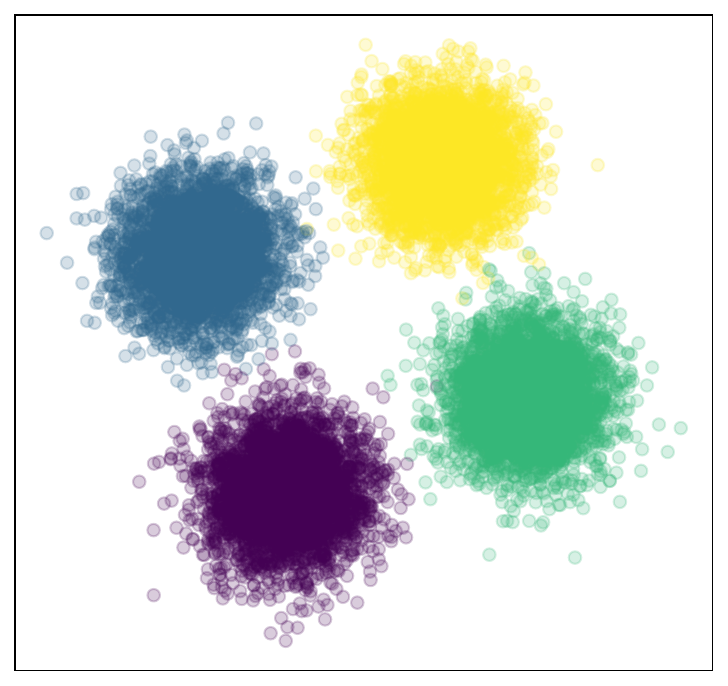}
        \caption{$T_\#\mu$ from GPE}
    \end{subfigure}
    \begin{subfigure}[b]{0.25\textwidth}
        \centering
        \includegraphics[height=2.9cm]{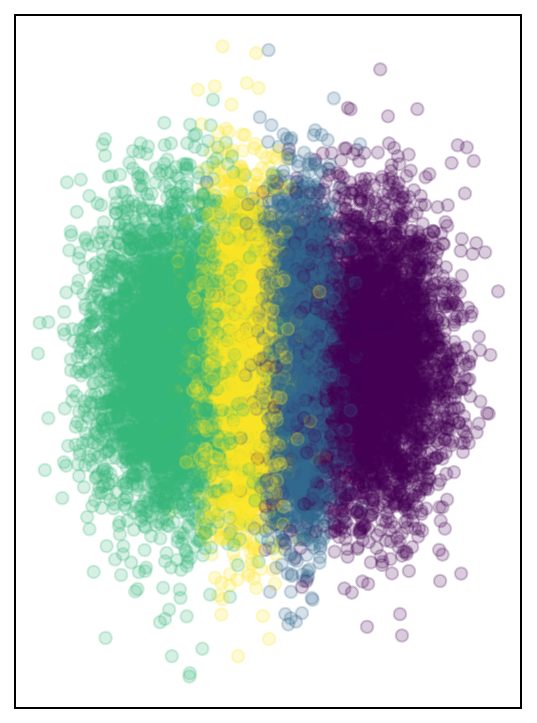}
        \caption{$T_\#\mu$ from VAE}
    \end{subfigure}
    
    \begin{subfigure}[b]{0.25\textwidth}
        \centering
        \includegraphics[height=2.9cm]{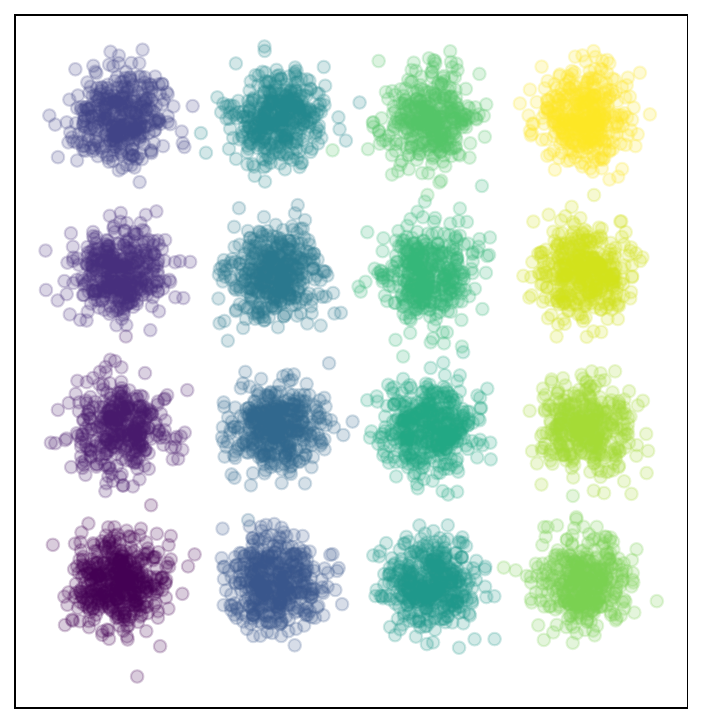}
        \caption{$\mu$ in 500D}
    \end{subfigure}
    \begin{subfigure}[b]{0.25\textwidth}
        \centering
        \includegraphics[height=2.9cm]{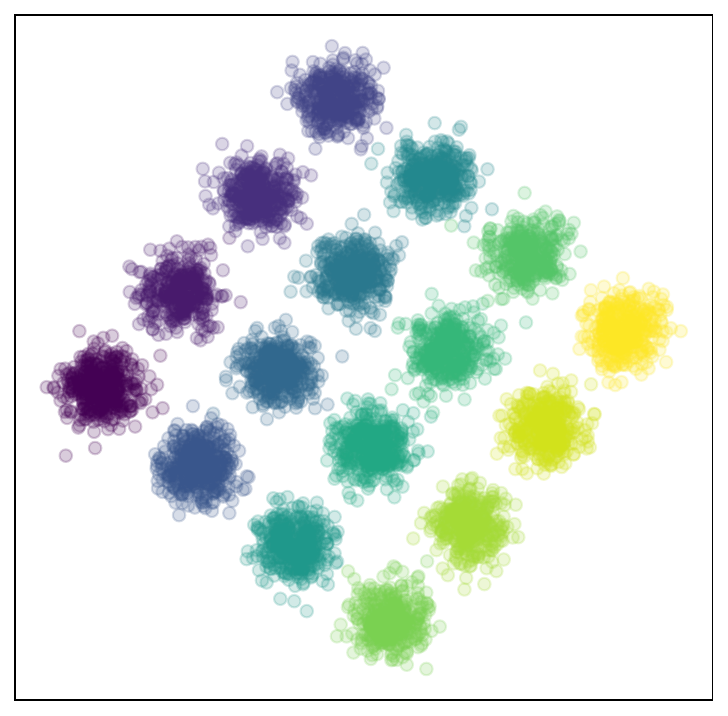}
        \caption{$T_\#\mu$ from GPE}
    \end{subfigure}
    \begin{subfigure}[b]{0.25\textwidth}
        \centering
        \includegraphics[height=2.9cm]{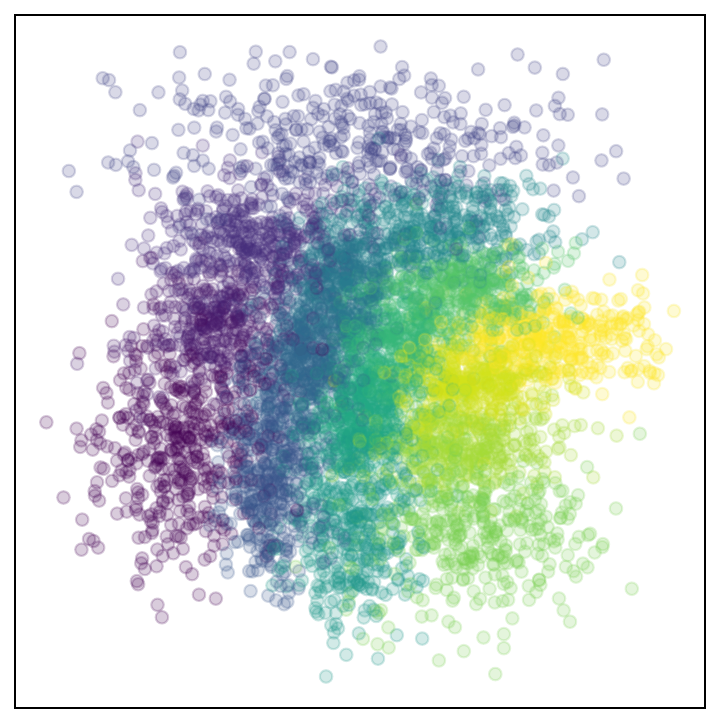}
        \caption{$T_\#\mu$ from VAE}
    \end{subfigure}
  
    \caption{
Comparison of embedded data distributions from GPE and VAE. The first column shows the input data distribution $\mu$ in 500D, while the second and third columns show the embedded distributions $T_\#\mu$ from GPE and VAE, respectively, in 2D.}
    \label{fig:toy-comparison}
\end{figure}

\subsection{Evaluating Local and Global Isometry in Latent Representations}\label{subsec:isometry_experiments}

To establish a well-conditioned latent space, an encoder must balance global and local geometric structures. Why employ a logarithmic cost function, like our proposed GME loss in \eqref{eq:gme-cost}, rather than a quadratic penalty like standard MDS? As shown analytically, the logarithmic formulation yields a tighter Hessian. This bounded curvature stabilizes the optimization landscape, preventing the severe gradient explosions common in unscaled quadratic methods. The following experiments evaluate these practical implications, specifically investigating how different cost functions impact overall stability, the trade-off between preserving global versus local distances, and the direct validation of our theoretical local distortion bounds.

We evaluate on the CIFAR-10 dataset using an encoder architecture consisting of four linear layers with Softplus activations. This network maps ambient inputs ($D=3072$) to a fixed latent space dimension of $d = 25$. To rigorously test the optimization stability of each cost function, we train the models across three varying learning rates: $1.2 \times 10^{-3}$, $1.2 \times 10^{-2}$, and $1.2 \times 10^{-1}$. We compare the GME loss against two baseline embedding methods
\begin{enumerate}
    \item \textbf{MDS1 (Vanilla MDS):} $\mathbb{E}[(\|T(x)-T(x')\|^2 - \|x-x'\|^2)^2]$.
    \item \textbf{MDS2 (Scaled MDS):} Incorporates dimension-based scaling to stabilize gradients and handle large distances: $\mathbb{E}[(\|T(x)-T(x')\|^2/d - \|x-x'\|^2/D)^2]$.
\end{enumerate}

To quantify the geometric distortion introduced by the encoder $T$, we track two primary metrics evaluated on a fixed set of test data points at the end of each epoch. First, we measure {global isometry ($\alpha_{95}$)} to capture worst-case global scaling. By computing the 95th percentile of the bi-Lipschitz bounds across the evaluated points, this metric measures the maximum distortion between arbitrary pairs of images. Second, we assess {local distortion} by numerically estimating the expected local geometric alteration $\int_{\mathcal{M}} A(x) d\mu(x)$ defined from \Cref{thm:local-distortion-bound}. This directly quantifies how closely the local Jacobian norm $\|d_x T\|$ hovers around the ideal isometric value of $1$.

\Cref{fig:isometry_comparison} illustrates stability and distortion trade-offs across learning rates. Notably, MDS1 suffers immediate gradient explosions ([X] markers) in all configurations due to unbounded gradients caused by its unconstrained Hessian on large squared distances. 

While MDS2 and GME achieve similar global isometry at lower learning rates (\Cref{fig:global_iso}), MDS2 proves brittle and explodes at the highest rate ($1.2 \times 10^{-1}$). Conversely, GME remains entirely stable across all configurations and achieves the lowest local distortion $\int A(x) d\mu$ (\Cref{fig:local_iso}). GME's logarithmic formulation naturally discounts massive global errors, organically shifting the optimization focus toward preserving fine-grained local relationships essential for stable latent encoding.

Finally, \Cref{fig:validation} directly validates \Cref{thm:local-distortion-bound}. According to our bound, the expected local distortion scales as $\mathcal{O}\big(\GME(T)^{\frac{2k-4}{m+2k}}\big)$. Because the network employs the analytically infinitely differentiable ($C^\infty$) Softplus activation, the theoretical limiting slope as $k \to \infty$ is exactly $1$. The log-log plot of expected local distortion versus GME loss reveals an empirical slope that closely approximates this theoretical limit when trained with a learning rate of $1.2 \times 10^{-2}$. While there is a very slight deviation from the exact unit slope, a natural artifact of the network's finite capacity and finite-precision arithmetic \citep{wang2021understanding}, the overall trajectory confirms our theoretical analysis. Ultimately, this demonstrates that the local distortion shrinks in strong agreement with the infinite-smoothness limit of our bound, constrained by the data's intrinsic dimensionality.

\begin{figure}[htbp]
    \centering
    \begin{subfigure}[b]{0.49\textwidth}
        \centering
        \includegraphics[width=\textwidth]{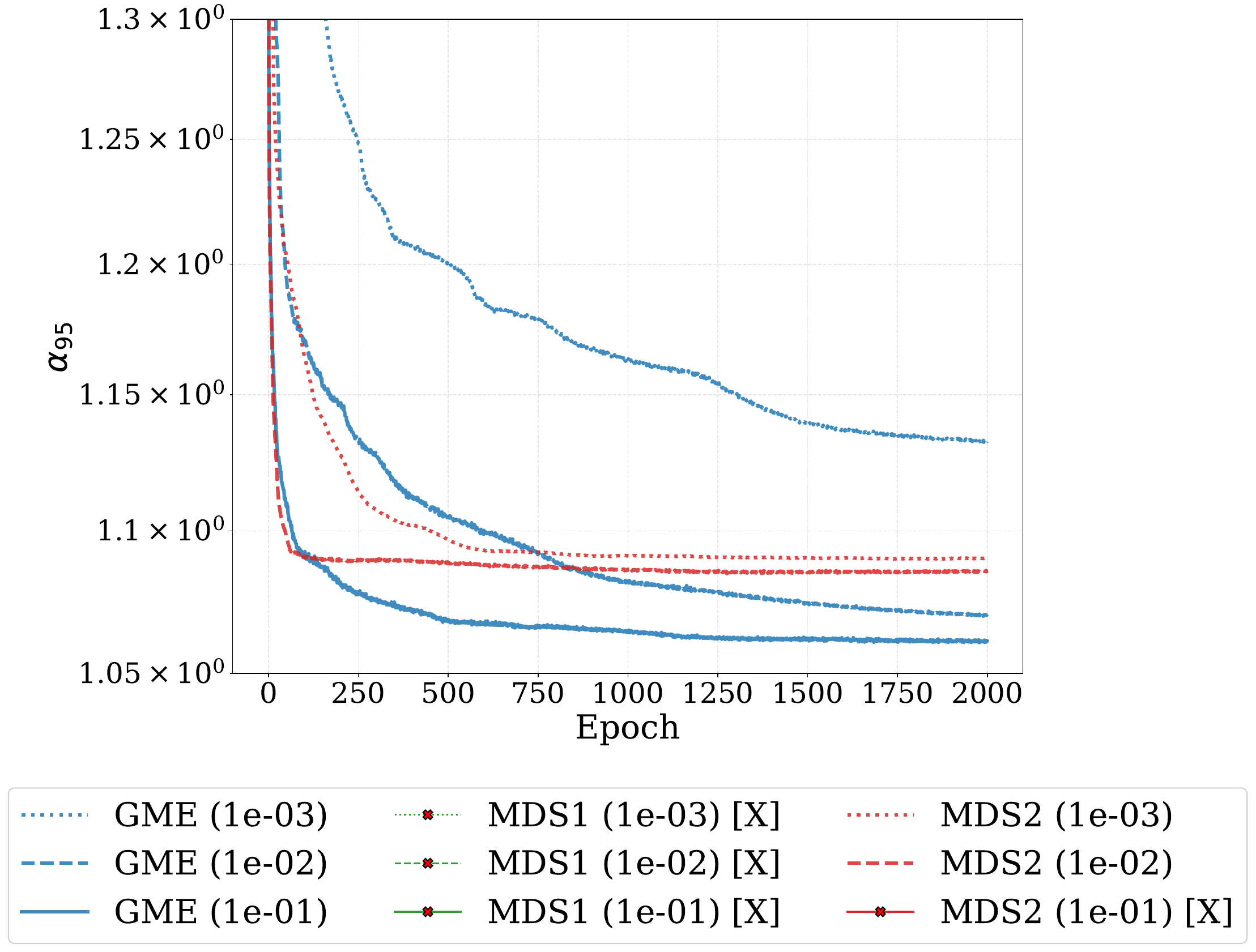}
        \caption{Global Bi-Lipschitz Bounds ($\alpha_{95}$)}
        \label{fig:global_iso}
    \end{subfigure}
    \hfill
    \begin{subfigure}[b]{0.49\textwidth}
        \centering
        \includegraphics[width=\textwidth]{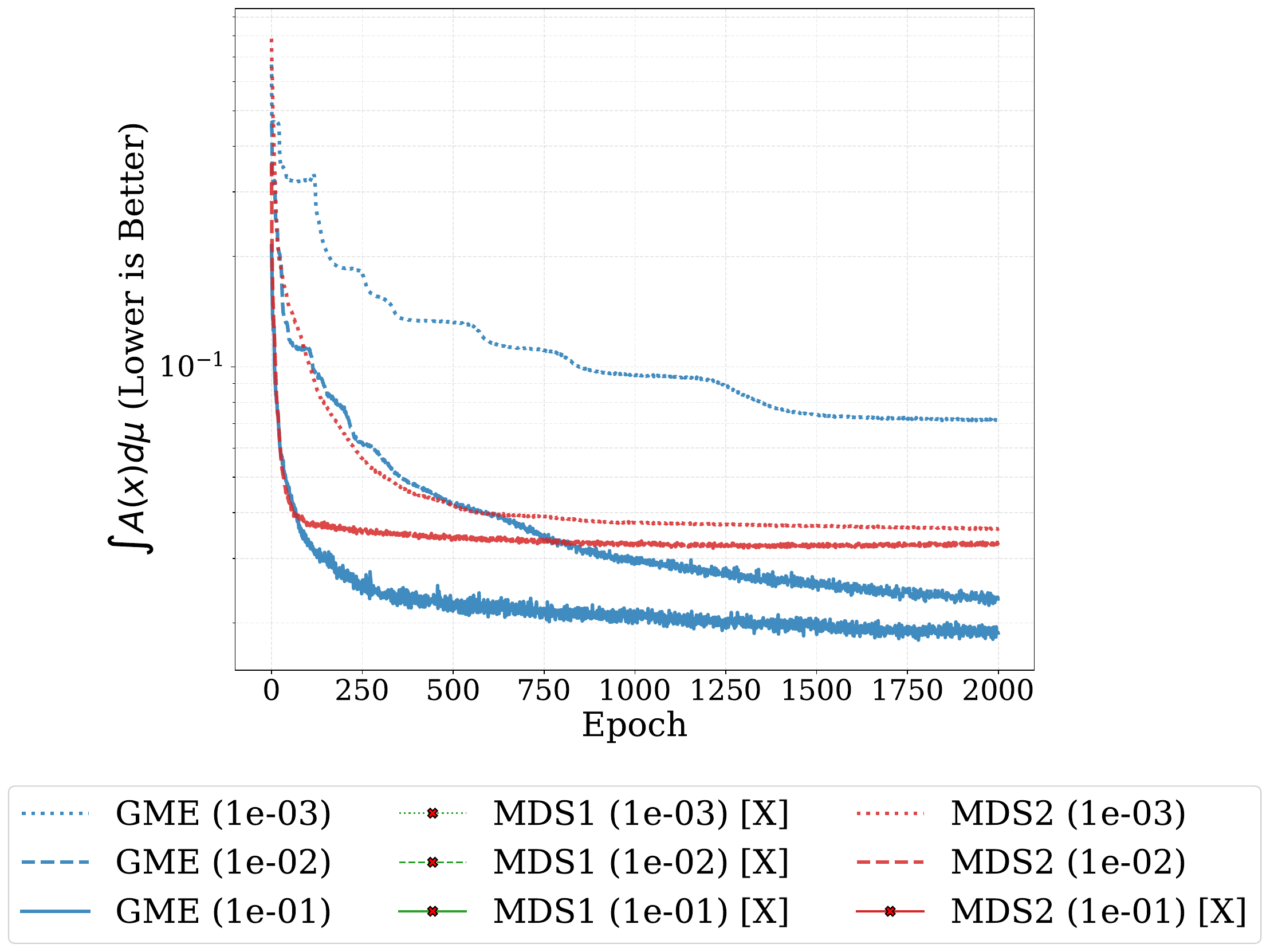}
        \caption{Local Distortion ($\int A(x) d\mu$)}
        \label{fig:local_iso}
    \end{subfigure}
    \caption{Comparison of geometric distortion metrics across training epochs for CIFAR-10 embeddings using a 4-layer encoder with Softplus activations ($d=25$). The models were evaluated across three learning rates ($1.2\times 10^{-3}$, $1.2\times 10^{-2}$, and $1.2\times 10^{-1}$). \textbf{Left:} The 95th percentile $\alpha$ value, representing the global distortion bound. \textbf{Right:} The expected local distortion across all stable configurations. GME demonstrates superior performance in both global and local isometry compared to MDS2. Furthermore, while GME remains highly stable, MDS2 suffers from gradient explosion at the highest learning rate (indicated by [X]). Note that MDS1 suffered from immediate gradient explosion across all configurations due to its unconstrained Hessian.}
    \label{fig:isometry_comparison}
\end{figure}

\begin{figure}[htbp]
    \centering
    \begin{subfigure}[b]{0.5\textwidth}
        \centering
        \includegraphics[width=\textwidth]{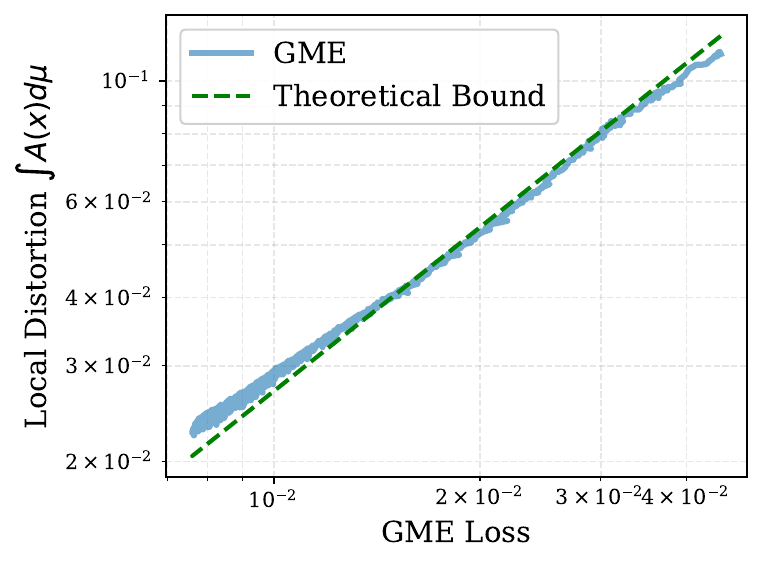}
    \end{subfigure}
    \caption{Empirical validation of the local distortion bound (lr = 0.012). The log-log plot of expected local distortion versus GME loss shows a slope of approximately 1. This matches the $C^\infty$ theoretical exponent limit of $\lim_{k\to\infty} \frac{2k-4}{m+2k} = 1$, directly validating \Cref{thm:local-distortion-bound}.}
    \label{fig:validation}
\end{figure}

\subsection{Optimization efficiency with GPE encoders}

Throughout the experiments, we implement the algorithm on real-world datasets. Specifically, we use the MNIST~\cite{deng2012mnist}, CIFAR10~\cite{cifar10}, CelebA~\cite{liu2015faceattributes}, and CelebA-HQ ($256 \times 256$)~\cite{celebahq} datasets to demonstrate the efficiency of the GPE framework compared to the VAE framework in terms of reconstruction and generation tasks.

In this experiment, we demonstrate the efficiency of the optimization process for computing the decoder while considering the geometry-preserving property of the encoder, as established in Theorem~\ref{thm:S-2}. Specifically, we illustrate that the computation of the decoder becomes more efficient and faster as the encoder's bi-Lipschitz constant $\alpha$ approaches 1. For this experiment, we utilize the CIFAR10 dataset.

We employ three distinct encoders, each computed using Algorithm~\ref{alg:T} and the GME cost, with training stopped at three difference $\TOL$ values as indicated in \Cref{fig:s_loss}. We use the same neural network architectures for encoders and the same learning rate ${lr}=10^{-4}$ and batch size $100$.

Using each pretrained encoder with a different value of $\TOL$, we use \Cref{alg:S} to minimize the cost in \eqref{eq:L-recons}. Figure~\ref{fig:s_loss} displays the loss plots for the first $5,000$ iterations. We use the same neural network architectures for decoders and the same learning rate ${lr}=10^{-4}$ and batch size $100$. The figure clearly illustrates the relationship between optimization efficiency and the level of geometry preservation exhibited by the encoder.

\begin{figure}[!ht]
    \centering
    \begin{subfigure}[t]{0.45\textwidth}
        \centering
        \includegraphics[width=\textwidth]{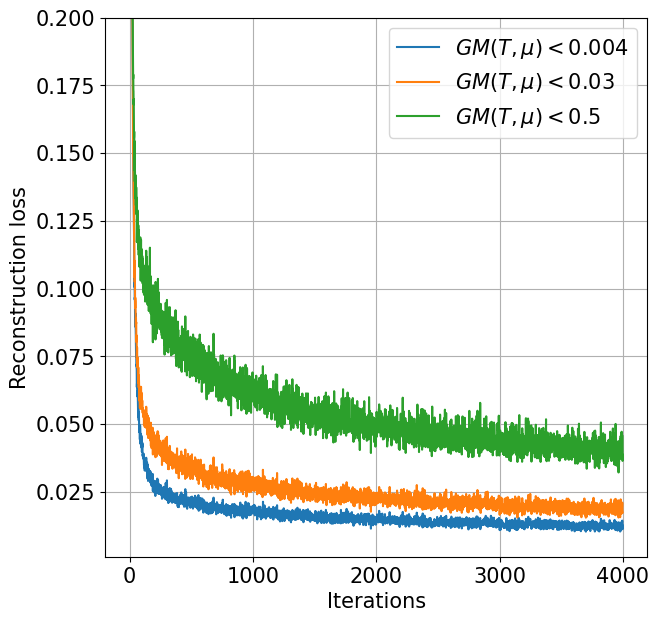}
        \caption{MNIST}
    \end{subfigure}
    \begin{subfigure}[t]{0.45\textwidth}
        \centering
        \includegraphics[width=\textwidth]{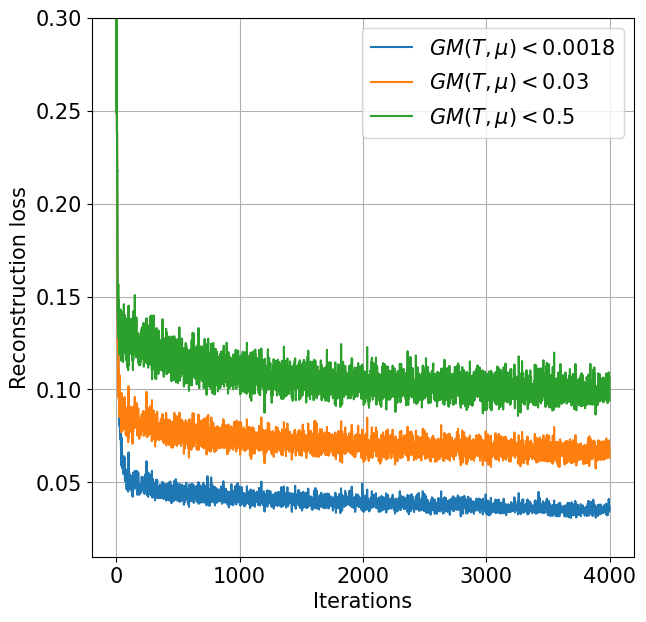}
        \caption{CIFAR10}
    \end{subfigure}

    \begin{subfigure}[t]{0.45\textwidth}
        \centering
        \includegraphics[width=\textwidth]{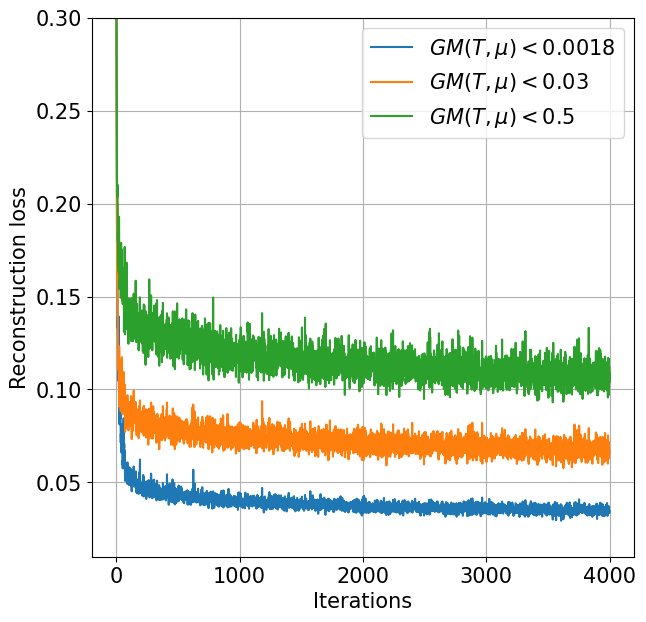}
        \caption{CelebA}
    \end{subfigure}
    \begin{subfigure}[t]{0.45\textwidth}
        \centering
        \includegraphics[width=\textwidth]{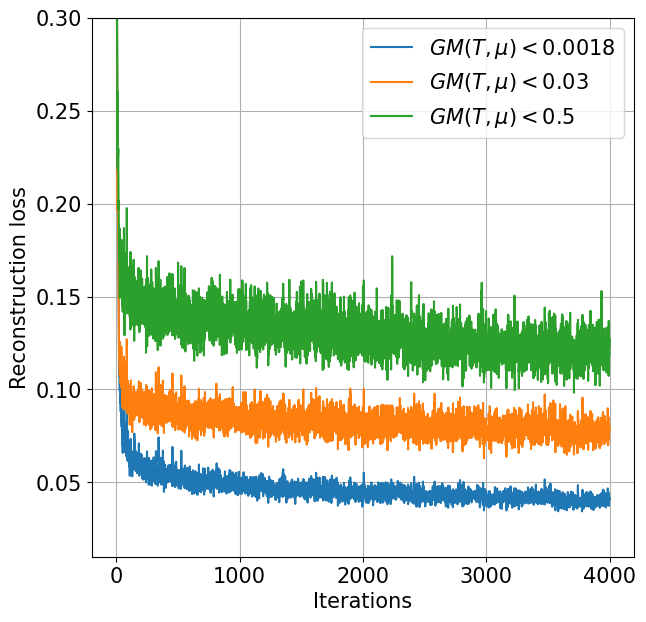}
        \caption{CelebA-HQ}
    \end{subfigure}
    \caption{The plots illustrate the results of minimizing the reconstruction loss in \eqref{eq:L-recons} on four datasets, MNIST, CIFAR10, CelebA, and CelebA-HQ using three different encoders, each characterized by a different tolerance values showing the encoder with a smaller GME loss converges faster.}
    \label{fig:s_loss}
\end{figure}

\subsection{Comparison of reconstruction task efficiency between GPE and VAE Frameworks}

In this experiment, we compare the efficiency of reconstruction tasks by training the encoders and decoders from both the VAE and GPE frameworks, using the same neural network architectures for both. To evaluate performance, we first train the encoders from both frameworks. Using the pre-trained encoders, we then separately train the decoders using only the reconstruction cost and assess the efficiency and performance of the training optimization results.

Note that, unlike GPE, where the optimization process involves the encoder independently of the decoder, VAE requires training both the encoder and decoder simultaneously using the ELBO loss. Therefore, to obtain the pre-trained encoders for VAE, we first train both the encoder and decoder. Afterward, we freeze the VAE encoder, ignore the trained decoder, and train a new decoder from scratch. This allows us to compare the optimization process of minimizing solely the reconstruction loss for decoders from both frameworks.

The results, shown in \Cref{tab:pretrained_encoders}, indicate the iterations required for each method to reach the specified loss value. For GPE, the third column includes the value of $\TOL$ used in \Cref{alg:T}. All parameters, including learning rates and neural network architectures for both the encoder and decoder, are identical across datasets. The results for GPE were obtained using two A40 GPUs, while two A100 GPUs were used for VAE.

The numerical results indicate a reduction in the number of training iterations required by GPE compared to VAE across all datasets. Specifically, some datasets exhibit a substantial improvement, with iteration reductions ranging from 3 times (CelebA) to 900 times (MNIST) with the given loss thresholds. Notably, the discrepancy in the number of iterations becomes more pronounced as the loss threshold is decreased. For instance, in the case of CelebA, GPE converges to a loss value of 0.005 within 225,000 iterations, whereas VAE fails to converge to this loss even after 700,000 iterations.

However, the performance on the CelebA-HQ dataset is similar for both VAE and GPE. This can be explained by the fact that, although CelebA-HQ resides in a higher-dimensional space ($\mathbb{R}^{256 \times 256 \times 3}$), it only consists of 30,000 images, which is much smaller in size compared to CelebA, which contains over 200,000 images. We believe this smaller dataset size allows for easier convergence of the decoder to a lower reconstruction loss.

The visualization of the convergence of the optimization process for minimizing the reconstruction cost is shown in \Cref{fig:mnist-training-s}, and the results of the reconstructed images from GPE and VAE are displayed in \Cref{fig:combined}. \Cref{fig:combined_single} highlights the notable difference in reconstruction performance between VAE and GPE, clearly showing better performance from GPE.

\begin{table*}[h!]
\centering
\begin{tabular}{|c|c|c||r|c||r|r|}
\hline
 {\begin{tabular}[c]{@{}c@{}}{Pretrained}\\ { Encoder}\end{tabular}}  
&   {\begin{tabular}[c]{@{}c@{}}{Latent}\\ {Dim.}\end{tabular}} 
&   {\begin{tabular}[c]{@{}c@{}}{Compression}\\ {Rate}\end{tabular}} 
&   {\begin{tabular}[c]{@{}c@{}}{Encoder}\\ {iterations }\end{tabular}} 
&   {$\TOL$}  
&   {\begin{tabular}[c]{@{}c@{}}{Decoder}\\ {iterations}\end{tabular}}   
&   {\begin{tabular}[c]{@{}c@{}}{Decoder}\\ {iterations}\end{tabular}}   \\  \hline\hline
 \multicolumn{5}{|c||}{{MNIST}}
 & \multicolumn{1}{c||}{{Loss: 0.1}}
 & \multicolumn{1}{c|}{{Loss: 0.022}}
 \\ \hline\hline
VAE & 30 & 34.13 & 2,700,000 & N/A  & 80    &  280,899    \\ \hline
GPE & 30 & 34.13 & 3,800 & 0.03  & 47  & 1,044 \\ \hline
GPE & 30 & 34.13 & 16,000 & 0.004  & 47  & 308 \\ \hline \hline
\multicolumn{5}{|c||}{{CIFAR10}}
& \multicolumn{1}{c||}{{Loss: 0.1}}
 & \multicolumn{1}{c|}{{Loss: 0.015}}
 \\ \hline\hline
VAE & 100 & 30.72 & 1,250,000 & N/A   & 37  & 159,036 \\ \hline
GPE & 100 & 30.72 & 2,700 & 0.03  & 19  & 64,281 \\ \hline
GPE & 100 & 30.72 & 23,700  & 0.0009  & 19  &  11,071\\ \hline\hline 
\multicolumn{5}{|c||}{{CelebA}}
& \multicolumn{1}{c||}{{Loss: 0.1}}
 & \multicolumn{1}{c|}{{Loss: 0.015}}
 \\ \hline\hline
VAE & 100 & 122.88 & 1,215,000  & N/A  & 28  & 120,497\\ \hline
GPE & 100 & 122.88 & 66,000 & 0.0007  & 20  & 40,023 \\ \hline\hline 
\multicolumn{5}{|c||}{{CelebA-HQ}}
& \multicolumn{1}{c||}{{Loss: 0.1}}
 & \multicolumn{1}{c|}{{Loss: 0.0045 }}
 \\ \hline\hline
VAE & 100 & 1966.08 & 180,000 & N/A &  23  & 105,178\\ \hline
GPE & 100 & 1966.08 & 141,100 & 0.0007  & 21  & 104,129\\ \hline 
\end{tabular}
\caption{Comparison of training decoders from pretrained encoders on various datasets. The table includes the latent dimension and the compression rate for each setup, calculated as the ratio of the ambient dimension to the latent dimension (ambient dimensions: $32 \times 32$ for MNIST, $32 \times 32 \times 3$ for CIFAR10, $64 \times 64 \times 3$ for CelebA, and $256 \times 256 \times 3$ for CelebA-HQ).}
\label{tab:pretrained_encoders}
\end{table*}

\begin{figure}[h!]
    \centering
    \begin{subfigure}[b]{0.16\textwidth}
        \centering
        \includegraphics[width=\textwidth]{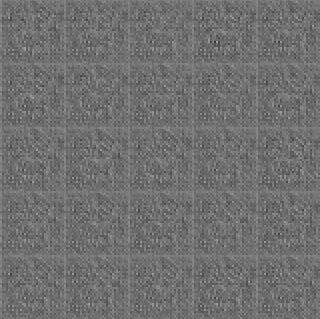}
        \caption*{It. 0 (0 s)}
    \end{subfigure}
    \hfill
    \begin{subfigure}[b]{0.16\textwidth}
        \centering
        \includegraphics[width=\textwidth]{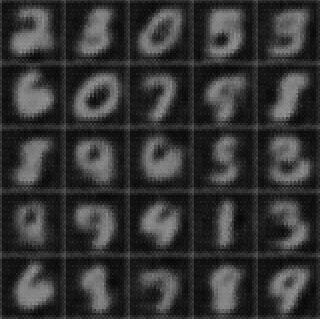}
        \caption*{It. 100 (4 s)}
    \end{subfigure}
    \hfill
    \begin{subfigure}[b]{0.16\textwidth}
        \centering
        \includegraphics[width=\textwidth]{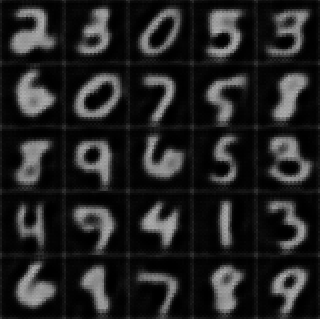}
        \caption*{It. 200 (8 s)}
    \end{subfigure}
    \hfill
    \begin{subfigure}[b]{0.16\textwidth}
        \centering
        \includegraphics[width=\textwidth]{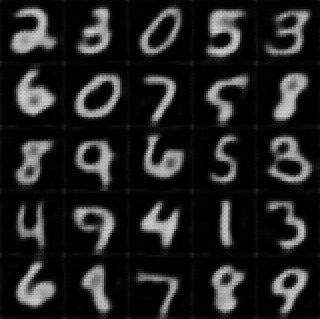}
        \caption*{It. 300 (12 s)}
    \end{subfigure}
    \hfill
    \begin{subfigure}[b]{0.16\textwidth}
        \centering
        \includegraphics[width=\textwidth]{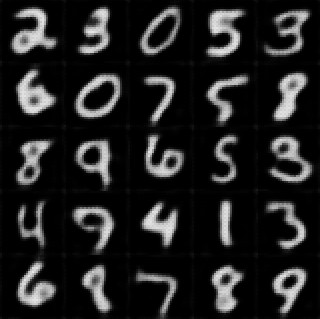}
        \caption*{It. 400 (16 s)}
    \end{subfigure}
    \hfill
    \begin{subfigure}[b]{0.16\textwidth}
        \centering
        \includegraphics[width=\textwidth]{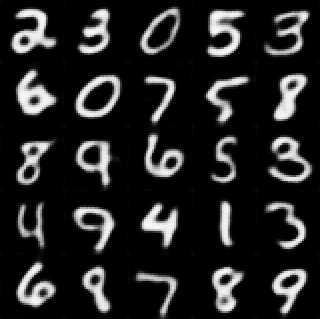}
        \caption*{It. 600 (24 s)}
    \end{subfigure}

   \begin{subfigure}[b]{0.19\textwidth}
        \centering
        \includegraphics[width=\textwidth]{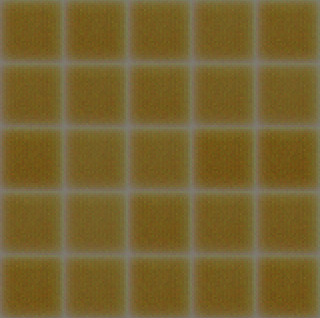}
        \caption*{It. 0 (0 m)}
    \end{subfigure}
    \hfill
    \begin{subfigure}[b]{0.19\textwidth}
        \centering
        \includegraphics[width=\textwidth]{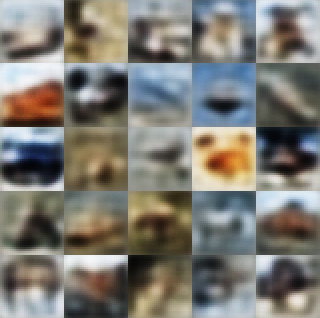}
        \caption*{It. 100 (0.15 m)}
    \end{subfigure}
    \hfill
    \begin{subfigure}[b]{0.19\textwidth}
        \centering
        \includegraphics[width=\textwidth]{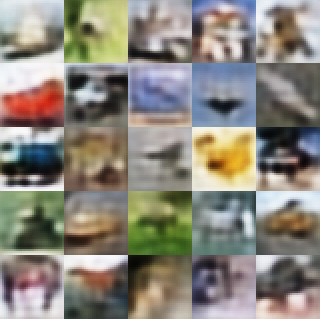}
        \caption*{It. 500 (0.75 m)}
    \end{subfigure}
    \hfill
    \begin{subfigure}[b]{0.19\textwidth}
        \centering
        \includegraphics[width=\textwidth]{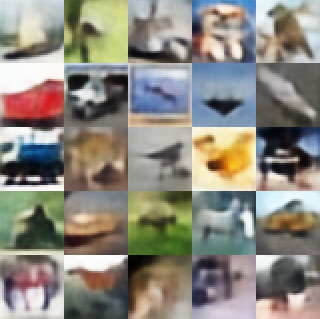}
        \caption*{It. 5,000 (7.5 m)} 
    \end{subfigure}
    \hfill
    \begin{subfigure}[b]{0.19\textwidth}
        \centering
        \includegraphics[width=\textwidth]{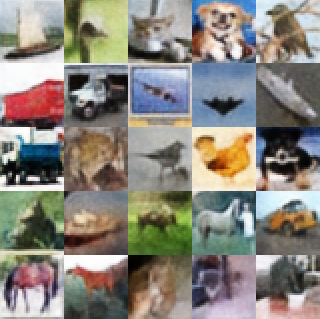}
        \caption*{It. 10,000 (15 m)}
    \end{subfigure} 
    
   \begin{subfigure}[b]{0.19\textwidth}
        \centering
        \includegraphics[width=\textwidth]{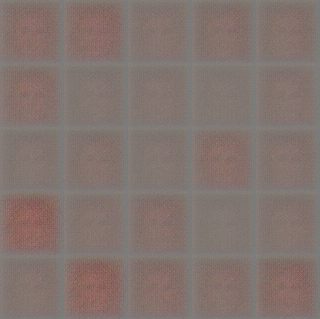}
        \caption*{It. 0 (0 s)}
    \end{subfigure}
    \hfill
    \centering
        \begin{subfigure}[b]{0.19\textwidth}
        \includegraphics[width=\textwidth]{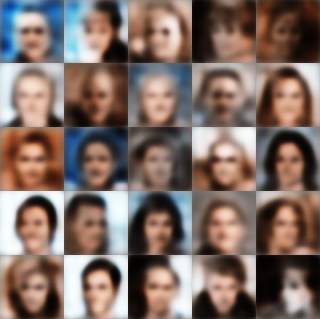}
        \caption*{It. 100 (0.2 m)}
    \end{subfigure}
    \hfill
    \begin{subfigure}[b]{0.19\textwidth}
        \centering
        \includegraphics[width=\textwidth]{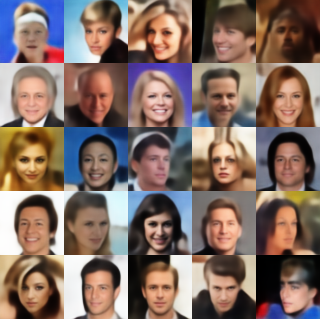}
        \caption*{It. 20,000 (40 m)} 
    \end{subfigure}
    \hfill
    \begin{subfigure}[b]{0.19\textwidth}
        \centering
        \includegraphics[width=\textwidth]{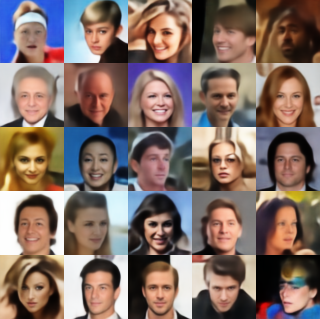}
        \caption*{It. 60,000 (2 hrs)} 
    \end{subfigure}
    \hfill
    \begin{subfigure}[b]{0.19\textwidth}
        \centering
        \includegraphics[width=\textwidth]{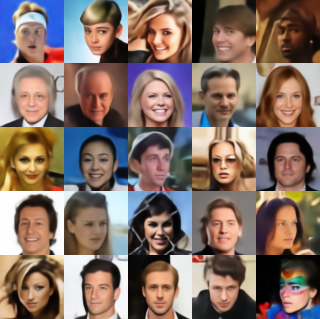}
        \caption*{It. 140,000 (8 hrs)} 
    \end{subfigure}

   \begin{subfigure}[b]{0.19\textwidth}
        \centering
        \includegraphics[width=\textwidth]{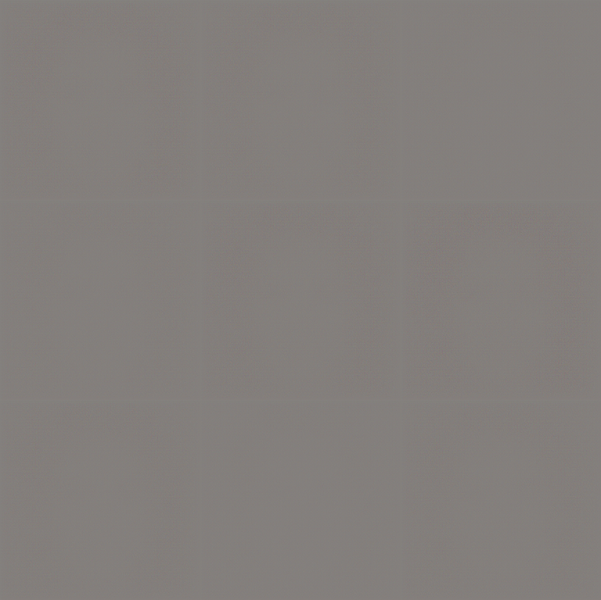}
        \caption*{It. 0 (0 s)}
    \end{subfigure}
    \hfill
    \centering
        \begin{subfigure}[b]{0.19\textwidth}
        \includegraphics[width=\textwidth]{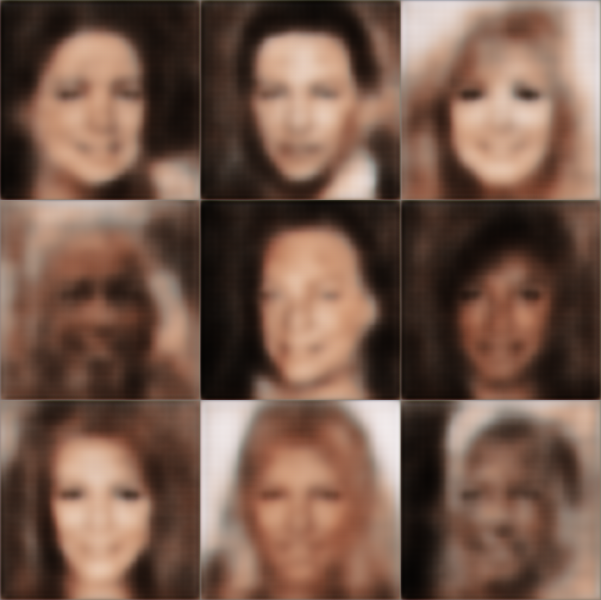}
        \caption*{It. 100 (0.8 m)}
    \end{subfigure}
    \hfill
    \begin{subfigure}[b]{0.19\textwidth}
        \centering
        \includegraphics[width=\textwidth]{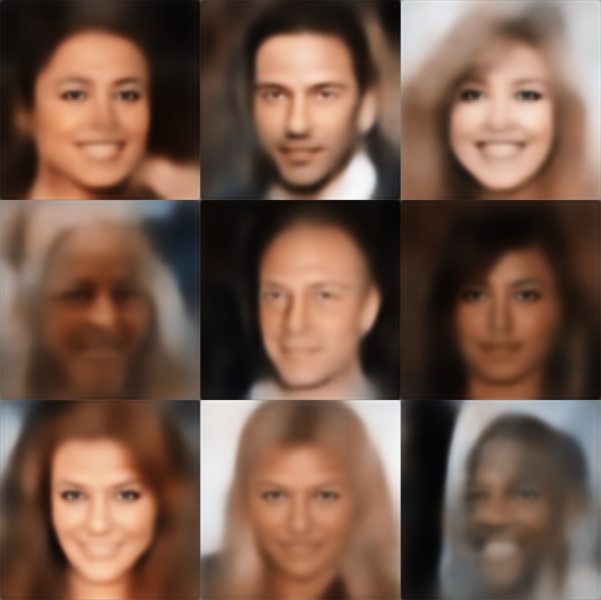}
        \caption*{It. 1,000 (7 m)} 
    \end{subfigure}
    \hfill
    \begin{subfigure}[b]{0.19\textwidth}
        \centering
        \includegraphics[width=\textwidth]{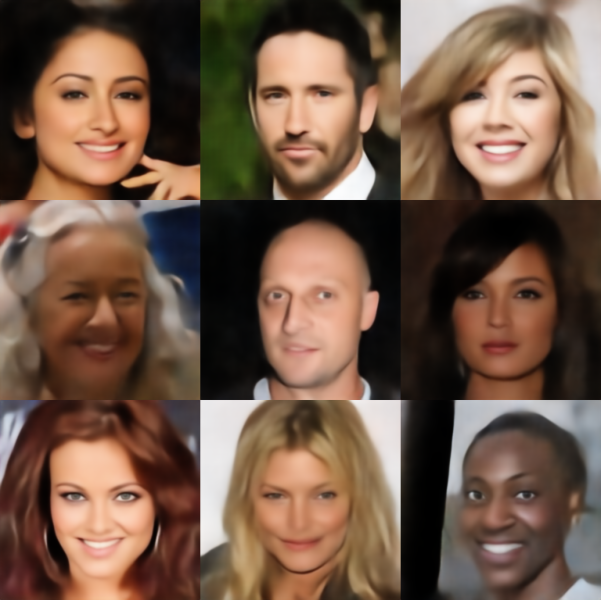}
        \caption*{It. 30,000 (3.4 hrs)} 
    \end{subfigure}
    \hfill
    \begin{subfigure}[b]{0.19\textwidth}
        \centering
        \includegraphics[width=\textwidth]{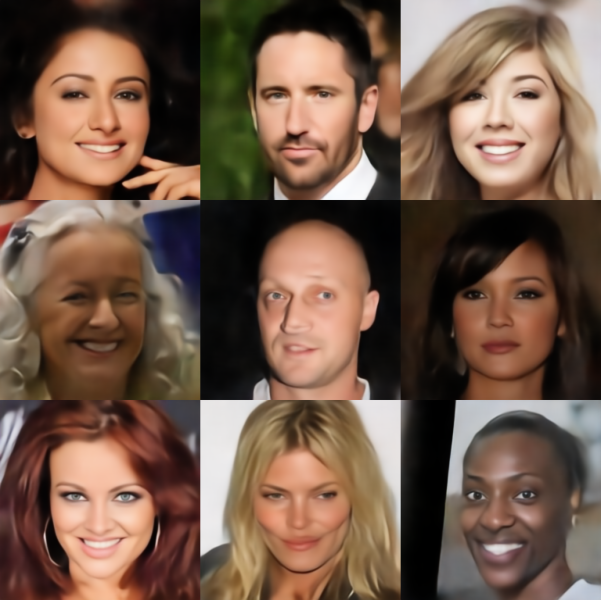}
        \caption*{It. 70,000 (7.8 hrs)} 
    \end{subfigure} 
    \caption{Convergence of training a decoder $G$ using a geometry-preserving encoder with various datasets: first row (MNIST), second row (CIFAR10), third row (CelebA), and fourth row (CelebA-HQ). Each figure caption shows the iteration number and the time taken in seconds (s) and minutes (m). Different neural network architectures for the decoder were used for each dataset. All experiments were done with the same GPU settings: 2 A40 GPUs.}
    \label{fig:mnist-training-s}
\end{figure}

\begin{figure}[h!]
    \centering
     \begin{subfigure}[b]{0.26\textwidth}
         \centering
         \includegraphics[width=\textwidth]{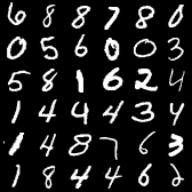}
         \caption{Original}
     \end{subfigure}
     \begin{subfigure}[b]{0.26\textwidth}
         \centering
         \includegraphics[width=\textwidth]{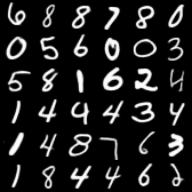}
         \caption{VAE}
     \end{subfigure}
     \begin{subfigure}[b]{0.26\textwidth}
         \centering
         \includegraphics[width=\textwidth]{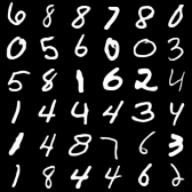}
         \caption{GPE}
     \end{subfigure}\\
    \begin{subfigure}[b]{0.26\textwidth}
        \centering
        \includegraphics[width=\textwidth]{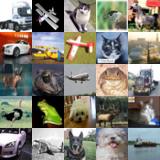}
        \caption{Original}
    \end{subfigure}
    \begin{subfigure}[b]{0.26\textwidth}
        \centering
        \includegraphics[width=\textwidth]{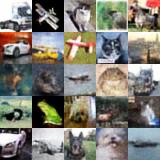}
        \caption{VAE}
    \end{subfigure}
    \begin{subfigure}[b]{0.26\textwidth}
        \centering
        \includegraphics[width=\textwidth]{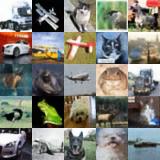}
        \caption{GPE}
    \end{subfigure}

    \begin{subfigure}[b]{0.26\textwidth}
        \centering
        \includegraphics[width=\textwidth]{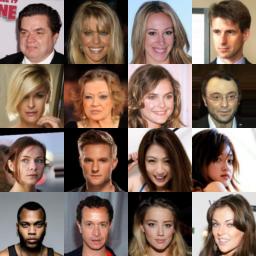}
        \caption{Original}
    \end{subfigure}
    \begin{subfigure}[b]{0.26\textwidth}
        \centering
        \includegraphics[width=\textwidth]{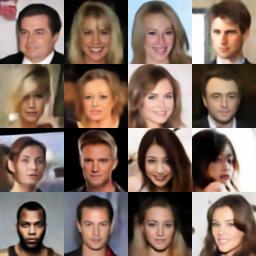}
        \caption{VAE}
    \end{subfigure}
    \begin{subfigure}[b]{0.26\textwidth}
        \centering
        \includegraphics[width=\textwidth]{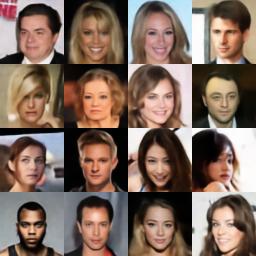}
        \caption{GPE}
    \end{subfigure}
    
    \begin{subfigure}[b]{0.26\textwidth}
        \centering
        \includegraphics[width=\textwidth, clip = true, trim = 0 555 0 0]{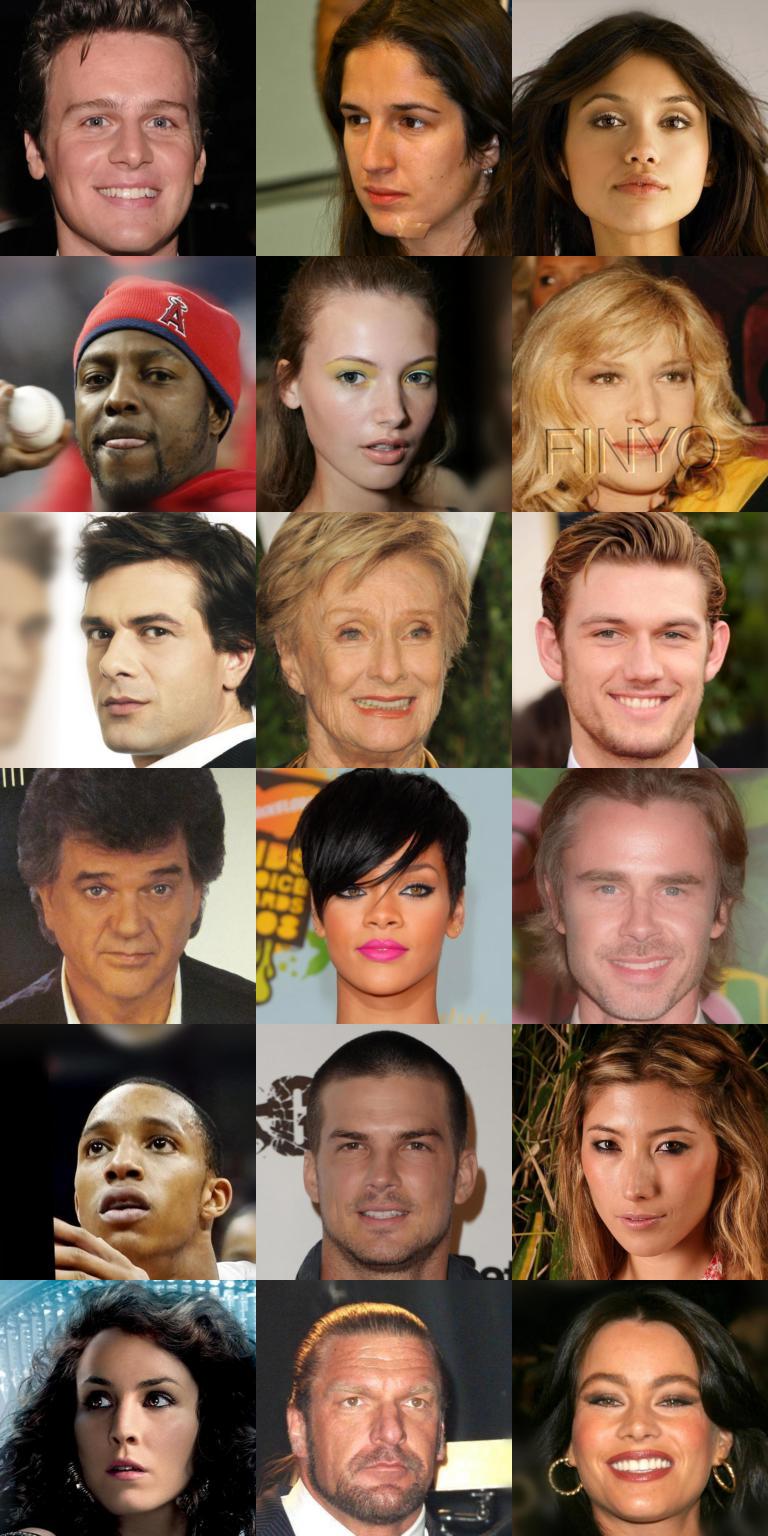}
        \caption{Original}
    \end{subfigure}
    \begin{subfigure}[b]{0.26\textwidth}
        \centering
        \includegraphics[width=\textwidth, clip = true, trim = 0 555 0 0]{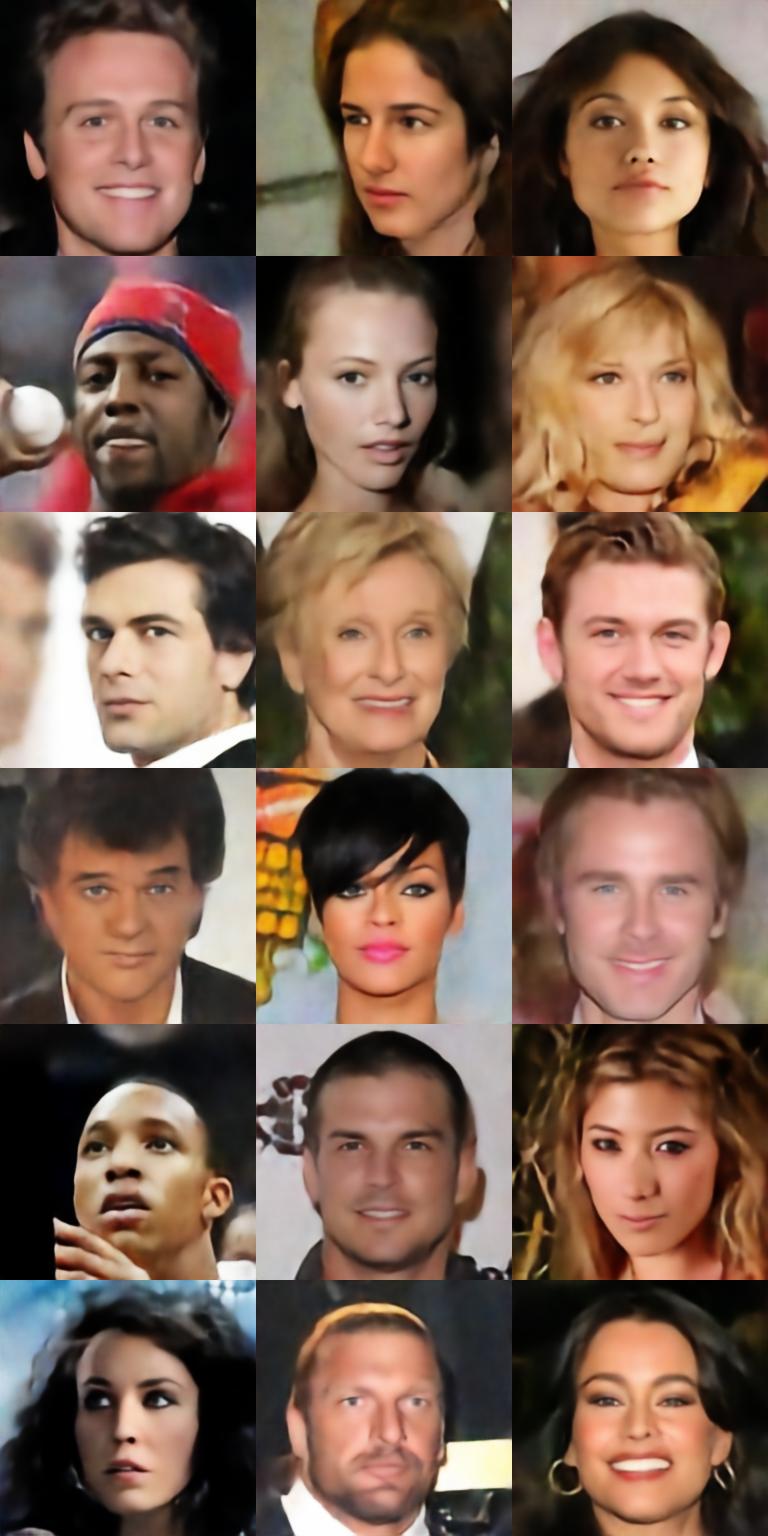}
        \caption{VAE}
    \end{subfigure}
    \begin{subfigure}[b]{0.26\textwidth}
        \centering
        \includegraphics[width=\textwidth, clip = true, trim = 0 555 0 0]{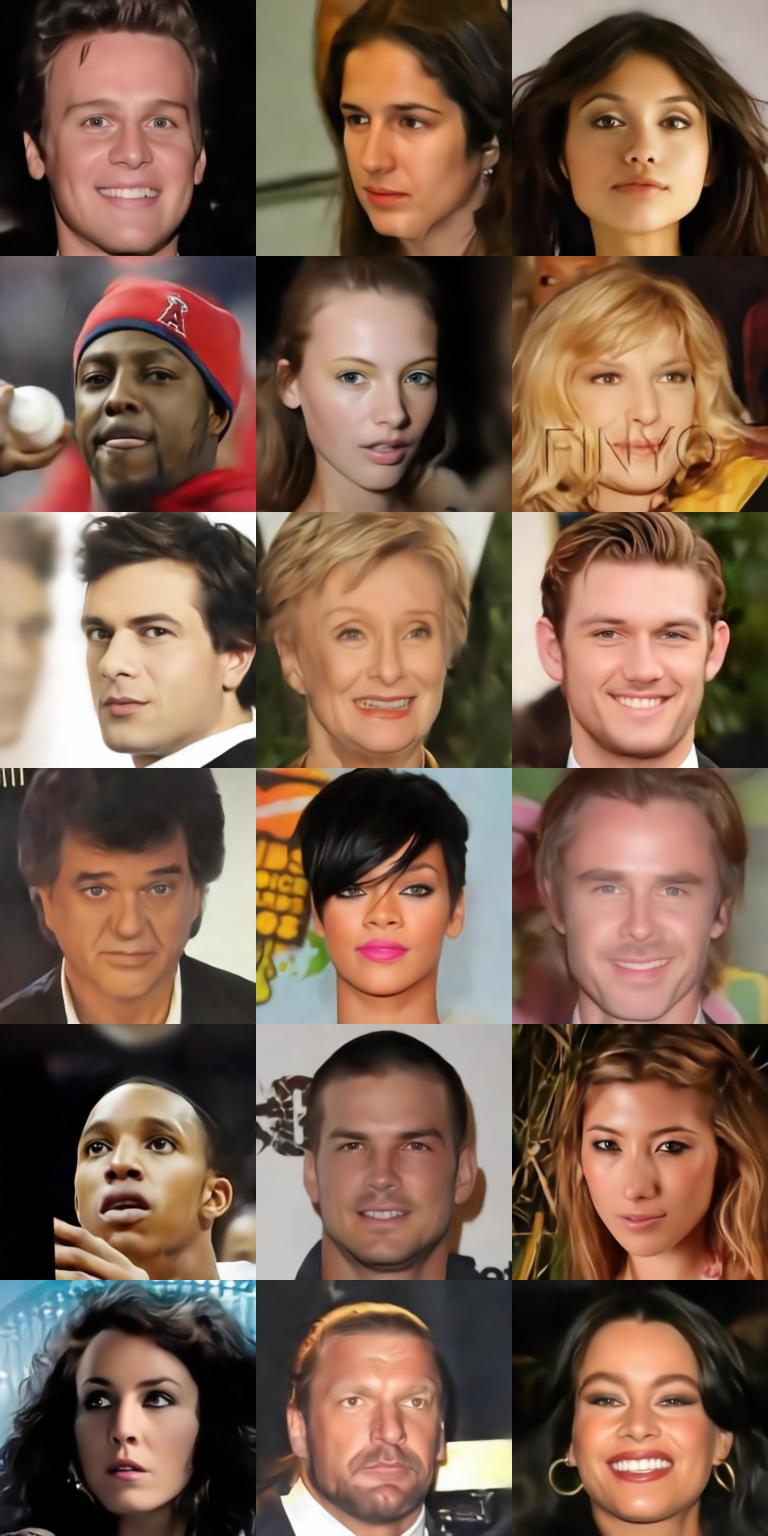}
        \caption{GPE}
    \end{subfigure}

    \caption{The figure shows the original images (left column), the corresponding reconstructed images from the VAE (center column), and those from the GPE (right column). The datasets used are MNIST (a) - (c), CIFAR10 (d) - (f), CelebA (g) - (i), and CelebA-HQ (256$\times$256) (j) - (l).}
    \label{fig:combined}
\end{figure}

\begin{figure}[htbp]
    \centering
    \begin{minipage}{0.24\linewidth}
        \centering
        \includegraphics[width=\linewidth]{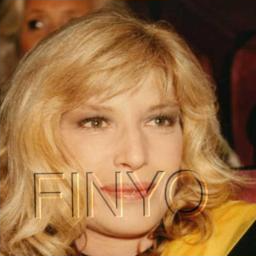}
        \caption*{Original Image}
    \end{minipage}%
    \hspace{0.1cm}
    \begin{minipage}{0.24\linewidth}
        \centering
        \includegraphics[width=\linewidth]{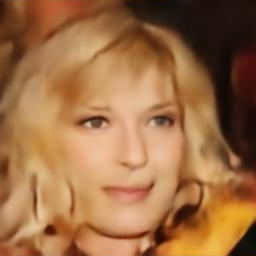}
        \caption*{VAE}
    \end{minipage}%
    \hspace{0.1cm}
    \begin{minipage}{0.24\linewidth}
        \centering
        \includegraphics[width=\linewidth]{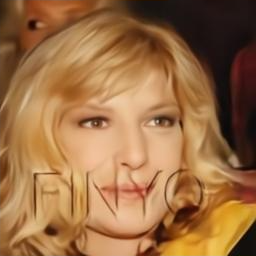}
        \caption*{GPE}
    \end{minipage}
    \caption{Examples of reconstructed images from CelebA-HQ, as shown in \Cref{fig:combined}, highlighting the clear distinction between VAE and GPE reconstructed images, presented alongside the original image.}
    \label{fig:combined_single}
\end{figure}

\subsection{Convergence rates of reconstruction loss: GPE vs.\ VAE-based methods}
The preceding experiment examined convergence of neural network training that minimizes only the reconstruction loss, using pretrained encoders from GPE and VAE. However, this setup does not provide a fair comparison of training time and efficiency between GPE and VAE, because the two frameworks differ in their training protocols: VAE trains the encoder and decoder simultaneously, whereas GPE trains them sequentially, first the encoder and then the decoder. Consequently, convergence results from decoder-only training cannot reflect the actual end-to-end training time for these two different methods.

In this experiment, we therefore compare the methods under their intended training regimes: simultaneous encoder--decoder training for VAE-based methods and sequential encoder--then--decoder training for GPE. We report convergence of the reconstruction loss for our GPE encoder and decoder against four baselines: VAE, VAE with multidimensional scaling (MDS) regularization \cite{yu2013embedding}, a plain autoencoder (AE, without the KL term), and AE with MDS regularization.
The MDS penalty is
\[
    L_{\mathrm{MDS}}(T)
    = \mathbb{E}_{x, x' \sim \mu} \left[\bigl(\|x - x'\| - \|T(x) - T(x')\|\bigr)^2\right].
\]
For VAE+MDS and AE+MDS, the MDS coefficient is set to $1$, chosen empirically to ensure stable, concurrent decay of the reconstruction term, the MDS term, and, when applicable, the KL term.

For the VAE we use a $\beta$-VAE \cite{higgins2017beta} with $\beta=0.1$. This value was selected empirically: larger $\beta$ slows the decay of the reconstruction loss, while very small $\beta$ prevents the Kullback-Leibler term from optimizing properly.

A key algorithmic difference is the training schedule. Our method trains the encoder and decoder separately, whereas the baselines train both components simultaneously. To ensure a fair comparison, we first train the GPE encoder until a target encoder–loss threshold is reached, and only then train the decoder. In all reconstruction–loss figures, the GPE decoder curve is shifted along the time axis by the encoder pretraining time to reflect true wall–clock usage. For VAE, AE, VAE+MDS, and AE+MDS, encoder and decoder are trained jointly with no offset.

All methods use the same encoder and decoder architectures and the batch size specified in \Cref{tab:pretrained_encoders}. Reconstruction loss is evaluated on the same held–out validation set of $5{,}000$ examples, identical across methods. Learning rates are selected from $\{10^{-3},10^{-4},10^{-5}\}$ by choosing the largest value that yields stable convergence for each method. For GPE, the tighter Hessian bound permits a larger stable step size. Concretely, the GPE encoder uses $10^{-4}$ on MNIST and $10^{-3}$ on the other datasets; the GPE decoder uses $10^{-4}$ on all datasets. The other methods use $10^{-5}$ on all datasets. The GPE encoder is trained until the thresholds in \Cref{tab:pretrained_encoders} are met, after which the GPE decoder is trained to its corresponding thresholds. The baselines are stopped when their thresholds in \Cref{tab:pretrained_encoders} are met. All experiments are performed on a single NVIDIA V100 GPU.

\Cref{fig:comparison-gpe-vae} reports reconstruction loss versus wall-clock time; the horizontal offset due to GPE encoder pretraining is indicated on the figure. \Cref{tab:comparison-gpe-vae} summarizes the seconds per iteration for each method. The per-iteration wall-clock times of VAE and VAE+MDS are nearly identical; on the CelebA dataset the difference is approximately $0.3\%$. Although the MDS regularizer has a double-summation form with cost $O(n^{2})$ in the minibatch size $n$, its additional runtime is negligible in our training setup.

Although the GPE decoder starts after an encoder pretraining phase, the plots show that GPE achieves a lower reconstruction loss sooner than the baselines. This effect is consistent across all four datasets, indicating faster convergence even with the initial offset. On CIFAR10, GPE attains its best reconstruction loss at least $15.6$ hours earlier than any of the four VAE-based methods. Across datasets, methods trained with a KL term converge more slowly, reflecting the difficulty of jointly optimizing the encoder and the decoder under both reconstruction and KL objectives. To reach comparable performance, the VAE-based methods required substantial hyperparameter tuning of the KL and MDS weights. The sequential GPE procedure reduces this burden by decoupling the encoder and decoder stages. This advantage is especially pronounced on CelebA-HQ, where a larger architecture is used: each iteration in the decoder-only GPE stage takes about $0.58$ s because it optimizes only the decoder, while the VAE-based methods take about $0.71$ s per iteration since they must update both the encoder and the decoder, resulting in longer runtimes per step.

\begin{figure}[h!]
    \centering
    \begin{subfigure}[b]{0.42\textwidth}
        \centering
        \includegraphics[width=\textwidth]{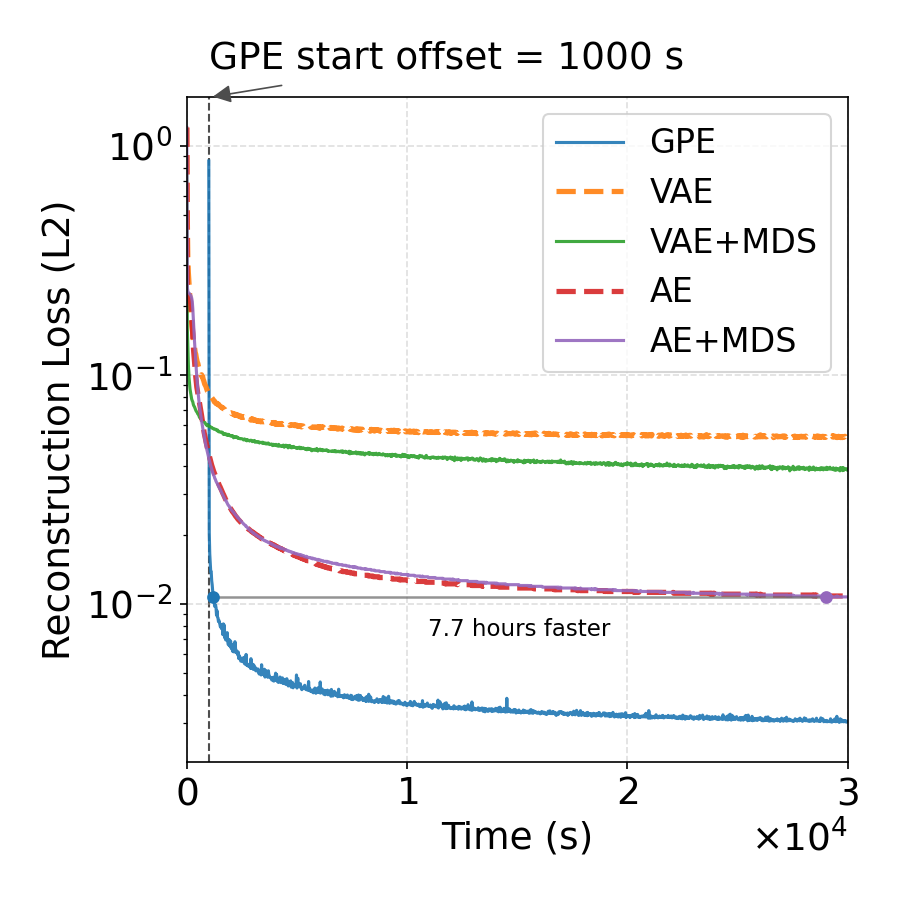}
        \caption{MNIST}
    \end{subfigure}
    \begin{subfigure}[b]{0.42\textwidth}
        \centering
        \includegraphics[width=\textwidth]{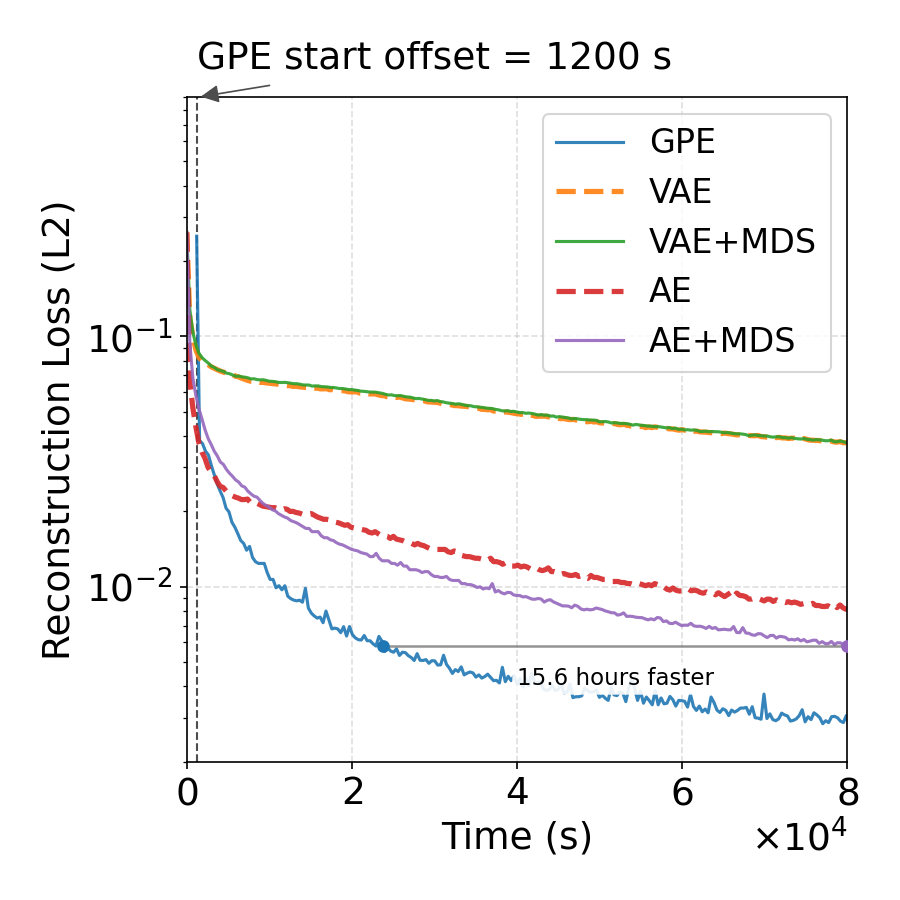}
        \caption{CIFAR10}
    \end{subfigure}
    \hfill

    \begin{subfigure}[b]{0.42\textwidth}
        \centering
        \includegraphics[width=\textwidth]{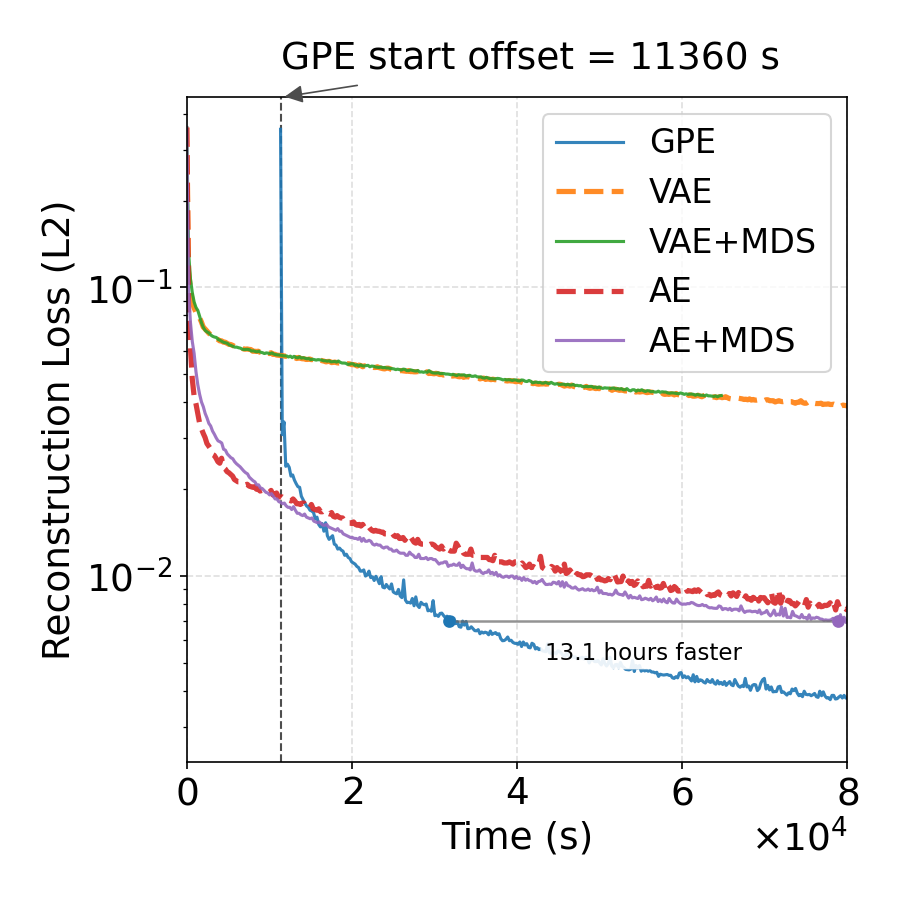}
        \caption{CelebA}
    \end{subfigure}
    \begin{subfigure}[b]{0.42\textwidth}
        \centering
        \includegraphics[width=\textwidth]{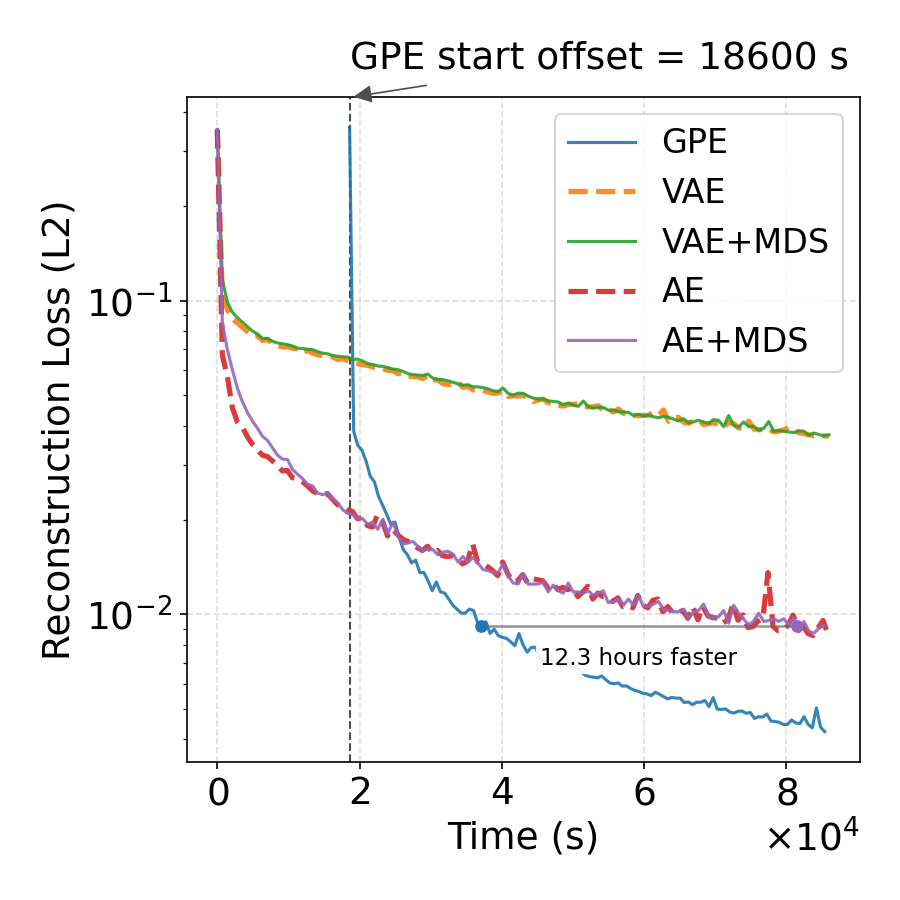}
        \caption{CelebA-HQ}
    \end{subfigure}
    \caption{Reconstruction loss versus wall-clock time during training for GPE (ours), VAE, VAE with MDS regularization, AE, and AE with MDS regularization. The computations are all performed on a single NVIDIA V100 GPU. The GPE curve is shifted on the time axis because decoder training begins only after pretraining the GPE encoder, so its iterations start later in wall-clock time. Despite this start-time offset from the pre-training of encoder, GPE still converges faster and reaches the minimum reconstruction loss achieved by the other methods sooner.
}
    \label{fig:comparison-gpe-vae}
\end{figure}

\begin{table}[h!]
\centering
\begin{tabular}{|l|c|c|c|c|}
\hline
                & MNIST   & CIFAR   & CelebA & CelebA-HQ \\
\hline
GPE Enc. Threshold & 3.5e-3      & 1.5e-3      & 5e-4  & 8e-4 \\
GPE Enc. (s/iter)  & 0.020  & 0.020  & 0.142  & 0.124 \\
GPE Dec. (s/iter)  & 0.022  & 0.306  & 0.191  & 0.588 \\
VAE (s/iter)       & 0.024  & 0.318  & 0.203  & 0.705\\
VAE+MDS (s/iter)   & 0.025  & 0.320  & 0.205  & 0.705 \\
AE (s/iter)        & 0.023  & 0.318  & 0.203  & 0.704 \\
AE+MDS (s/iter)    & 0,025  & 0.320  & 0.204  & 0.705 \\
\hline
\end{tabular}
\caption{ Wall-clock seconds per iteration for each method and dataset, measured on a single V100 GPU.}
\label{tab:comparison-gpe-vae}
\end{table}

\subsection{Comparison of latent generative task efficiency between GPE and VAE Frameworks}

In this experiment, we aim to compare the performance of latent generative tasks between the GPE and VAE frameworks.
Given pretrained encoders and decoders for each framework, we train a flow map 
$R:\mathbb{R}^d \to \mathbb{R}^d$ so that $R_{\#}\nu \approx T_{\#}\mu$, 
where $\nu$ is the latent prior and $\mu$ is the embedded data distribution induced by the encoder $T$. 
The flow map $R$ is learned with the conditional flow matching (CFM) algorithm \cite{lipman2022flow} 
using the same neural network architecture and learning rate $\mathrm{lr}=2\times 10^{-5}$ 
for both GPE and VAE. 
Generation quality is evaluated by the Fr\'echet Inception Distance (FID) \cite{heusel2017gans}, 
computed using $1000$ generated samples 
and $1000$ real samples.

Figure~\ref{fig:gpe-vae-fid} reports FID (lower is better) versus training iterations on MNIST, CIFAR10, 
and CelebA. Both methods are trained under an equal wall--clock budget of 24 hours on a single NVIDIA V100 GPU. 
Each plot also includes the FID attained by the VAE without the flow model, which represents the baseline score 
before training $R$. For GPE, we pretrain the encoder and then train the decoder, with the combined time totaling 
24 hours. For VAE, the encoder and decoder are trained jointly for 24 hours. 
Across all datasets, the GPE variant achieves lower FID at matched iteration counts and attains strong quality earlier, 
indicating faster convergence under identical compute and model settings. 
This improvement aligns with the lower reconstruction error achieved by the GPE encoder, 
consistent with Figure~\ref{fig:comparison-gpe-vae}. 
\emph{Qualitative results:} Figure~\ref{fig:generation-results} visualizes the GPE-based latent generative model on four datasets, 
showing high-quality interpolations across all cases.

\begin{figure}[t]
\centering
\begin{subfigure}[t]{0.28\linewidth}
    \centering
    \includegraphics[width=\linewidth]{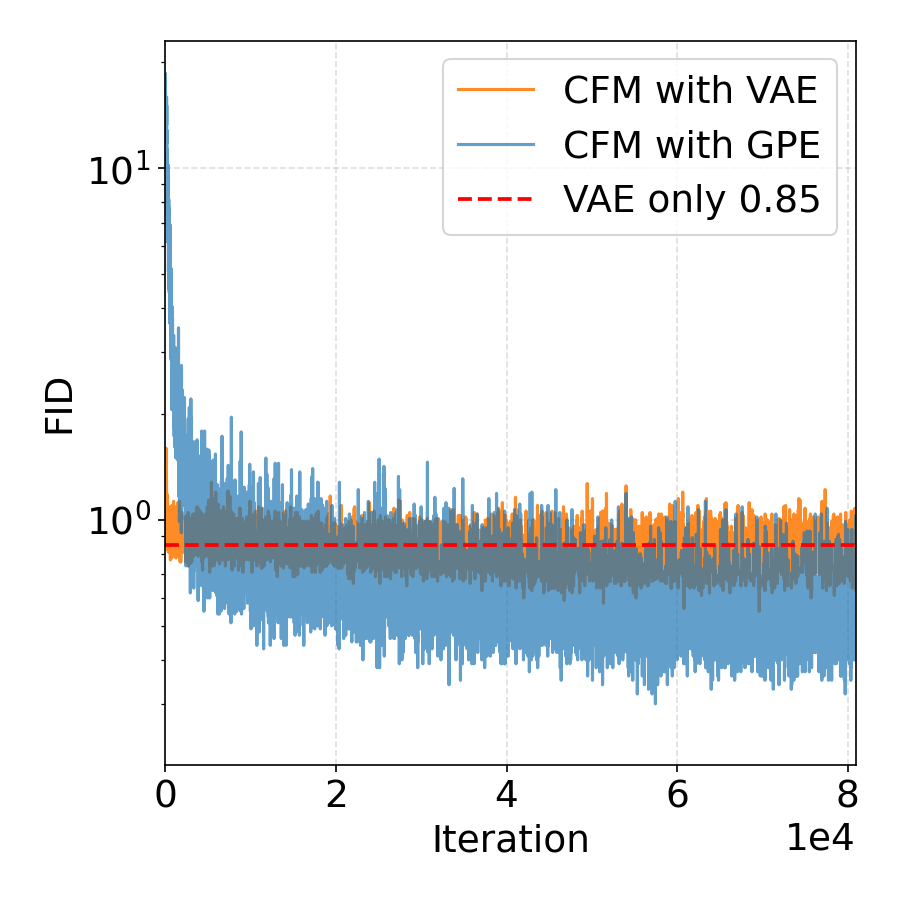}
    \caption{MNIST}
    \label{fig:gpe-vae-mnist}
\end{subfigure}\hfill
\begin{subfigure}[t]{0.28\linewidth}
    \centering
    \includegraphics[width=\linewidth]{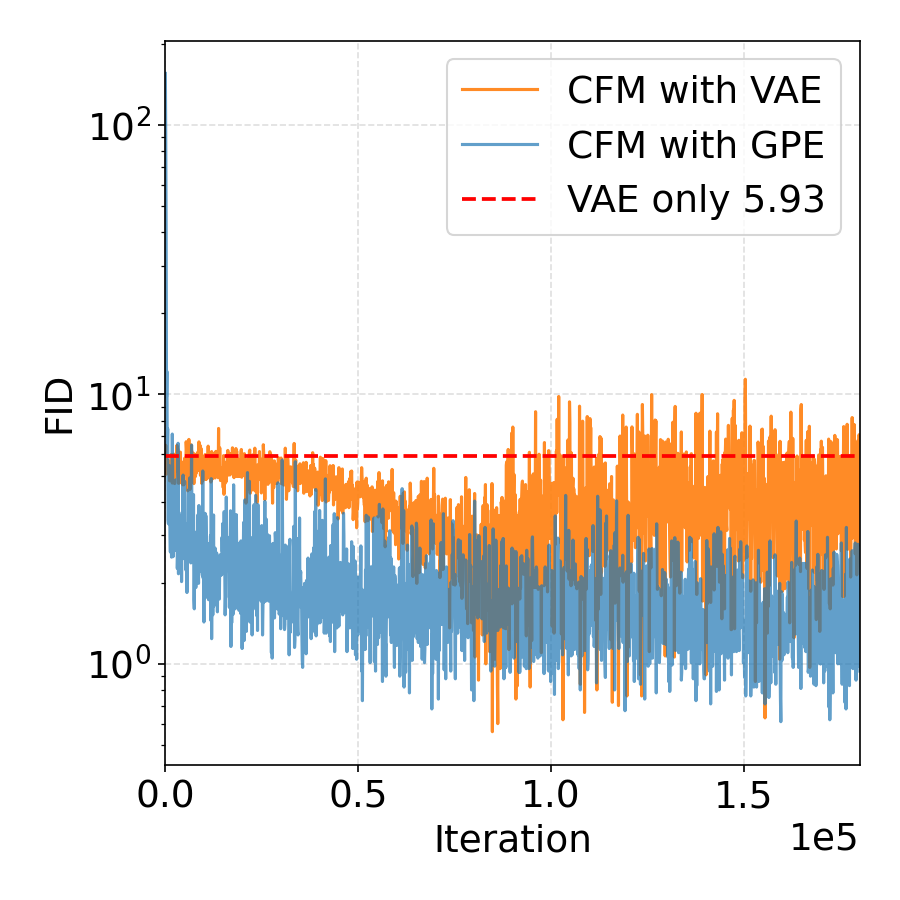}
    \caption{CIFAR-10}
    \label{fig:gpe-vae-cifar10}
\end{subfigure}\hfill
\begin{subfigure}[t]{0.28\linewidth}
    \centering
    \includegraphics[width=\linewidth]{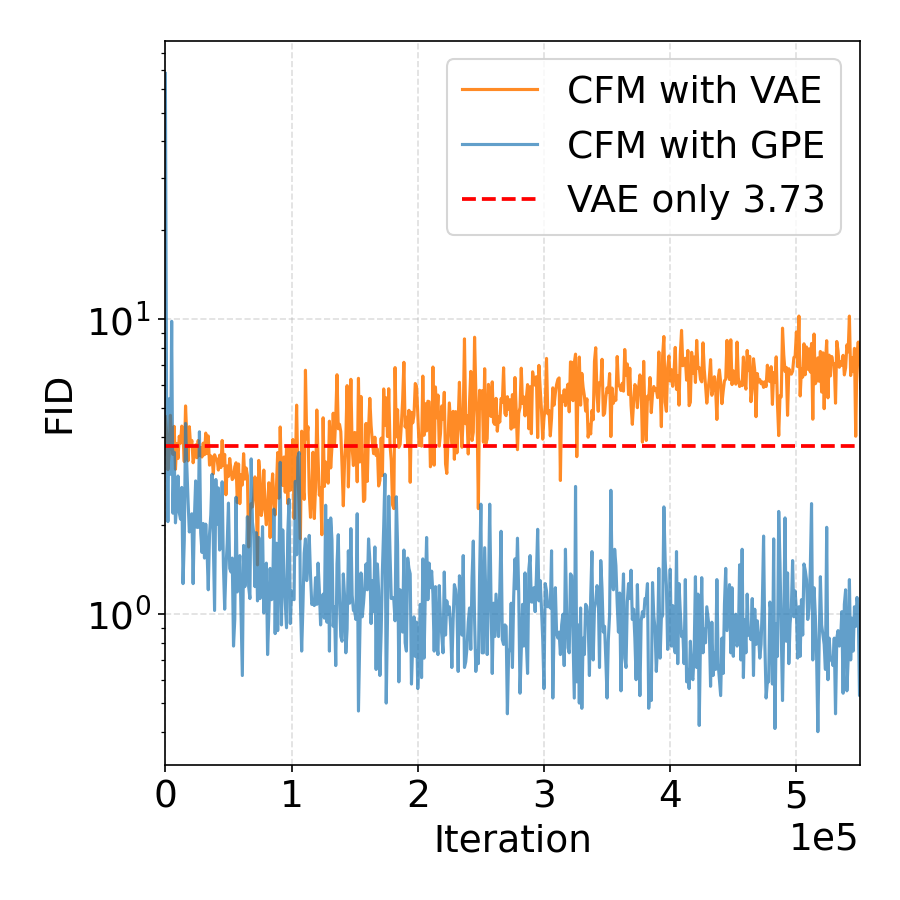}
    \caption{CelebA}
    \label{fig:gpe-vae-celeba}
\end{subfigure}
\caption{ Convergence comparison of conditional flow matching using GPE-based and VAE-based encoders.
The vertical axis shows FID (lower is better), and the horizontal axis shows training iterations.
Both VAE and GPE are trained for 24 hours of wall-clock time on the same hardware.
The red dotted line in each plot indicates the FID of the VAE without the flow-based model.
Across MNIST, CIFAR10, and CelebA, the GPE curves achieve lower FID than the VAE baselines at comparable iteration counts.
}
\label{fig:gpe-vae-fid}
\end{figure}

\begin{figure}[h!]
    \centering
    \begin{minipage}[b]{0.3\textwidth}
        \centering
            \includegraphics[width=\textwidth]{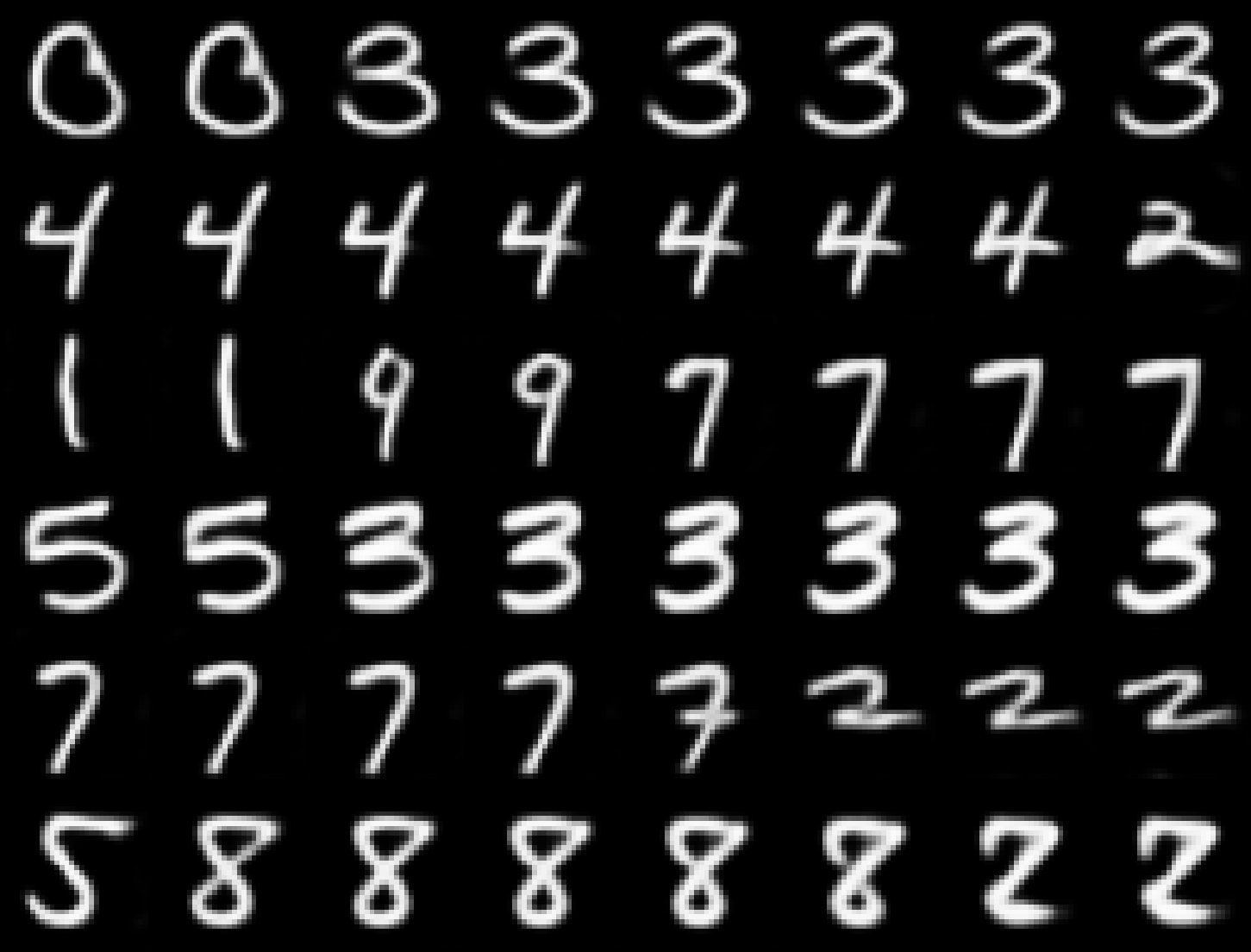}
            \includegraphics[width=\textwidth]{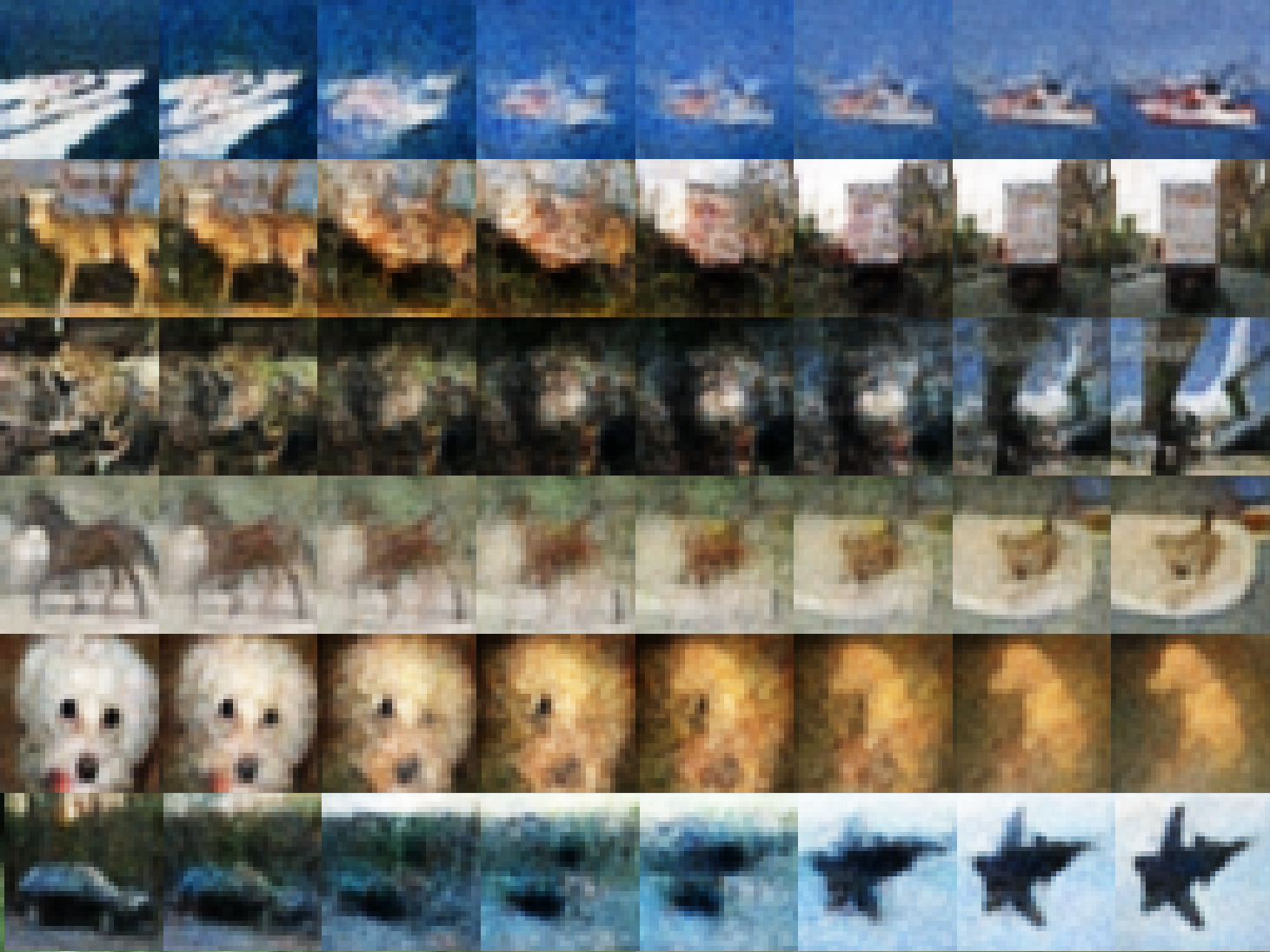}
    \end{minipage}
    \begin{minipage}[b]{0.6\textwidth}
        \centering
        \includegraphics[width=\textwidth]{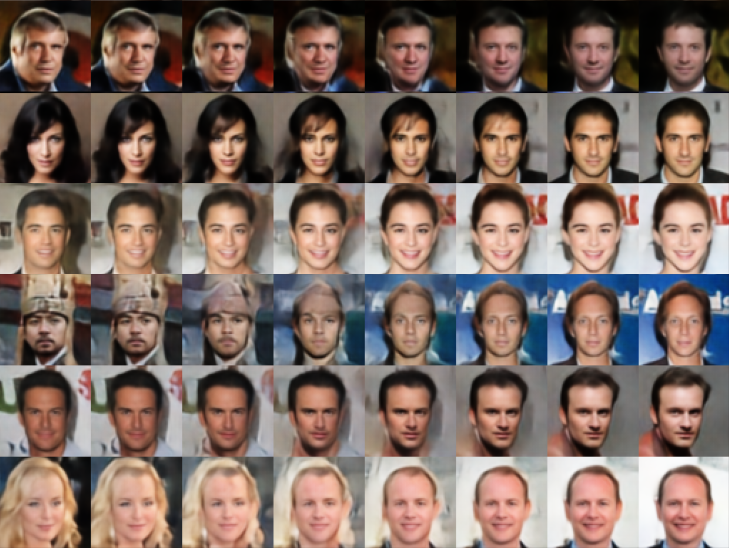}
    \end{minipage}
    
    \begin{minipage}[t]{0.91\textwidth}
        \centering
        \includegraphics[width=\textwidth]{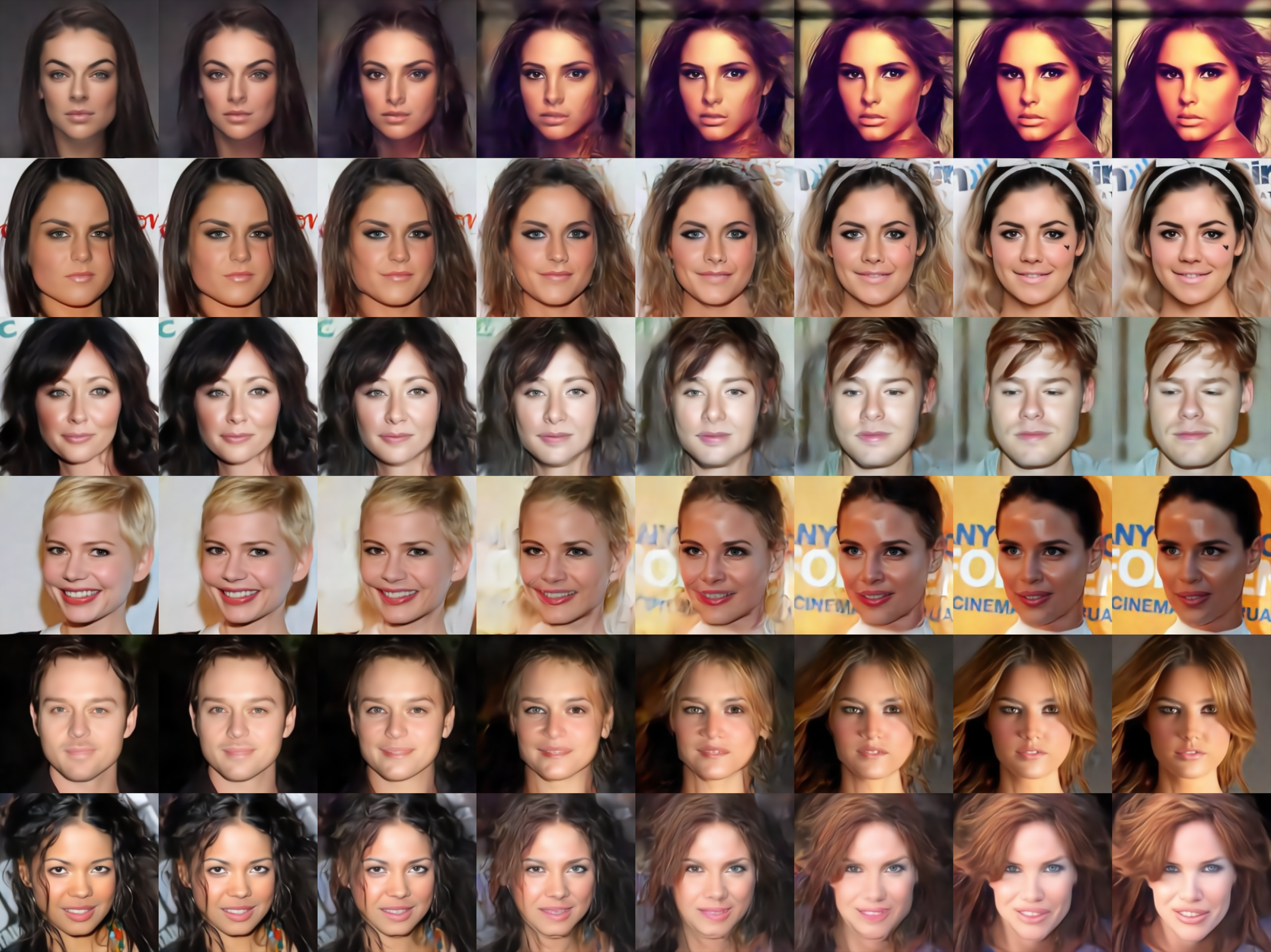}
    \end{minipage}
    \caption{
The figure shows the generation results from the GPE-based LGM. Using the diffusion model (specifically the conditional flow matching model) trained in the latent space, the interpolation between two images is generated for various datasets, including MNIST (top left, 1st row), CIFAR10 (top left, 2nd row), CelebA (top right), and CelebA-HQ ($256\times256$, bottom). The training efficiency for the encoder/decoder and diffusion model for all datasets are displayed in Table~\ref{tab:pretrained_encoders}, Figure~\ref{fig:comparison-gpe-vae}, and Figure~\ref{fig:gpe-vae-fid}.
}

    \label{fig:generation-results}
\end{figure}

\section{Conclusion}\label{sec:conclusion}
We introduce a geometry-preserving encoder/decoder framework for latent generative models. 
Theoretical analysis and numerical experiments demonstrate that this geometric constraint 
improves both training speed and reconstruction accuracy.

While these results focus on standard encoder/decoder networks, the framework could 
potentially adapt to larger transformer architectures. Key challenges include defining 
meaningful distance metrics on discrete data manifolds and adapting the framework to 
encoder-only or decoder-only models. A natural future direction is extending this 
geometric approach to multi-modal tasks that integrate text and images.

\section*{Acknowledgments and Disclosure of Funding}

WL acknowledges funding from the National Institute of Standards and Technology (NIST) under award number 70NANB22H021 and the startup fund from the Ohio State University. RCWO was supported by NSF GRFP-2237827 and NSF-DMS:1944925. DZ acknowledge funding from the Kunshan
Municipal Government research funding. JC was supported by NSF-DMS:1944925, NSF-CCF:2212318, the Alfred P. Sloan Foundation, and the Albert and Dorothy Marden Professorship, and GL acknowledges funding from NSF-DMS awards 2124913 and 2427955. We are grateful to the Minnesota Super-computing Institute for allowing use of their computational resources.

\appendix
\onecolumn


\section{Proof of \Cref{thm:weak-bilip}}

To show the proof of \Cref{thm:weak-bilip}, we first recall the Johnson-Lindenstrauss (JL) lemma.
\begin{lemma}[JL Lemma \cite{johnson1984extensions}]
Let $0 < \epsilon < 1$ and let $X$ be a set of $N$ points in $\mathbb{R}^d$. If $k$ is a positive integer such that
$k \geq \frac{8 \ln N}{\epsilon^2}$
then there exists a linear map $f: \mathbb{R}^d \to \mathbb{R}^k$ such that for all $u, v \in X$
\[(1 - \epsilon) \|u - v\|_2^2 \leq \|f(u) - f(v)\|_2^2 \leq (1 + \epsilon) \|u - v\|_2^2.\]
\end{lemma}
Next, we present the proof of \Cref{thm:weak-bilip}
\begin{proof}[Proof of \Cref{thm:weak-bilip}]
Let $\epsilon > 0$ be a covering radius. Because $\mathcal{M}$ is compact and has intrinsic dimension $m$, we can construct a finite $\epsilon$-net $\mathcal{N} = \{p_1, \dots, p_N\} \subset \mathcal{M}$ such that for every $x \in \mathcal{M}$, there exists some $p_i \in \mathcal{N}$ with $\|x - p_i\| \le \epsilon$. The size of this net is bounded by $N \le C_{\mathrm{cov}} (1/\epsilon)^m$.

By the Johnson-Lindenstrauss (JL) lemma, for any tolerance $\delta \in (0,1)$, there exists a linear map $P: \mathbb{R}^D \to \mathbb{R}^d$ that preserves the pairwise distances of the finite set $\mathcal{N}$
\begin{equation}\label{eq:jl-net}
(1-\delta)\|p_i - p_j\| \ \le\ \|P p_i - P p_j\| \ \le\ (1+\delta)\|p_i - p_j\| \qquad \forall p_i, p_j \in \mathcal{N},
\end{equation}
provided $d \ge 8 \delta^{-2} \log N$. We define our global encoder simply as the linear map $T(x) := P(x)$, which guarantees that $T$ is smooth. 

To guarantee the weak bi-Lipschitz bounds, we constrain our JL tolerance $\delta$ and our net radius $\epsilon$. First, we enforce the multiplicative bounds by setting $1+\delta \le \alpha$ and $1-\delta \ge \alpha^{-1}$, which is satisfied by choosing $\delta = 1 - \alpha^{-1}$. 

Next, we evaluate the distances for any arbitrary pair $x, y \in \mathcal{M}$. Let $p_x, p_y \in \mathcal{N}$ be their nearest neighbors in the net, such that $\|x - p_x\| \le \epsilon$ and $\|y - p_y\| \le \epsilon$. Using the triangle inequality, the distance between the net points is bounded by
\[
\|x-y\| - 2\epsilon \ \le\ \|p_x - p_y\| \ \le\ \|x-y\| + 2\epsilon.
\]

Because $P$ is a linear map represented by a $d \times D$ matrix, its effect on continuous spaces is governed by its operator norm $L_P := \sup_{v \neq 0} \frac{\|Pv\|}{\|v\|}$. For a  JL construction, the entries of $P$ are drawn independently from a distribution with zero mean and variance $1/d$. The operator norm is upper-bounded by the Frobenius norm $\|P\|_F = \sqrt{\sum_{i,j} P_{ij}^2}$. Since the expected squared Frobenius norm is simply $\mathbb{E}[\|P\|_F^2] = d \cdot D \cdot \frac{1}{d} = D$, we have $L_P \le \|P\|_F \le \mathcal{O}(\sqrt{D})$ with high probability. This bound guarantees that for any local variation, $\|P(x - p_x)\| \le L_P \|x - p_x\| \le L_P \epsilon$.

Applying the upper bound of \eqref{eq:jl-net} and the triangle inequality, we have
\begin{align*}
\|T(x) - T(y)\| &\le \|P(p_x) - P(p_y)\| + \|P(x - p_x)\| + \|P(y - p_y)\| \\
&\le \alpha \|p_x - p_y\| + 2 L_P \epsilon \\
&\le \alpha \big(\|x-y\| + 2\epsilon\big) + 2 L_P \epsilon = \alpha\|x-y\| + 2\epsilon(\alpha + L_P).
\end{align*}
Similarly, for the lower bound
\begin{align*}
\|T(x) - T(y)\| &\ge \|P(p_x) - P(p_y)\| - \|P(x - p_x)\| - \|P(y - p_y)\| \\
&\ge \alpha^{-1} \|p_x - p_y\| - 2 L_P \epsilon \\
&\ge \alpha^{-1} \big(\|x-y\| - 2\epsilon\big) - 2 L_P \epsilon = \alpha^{-1}\|x-y\| - 2\epsilon(\alpha^{-1} + L_P).
\end{align*}

To satisfy the additive error constraints of the weak $\alpha$-bi-Lipschitz property (e.g., bounding the additive terms by $\alpha - 1$ and $1 - \alpha^{-1}$ respectively), we choose the net radius to be sufficiently small
\[
\epsilon = \frac{\alpha - 1}{2(\alpha + L_P)}.
\]
With this choice of $\epsilon$ and $\delta$, the weak $\alpha$-bi-Lipschitz bounds are globally satisfied for all $x,y \in \mathcal{M}$. Substituting $\epsilon$ and $\delta$ into the JL dimension requirement yields
\[
d \ge \frac{8}{(1-\alpha^{-1})^2} \log \Big( C_{\mathrm{cov}} \Big(\frac{2(\alpha + L_P)}{\alpha-1}\Big)^m \Big),
\]
which simplifies to $d = \mathcal{O}\big(\alpha^2 m \log(\frac{\alpha + L_P}{\alpha-1})\big)$, completing the proof.
\end{proof}

\section{Change of variables in Riemannian manifolds}
In this section, we briefly review the change of variables formula for Riemannian manifolds. To properly evaluate integrals over an embedded manifold, we must account for the local geometric distortion introduced by the embedding map. For a foundational treatment of this topic, see Section 16 of \cite{john2012introduction}. The following lemma formally establishes this volume transformation for a smooth embedding $T:\M \to \R^d$.
\begin{proposition}\label{lem:change-of-var}
Let $T: \mathcal{M} \to \mathbb{R}^d$ be a smooth embedding of an $m$-dimensional Riemannian manifold $(\mathcal{M}, g)$. Let $J_T(p) = \sqrt{\det(dT_p^\top dT_p)}$ denote the Jacobian of $T$, where $dT_p: \mathcal{T}_p\mathcal{M} \to \mathbb{R}^d$ is the differential of $T$ evaluated at $p \in \mathcal{M}$. Then, for any integrable function $f: T(\mathcal{M}) \to \mathbb{R}$, the integral change of variables formula is given by
\begin{equation*}
\int_{\mathcal{M}} f(T(p)) \, d\mathrm{vol}_{\mathcal{M}}(p) = \int_{T(\mathcal{M})}  \frac{f(q)}{J_T(T^{-1}(q))} \, d\mathrm{vol}_{T(\mathcal{M})}(q).
\end{equation*}
\end{proposition}

\begin{proof}
We first establish the identity locally. Fix a point $p \in \mathcal{M}$ and choose an orthonormal basis $\{v_1, \dots, v_m\}$ for the tangent space $\mathcal{T}_p\mathcal{M}$ equipped with the intrinsic metric $g_p$. We use the exponential map $\exp_p$ to define a local parameterization $P(y) = \exp_p\left( \sum_{i=1}^m y^i v_i \right)$ that maps a flat coordinate grid in $\mathbb{R}^m$ onto a local neighborhood $\mathcal{U} \subset \mathcal{M}$ near $p$.

For any function $h$ supported in this local neighborhood, its integral with respect to the Riemannian volume measure $d\mathrm{vol}_{\mathcal{M}}$ is computed by pulling the domain back to the flat coordinate space $\mathbb{R}^m$
\begin{equation}\label{eq:local_M}
\int_{\mathcal{M}} h(p) \, d\mathrm{vol}_{\mathcal{M}}(p) = \int_{\mathbb{R}^m} h(P(y)) \sqrt{\det(g_{ij}(y))} \, dy^1 \dots dy^m.
\end{equation}
Here, the term $\sqrt{\det(g_{ij}(y))}$ acts as a correction factor. If we take a perfectly square coordinate box in $\mathbb{R}^m$, the parameterization $P$ maps it to a curved, skewed patch on $\mathcal{M}$. The edges of this mapped patch are spanned by the tangent vectors $\frac{\partial P}{\partial y^i}$. The $(i,j)$-th entry of the metric matrix $g_{ij}(y)$ is defined as the inner product of these tangent vectors, meaning $g_{ij}(y)$ is the Gram matrix of these local basis vectors. The square root of its determinant measures the physical volume of the $m$-dimensional parallelepiped they form.

Now, we consider the same neighborhood embedded in the target space, parameterized by the composition $(T\circ P)(y)$. To integrate over the embedded manifold $T(\mathcal{M})$, we pull its volume measure back to the same flat coordinate grid
\begin{equation}\label{eq:prr1}
    \int_{T(\mathcal{M})} h(q) \, d\mathrm{vol}_{T(\mathcal{M})}(q) = \int_{\mathbb{R}^m} h(T(P(y))) \sqrt{\det(\tilde{g}_{ij}(y))} \, dy^1\dots dy^m.
\end{equation}
Here, $\tilde{g}_{ij}(y)$ is the new metric tensor induced in the target space. By the chain rule, the tangent vectors spanning the coordinate box in the target space are $\frac{\partial(T \circ P)}{\partial y^i} = dT_{P(y)}\left(\frac{\partial P}{\partial y^i}\right)$. The entries of the induced metric tensor $\tilde{g}_{ij}(y)$ are the Euclidean inner products of these mapped vectors in $\mathbb{R}^d$
\begin{equation*}
\tilde{g}_{ij}(y) = \left\langle dT_{P(y)} \left( \frac{\partial P}{\partial y^i} \right), dT_{P(y)} \left( \frac{\partial P}{\partial y^j} \right) \right\rangle_{\mathbb{R}^d}.
\end{equation*}
Geometrically, the linear map $dT_{P(y)}$ scales the volume of any tangent parallelepiped by exactly its Jacobian, $J_T(P(y)) = \sqrt{\det\left(dT^\top_{P(y)} dT_{P(y)}\right)}$. Therefore, the volume of the mapped parallelepiped relates to the volume of the base parallelepiped by
\begin{equation*}
    \sqrt{\det(\tilde{g}_{ij}(y))} = J_T(P(y))\sqrt{\det(g_{ij}(y))}.
\end{equation*}
Substituting this relationship back into the coordinate integral \eqref{eq:prr1} yields
\begin{equation}\label{eq:prr3}
    \int_{T(\mathcal{M})} h(q) \, d\mathrm{vol}_{T(\mathcal{M})}(q) 
    = \int_{\mathbb{R}^m} h(T(P(y))) J_T(P(y)) \sqrt{\det(g_{ij}(y))} \, dy^1\dots dy^m.
\end{equation}
Comparing \eqref{eq:prr3} with \eqref{eq:local_M}, the right-hand side is the coordinate integration of the function $(h \circ T) \cdot J_T$ over the base manifold $\mathcal{M}$. Pushing the integral back from local coordinates to $\mathcal{M}$ gives the local relation
\begin{equation}\label{eq:local_identity}
    \int_{T(\mathcal{M})} h(q) \, d\mathrm{vol}_{T(\mathcal{M})}(q) = \int_{\mathcal{M}} h(T(p)) J_T(p) \, d\mathrm{vol}_{\mathcal{M}}(p).
\end{equation}

To extend this to a global statement for any integrable function $f: T(\mathcal{M}) \to \mathbb{R}$, let $\{\mathcal{U}_\alpha\}$ be an open cover of $\mathcal{M}$ by local neighborhoods where the parameterization $P$ is valid, and let $\{\psi_\alpha\}$ be a smooth partition of unity subordinate to $\{\mathcal{U}_\alpha\}$ such that $\sum_\alpha \psi_\alpha(p) = 1$ for all $p \in \mathcal{M}$. 
Since $T$ is a smooth embedding, it is a diffeomorphism onto its image $T(\mathcal{M})$. The collection $\{T(\mathcal{U}_\alpha)\}$ forms an open cover of $T(\mathcal{M})$, and the compositions $\{\psi_\alpha \circ T^{-1}\}$ define a partition of unity subordinate to this cover. We decompose the global integral of $f$ by applying the local relation \eqref{eq:local_identity} to each localized component
\begin{align*}
\int_{T(\mathcal{M})} \frac{f(q)}{J_T(T^{-1}(q))} \, d\mathrm{vol}_{T(\mathcal{M})}(q) &= \sum_\alpha \int_{T(\mathcal{M})} (\psi_\alpha \circ T^{-1})(q) \frac{f(q)}{J_T(T^{-1}(q))} \, d\mathrm{vol}_{T(\mathcal{M})}(q) \\
&= \sum_\alpha \int_{\mathcal{M}} \psi_\alpha(p) \frac{f(T(p))}{J_T(p)} J_T(p) \, d\mathrm{vol}_{\mathcal{M}}(p) \\
&= \int_{\mathcal{M}} \left( \sum_\alpha \psi_\alpha(p) \right) f(T(p)) \, d\mathrm{vol}_{\mathcal{M}}(p) \\
&= \int_{\mathcal{M}} f(T(p)) \, d\mathrm{vol}_{\mathcal{M}}(p).
\end{align*}
This completes the proof.
\end{proof}

    \section{Proof of \Cref{prop:compact-minimizer}}

\begin{proof}
    We first establish that $\mathcal{T}_\beta^k$ is compact by applying the Arzelà-Ascoli theorem, demonstrating that the families of functions $\{T\}_{T \in \mathcal{T}_\beta^k}$ and their differentials $\{dT\}_{T \in \mathcal{T}_\beta^k}$ are uniformly bounded and equicontinuous.

    Let $d_{\mathcal{M}}$ denote the intrinsic geodesic distance on the compact Riemannian manifold $\mathcal{M}$. For any $T \in \mathcal{T}_\beta^k$ and any $x \in \mathcal{M}$, applying the Mean Value Theorem on manifolds alongside the bound $\|dT\|_\infty \leq \beta$ and the fixed point constraint $T(x_0) = 0$, we have:
    \begin{equation*}
        \|T(x)\|_{\mathbb{R}^d} = \|T(x) - T(x_0)\|_{\mathbb{R}^d} \leq \sup_{z \in \mathcal{M}} \|dT(z)\|_{\mathrm{op}} \, d_{\mathcal{M}}(x, x_0) \leq \beta \operatorname{diam}_{\mathcal{M}}(\mathcal{M}).
    \end{equation*}
    Because $\mathcal{M}$ is compact, its geodesic diameter is finite, which guarantees that the family $\{T\}$ is uniformly bounded. Additionally, the uniform bound on the first derivative ($\|dT\|_\infty \leq \beta$) ensures that $\{T\}$ is uniformly Lipschitz, and therefore equicontinuous. 
    
    Next, we evaluate the family of differentials $\{dT\}$. By definition, this family is uniformly bounded since $\|dT\|_\infty \leq \beta$. The bound on the second derivative, $\|D^2 T\|_\infty \leq \beta$, strictly limits the rate of change of the first derivative. This guarantees that $\{dT\}$ is also uniformly Lipschitz and equicontinuous. Since both $\{T\}$ and $\{dT\}$ are uniformly bounded and equicontinuous, the Arzelà-Ascoli theorem guarantees that any sequence in $\mathcal{T}_\beta^k$ contains a subsequence converging uniformly in the $C^1$ topology to a limit function $T^* \in C^1(\mathcal{M}, \mathbb{R}^d)$.

    To verify that $\mathcal{T}_\beta^k$ is closed, note that the pointwise constraint $T^*(x_0) = 0$ is naturally preserved under uniform convergence. Because $T^*$ is the $C^1$-uniform limit of a sequence of functions with $\beta$-Lipschitz first derivatives, its differential $dT^*$ is also $\beta$-Lipschitz. Thus, $T^* \in C^{1,1}(\mathcal{M}, \mathbb{R}^d)$ and satisfies $\max(\|dT^*\|_\infty, \|D^2 T^*\|_\infty) \leq \beta$ almost everywhere. Therefore, $T^* \in \mathcal{T}_\beta^k$, proving that the set is compact.

    Next, we bound the GME cost over this compact set. From the definition of $ \GME $ in \eqref{eq:gme-cost}, we have
    \begin{align*}
        \GME(T)
        &= \mathbb{E}_{(x,x')\sim \mu^2} \left[ \log\left(\frac{1+\|T(x)-T(x')\|^2}{1+\|x-x'\|^2}\right)^2\right] \\
        &\leq \mathbb{E}_{(x,x')\sim \mu^2} \left[ \log\left( \max\left\{
        \frac{1 + \beta^2\|x-x'\|^2}{1+\|x-x'\|^2},
        1 + \|x-x'\|^2
        \right\} \right)^2 \right] \\
        &\leq \left(\log\left( \max\left\{
        \beta^2, 1 + \operatorname{diam}(\mathcal{M})^2
        \right\} \right)\right)^2
        \leq 4 (\log(\beta))^2,
    \end{align*}
    where we used the $ \beta $-Lipschitz condition, the boundedness of $ \mathcal{M} $, and the assumption in \cref{eq:assmpt-beta-diam}.

    Finally, to establish the existence of a minimizer, we show that $\GME$ is continuous. Define the integrand $F_T(x, x') := \log\left( \frac{1 + \|T(x) - T(x')\|^2}{1 + \|x - x'\|^2} \right)^2$. Let $ T_n \to T $ uniformly in $ \mathcal{T}_\beta^k $. Uniform convergence implies $F_{T_n}(x, x') \to F_T(x, x')$ uniformly on $\mathcal{M} \times \mathcal{M}$. Moreover, using the $\beta$-Lipschitz bound, we have a uniform bound $|F_{T_n}(x, x')| \leq 4 (\log(\beta))^2$, which is integrable. By the dominated convergence theorem, $\lim_{n \to \infty} \GME(T_n, \mu) = \GME(T, \mu)$, confirming that $ \GME $ is continuous with respect to $ T $ under the uniform topology.
    Because $ \GME $ is a continuous functional on the compact set $ \mathcal{T}_\beta^k $, the Extreme Value Theorem ensures that it attains its minimum, confirming the existence of a minimizer $T^* \in \mathcal{T}_\beta^k$.
\end{proof}

\section{Proof of \Cref{prop:hessian-chain}}

\begin{proof}
By applying the chain rule for the second derivative, the quadratic form of the parameter-space Hessian along any direction $v \in \mathbb{R}^p$ is given by
\begin{equation}\label{eq:second-var-chain}
    v^\top H_\theta v = \delta^2 F(T_\theta)(J_\theta v, J_\theta v) + \langle \nabla_{L^2(\mathcal{M})} F(T_\theta), \nabla_\theta^2 T_\theta(v,v) \rangle_{L^2(\mathcal{M})}.
\end{equation}
Near a local minimum, the gradient residual term vanishes. Because $J_\theta v \in L^2(\mathcal{M})$, we can substitute the induced variation function $h = J_\theta v$ into the assumed functional bounds, yielding
\begin{equation}
    \lambda_{\min} \|J_\theta v\|_{L^2(\mathcal{M})}^2 \leq v^\top H_\theta v \leq \lambda_{\max} \|J_\theta v\|_{L^2(\mathcal{M})}^2.
\end{equation}
Substituting the operator bounds of the Jacobian, $\sigma_{\min}^2 \|v\|^2 \leq \|J_\theta v\|_{L^2(\mathcal{M})}^2 \leq \sigma_{\max}^2 \|v\|^2$, directly provides the parameter-space curvature bounds $\lambda_{\min} \sigma_{\min}^2 \|v\|^2 \leq v^\top H_\theta v \leq \lambda_{\max} \sigma_{\max}^2 \|v\|^2$. 

From  convex optimization theory, an objective with an $L$-Lipschitz gradient and $\mu$-strong convexity (i.e., $\mu I \preceq H_\theta \preceq L I$) requires a step size $\eta < 2/L$ for stable gradient descent, yielding a convergence rate proportional to $1 - \frac{\mu}{L}$. Substituting $\mu = \lambda_{\min} \sigma_{\min}^2$ and $L = \lambda_{\max} \sigma_{\max}^2$ completes the proof.
\end{proof}

\section{The first and second variations of GME cost}

To derive the explicit formulas of the first and second variations of GME cost, we will assume general form on cost functions $c_X$ and $c_Y$. Throughout the proofs, we assume the cost functions $c_X$ and $c_Y$ take the form
\begin{align*}
    c_X(x,x') = \eta_X(x-x'),\quad c_Y(y,y') = \eta_Y(y-y')
\end{align*}
where $\eta_X:X\rightarrow \mathbb{R}$ and $\eta_Y:Y \rightarrow \mathbb{R}$ are differentiable functions. 
In the following, for simplicity, we denote by
    \begin{align*}
        \eta_X &= \eta_X(x - x')\\
        \eta_Y &= \eta_Y(T(x) - T(x'))\\
        \nabla \eta_Y &= \nabla \eta_Y(T(x) - T(x'))\\
        \nabla^2\eta_Y &= \nabla^2 \eta_Y(T(x) - T(x'))\\
        h_{x,x'} &= h(x) - h(x').
    \end{align*} 
Thus, the GME cost can be written as
\begin{align*}
    \GME(T) = \int_{\mathcal{M}^2} \Bigl(\eta_X - \eta_Y\Bigr)^2 d\mu d\mu.
\end{align*}
\begin{proposition}\label{prop:first-second-var}
    The first and second variations of the GME cost functional at $T:\mathcal{M}\rightarrow \mathbb{R}^d$ in the direction of $h:\mathcal{M}\rightarrow \mathbb{R}^d$ take the form
  \begin{align*}
    &\delta \GME(T)(h) =
    2 \int_{\mathcal{M}^2}\Big( \eta_Y- \eta_X \Big) \langle \nabla \eta_Y, h_{x,x'}\rangle d\mu(x)d\mu(x')
  \end{align*}
  and
  \begin{align*}
    \delta^2 \GME(T)(h,h) 
    =2 \int_{\mathcal{M}^2}  \langle \nabla \eta_Y,h_{x,x'}\rangle^2 +  (\eta_Y -\eta_X )\langle \nabla^2 \eta_Y h_{x,x'}, h_{x,x'}\rangle d\mu(x) d\mu(x').
  \end{align*}
  respectively. 
\end{proposition}
\begin{proof}
    Let $h:\mathcal{M} \rightarrow \mathbb{R}^d$ be a function, and let $t\in \mathbb{R}$. Since $\eta_Y$ is assumed to be differentiable, it follows that
    \begin{equation}\label{eq:prop-proof-1}
        \eta_Y((T+th)_{x,x'}) = \eta_Y(T_{x,x'} + t h_{x,x'}) = \eta_Y + t \langle \nabla \eta_Y, h_{x,x'} \rangle + O(t^2).
    \end{equation}
    Using \eqref{eq:prop-proof-1}, we have
    \begin{align*}
        &\GME(T+th;\mu)\\
        =&\int_{\mathcal{M}^2}  \Big( \eta_X- \eta_Y((T+th)_{x,x'}) \Big)^2 d\mu(x)d\mu(x')\\
        =&\int_{\mathcal{M}^2}  \Big( \eta_X- \eta_Y
        - t \langle \nabla \eta_Y, h_{x,x'} \rangle + O(t^2)\Big)^2 d\mu(x)d\mu(x')\\
        =&\int_{\mathcal{M}^2}  \Big( \eta_X- \eta_Y\Big)^2 
        - t  \bigg(2 \Big( \eta_X- \eta_Y\Big) \langle \nabla \eta_Y, h_{x,x'} \rangle  
        \bigg)d\mu(x)d\mu(x')
        + O(t^2).
    \end{align*}
    From the definition of the first variation in~\eqref{eq:def-first}, we get the formulation for $\delta \GME(T)(h)$.

    Next, we compute the second variation.
    For any $t\in\mathbb{R}$, using \eqref{eq:prop-proof-1} and the first variation $\delta \operatorfont{GM}(T)$,
    \begin{align*}
    &\delta \GME(T+th)(h) \nonumber \\
        =& 2\int_{\mathcal{M}^2} \Big( \eta_Y + t\langle \nabla \eta_Y , h_{x,x'} \rangle +O(t^2) - \eta_X  \Big) \Big\langle \nabla \eta_Y + t \nabla^2 \eta_Y h_{x,x'}+O(t^2),h_{x,x'} \Big\rangle d\mu(x) d\mu(x')  \\
        =&2 \int_{\mathcal{M}^2}   (\eta_Y - \eta_X  )\langle \nabla \eta_Y, h_{x,x'}\rangle + t  \langle \nabla \eta_Y,h_{x,x'}\rangle^2 \\
         &\hspace{4cm}+ t (\eta_Y - \eta_X ) \langle \nabla^2 \eta_Y h_{x,x'}, h_{x,x'}\rangle d\mu(x) d\mu(x') 
         + O(t^2).
    \end{align*}
    From the definition of the second variation in~\eqref{eq:def-second}, we get the formulation for $\delta^2 \GME(T)(h,h)$. 
    This concludes the proof.
    
\end{proof}

Using Proposition~\ref{prop:first-second-var}, by plugging in $\eta_X(x-x') = \log(1+ \|x-x'\|^2)$ and $\eta_Y(y-y') = \log(1+\|y-y'\|^2)$, we get the first and second variations of the GME cost.
\begin{proposition}\label{prop:second-var-gme}
    Given a distribution $\mu \in \mathbb{P}(\mathcal{M})$ and maps $T, h \in L^4(\mathcal{M},\mathbb{R}^d)$, the first and the second variations of the GME cost at $T$ in the direction of $h$ take the following forms:
    \begin{align*}
        \delta \GME(T)(h)
        &= \int_{\mathcal{M}^2} 4 \log\left(\frac{\|T_{x,x'}\|^2+1}{\|x-x'\|^2+1}\right) \frac{\langle T_{x,x'}, h_{x,x'} \rangle}{\|T_{x,x'}\|^2+1} \, d\mu \, d\mu, \\
        \delta^2 \GME(T)(h,h)
        &= \int_{\mathcal{M}^2} 4 \log\left(\frac{\|T_{x,x'}\|^2 + 1}{\|x-x'\|^2 + 1}\right)\left(\frac{\|h_{x,x'}\|^2}{\|T_{x,x'}\|^2 + 1} - 2 \left(\frac{\langle T_{x,x'}, h_{x,x'} \rangle}{\|T_{x,x'}\|^2 + 1}\right)^2 \right) \\
        &\hspace{7cm} 
        + 8 \left(\frac{\langle T_{x,x'}, h_{x,x'} \rangle}{\|T_{x,x'}\|^2 + 1}\right)^2 \, d\mu \, d\mu.
    \end{align*}
    Here, for any function $f: \mathcal{M} \rightarrow \mathbb{R}^d$, we denote $f_{x,x'} = f(x) - f(x')$ for simplicity.
\end{proposition}

\section{Proof of \Cref{cor:gme-cost-gd-conv}}\label{appendix:proof-gme-gd}
\begin{proof}
    Let $F(T) = \GME(T)$. From \Cref{thm:log-GME-bound}, the upper bound of the Hessian of $F$ is bounded by $L_k = 8(4\log(\beta_k) + 1)$. Using the descent lemma, we bound the objective at step $k+1$
    \[
        F(T^{(k+1)}) \leq F(T^{(k)}) - \sigma_k \|\nabla_{L^2} F(T^{(k)})\|_{L^2(\mathcal{M})}^2 + \frac{\sigma_k^2 L_k}{2} \|\nabla_{L^2} F(T^{(k)})\|_{L^2(\mathcal{M})}^2.
    \]
    Substituting $\sigma_k = 1/L_k = \frac{1}{8(4\log(\beta_k) + 1)}$ perfectly balances the quadratic penalty, ensuring strict descent
    \[
        F(T^{(k+1)}) \leq F(T^{(k)}) - \frac{1}{2L_k} \|\nabla_{L^2} F(T^{(k)})\|_{L^2(\mathcal{M})}^2.
    \]
    Rearranging this expression gives
    \[
        \frac{1}{2L_k} \|\nabla_{L^2} F(T^{(k)})\|_{L^2(\mathcal{M})}^2 \leq F(T^{(k)}) - F(T^{(k+1)}).
    \]
    Summing this inequality over all steps from $k=0$ to $K$ yields a telescoping sum
    \[
        \sum_{k=0}^{K} \frac{1}{2L_k} \|\nabla_{L^2} F(T^{(k)})\|_{L^2(\mathcal{M})}^2 \leq F(T^{(0)}) - F(T^{(K+1)}).
    \]
    Since $F(T^{(K+1)})$ is bounded below by the global infimum $\inf_T F(T)$, we replace $F(T^{(K+1)})$ with this lower bound. Taking the limit as $K \to \infty$ yields the final result.
\end{proof}

\bibliographystyle{plainnat}
\bibliography{gf}

\end{document}